\pdfoutput=1

\documentclass{article}

\usepackage{lipsum}
\usepackage{authblk}
\usepackage[top=2cm, bottom=2cm, left=2cm, right=2cm]{geometry}
\usepackage{fancyhdr}
\usepackage{amsmath}
\usepackage{amsthm}
\usepackage[utf8]{inputenc}
\usepackage[english]{babel}
\usepackage{graphicx}
\usepackage{subcaption}
\usepackage{multicol}
\usepackage{mathtools}
\usepackage{bm}
\usepackage{bbm}
\usepackage[bbgreekl]{mathbbol}
\usepackage{mathrsfs}
\usepackage{amssymb}
\usepackage{tabularx}
\usepackage{textcomp}
\usepackage{xcolor}
\usepackage{stmaryrd}
\usepackage{graphicx}
\usepackage{tikz}
\usepackage[section]{placeins}
\usepackage{array}
\usepackage{booktabs}
\usepackage{wrapfig}
\usepackage[title]{appendix}

\newcommand{\efrac}[2]{%defines phinary division bar
  \mathchoice
    {\ooalign{%
      $\genfrac{}{}{2pt}0{\hphantom{#1}}{\hphantom{#2}}$\cr%
      $\color{white}\genfrac{}{}{1pt}0{\color{black}#1}{\color{black}#2}$}}%
    {\ooalign{%
      $\genfrac{}{}{1.2pt}1{\hphantom{#1}}{\hphantom{#2}}$\cr%
      $\color{white}\genfrac{}{}{.4pt}1{\color{black}#1}{\color{black}#2}$}}%
    {\ooalign{%
      $\genfrac{}{}{1.2pt}2{\hphantom{#1}}{\hphantom{#2}}$\cr%
      $\color{white}\genfrac{}{}{.4pt}2{\color{black}#1}{\color{black}#2}$}}%
    {\ooalign{%
      $\genfrac{}{}{1.2pt}3{\hphantom{#1}}{\hphantom{#2}}$\cr%
      $\color{white}\genfrac{}{}{.4pt}3{\color{black}#1}{\color{black}#2}$}}%
}

\DeclarePairedDelimiter\floor{\lfloor}{\rfloor}

\newtheorem{theorem}{Theorem}[section]
\newtheorem{proposition}{Proposition}[section]
\newtheorem{corollary}{Corollary}[theorem]
\newtheorem{lemma}[theorem]{Lemma}
\theoremstyle{definition}
\newtheorem{definition}{Definition}[section]
\theoremstyle{remark}
\newtheorem*{remark}{Remark}
\theoremstyle{definition}

\theoremstyle{remark}
\newtheorem*{example}{Example}

\title {On ordinal dynamics and the multiplicity of transfinite cardinality}
\author{Scott V. Tezlaf\footnote{Department of Mathematics, StGIS, St. Gilgen, Austria}
}
\date{1 June 2018\footnote{Phi Day (1.6.18)}}

\begin{document}

\maketitle

\abstract{This paper explores properties and applications of an ordered subset of the quadratic integer ring $\mathbbm{Z}[\frac{1+\sqrt{5}}{2}]$. The numbers are shown to exhibit a parity triplet, as opposed to the familiar even/odd doublet of the regular integers. Operations on these numbers are defined and used to generate a succinct recurrence relation for the well-studied Fibonacci diatomic sequence, providing the means for generating analogues to the famed Calkin-Wilf and Stern-Brocot trees. Two related fractal geometries are presented and explored, one of which exhibits several identities between the Fibonacci numbers and golden ratio, providing a unique geometric expression of the Fibonacci words and serving as a powerful tool for quantifying the cardinality of ordinal sets. The properties of the presented set of numbers illuminate the symmetries behind ordinals in general, as well as provide perspective on the natural numbers and raise questions about the dynamics of transfinite values. In particular, the first transfinite ordinal $\omega$ is shown to be logically consistent with a value whose cardinality is dual: both zero and one. Considerations of these points and opportunities for further study are discussed.}

\section*{Introduction}
In this paper, we explore a subset of $\mathbbm{Z}[\frac{1+\sqrt{5}}{2}]$, the quadratic integer subring of the quadratic field $\mathbbm{Q}[\sqrt{5}]$. The ring $\mathbbm{Z}[\frac{1+\sqrt{5}}{2}]$ is characterized by its unique factorization property and were used in Peter Gustav Lejeune Dirichlet's proof of Fermat's Last Theorem for n=5 \cite[pp. 68-71]{edwards1996fermat}. Dirichlet's integers take on linear combinations\footnote{The typical convention is to write the integer term first followed by the irrational term (i.e. $a+b\phi$). However, in this paper, the order will be swapped to aid in the readability of the values of interest.} of the golden ratio $\phi=\frac{1+\sqrt{5}}{2}$, such that 
\begin{align*}
\mathbbm{Z}[\phi]=\{a\phi+b \ | \ a,b \in \mathbbm{Z}\}.
\end{align*}
Here, however, we will be studying the remarkable properties of a unique subset of these values---the first few nonnegative elements being
\begin{align*}
\dots, \ 0, \ 1, \ \phi, \ \phi+1, \ \phi+2, \ 2\phi+1, \ 2\phi+2, \ 3\phi+1, \ 3\phi+2, \ 3\phi+3, \ 4\phi+2, \ 4\phi+3, \ 4\phi+4, \ 5\phi+3, \  \dots
\end{align*}
These numbers correspond with the ``integer" values of Bergman's base-phi number system (discussed in \S \ref{subsection: phinary numbers}). Here they will be referred to as the \textit{phinary integers}, and although only partially closed under ordinary addition and multiplication, they do form a ring under a native set of operations, including analogues to addition and multiplication (\S \ref{subsection: phinary operators}). By defining functions that combine both the ordinary and new operators, unique properties are realized. Through such methods, the existence of a parity triplet within the phinary integers is proved. The parity set behaves similarly to the even/odd doublet of the normal integers but with an aperiodic distribution defined by the Fibonacci word pattern (\S \ref{section: phinary parity}). For example, the following phinary numbers are indicated by their parity, with bold typeface corresponding to \textbf{\textit{even}} numbers, gray type for \textcolor{gray}{\textit{odd}}, and underlined type for the third parity, dubbed \underline{\textit{curious}}:
\begin{align*}
\dots, \ \bm{0}, \ \textcolor{gray}{1}, \ \underline{\phi}, \ \bm{\phi+1}, \ \textcolor{gray}{\phi+2}, \ \bm{2\phi+1}, \ \underline{2\phi+2}, \ \textcolor{gray}{3\phi+1}, \ \bm{3\phi+2}, \ \textcolor{gray}{3\phi+3}, \ \underline{4\phi+2}, \ \bm{4\phi+3}, \ \textcolor{gray}{4\phi+4}, \ \bm{5\phi+3}, \  \dots
\end{align*}
This classification is more than a novelty, as a direct application is demonstrated in defining a concise recurrence relation for the Fibonacci diatomic sequence (\S \ref{The Fibonacci diatomic Recurrence Relation}). Additionally, versions of the Stern-Brocot and Calkin-Wilf trees are discovered through these means (\S \ref{The Nested CW and SB Trees}). As the Fibonacci word pattern has found applications in crystallography with respect to quasicrystals and the field of condensed matter physics \cite[p. 141]{marder2010condensed}\cite{backman2013quantum}\cite{mansuy2017fibonacci}, the findings of this paper may have meaningful implications for further scientific research.

In order to provide the necessary perspective, this paper will follow the chronological order of discoveries that led logically to the main result. We will begin with the discovery of a fractal, which will serve as a window into the phinary numbers (\S \ref{section: fractal}). From there, we will make deep connections with the well-studied Fibonacci word pattern, allowing us to frame the infamous structure in a new light, as the counting pattern behind a set of ordinals. The presented fractal is shown to provide a bijection between features of its geometry and the phinary ordinals via a perspective projection (\S \ref{section: perspective projection}). Operations on the new numbers are defined and provide the means for applying the system to functions and practical purposes (\S \ref{section: operations}). We will then draw parallels between the Stern-Brocot and Calkin-Wilf trees, discovering a phinary analogue with an intricate and remarkable symmetry, and subsequently define a recurrence relation for the Fibonacci diatomic sequence (\S \ref{section: applications}).

Finally, equipped with the results outlined above, a geometric proof is formulated to define the cardinality of the natural numbers. The unexpected result describes a cardinality with dual values, zero and one (\S \ref{section: cardinal multiplicity}). Remaining questions and direction for future study are discussed in the concluding remarks.

\section{Fractal Origins} \label{section: fractal}
The origin of this research began through the discovery of a fractal; the geometry described here is reminiscent of a subset of Sierpinski's Gasket, dissecting an equilateral triangle through an infinite set of self-similar copies. After giving some background on the subject, we will present the fractal along with some of this geometry's noteworthy features and methods of generation.

\subsection{Monomorphic Tilings}

The complete dissection of a polygon into smaller, non-overlapping, similar copies of itself is known as a tessellation, and is a source of not only beautiful design but an object of interest in the applied sciences, as well---finding use in fields such as crystallography, materials science, architecture, and cellular networks \cite{mansuy2017fibonacci}\cite{zhang2018complex}\cite{park2017practical}\cite{arrighetti2005circuit}. Here, we address tessellations composed purely of self-similar tiles, a quality described as monomorphic \cite{n1992fivefold}, and consider some notable monomorphic tilings of the equilateral triangle.

\begin{figure}[ht]
    \begin{multicols}{3}
     \centering
    \begin{subfigure}[b]{0.3\textwidth}
        \includegraphics[width=\textwidth]{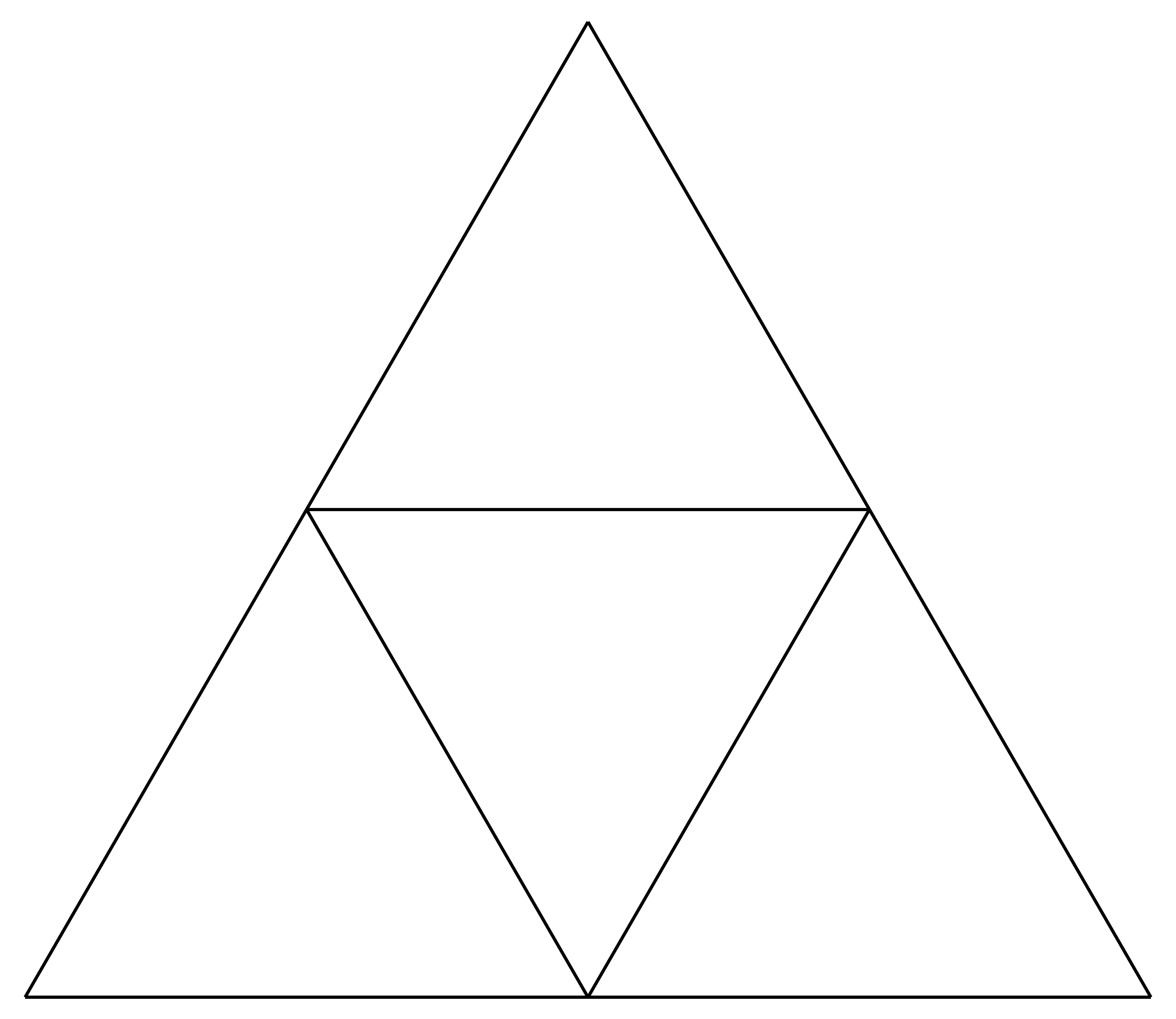}
        \caption{rep-4 (min.)}
        \label{fig:rep-4}
    \end{subfigure}
    ~ %add desired spacing between images, e. g. ~, \quad, \qquad, \hfill etc. 
    %(or a blank line to force the subfigure onto a new line)
    \par
     \begin{subfigure}[b]{0.3\textwidth}
        \includegraphics[width=\textwidth]{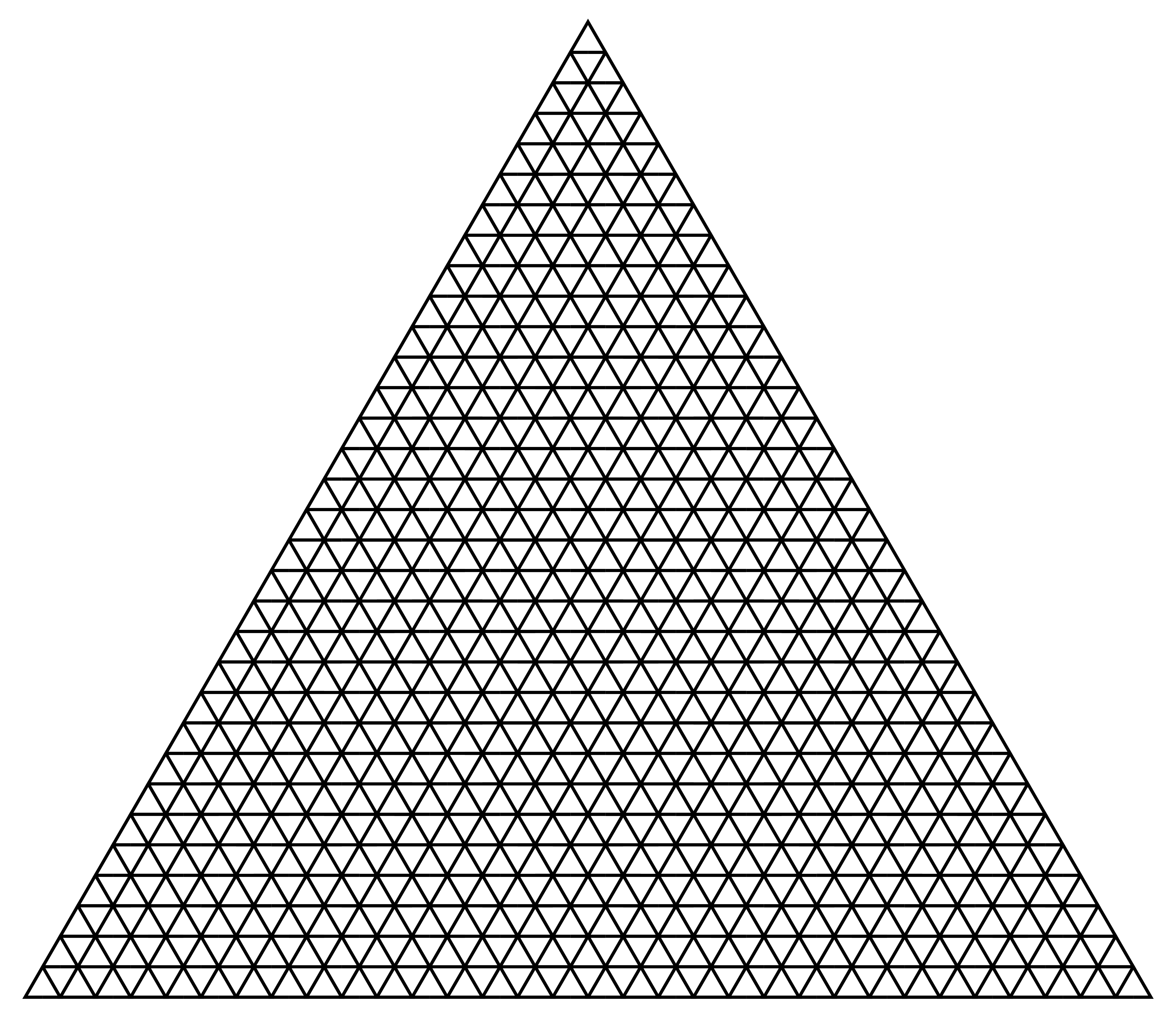}
        \caption{rep-$\infty$}
        \label{fig:rep-inf}
    \end{subfigure}
    ~ %add desired spacing between images, e. g. ~, \quad, \qquad, \hfill etc. 
    %(or a blank line to force the subfigure onto a new line)
    \par
    \begin{subfigure}[b]{0.3\textwidth}
        \includegraphics[width=\textwidth]{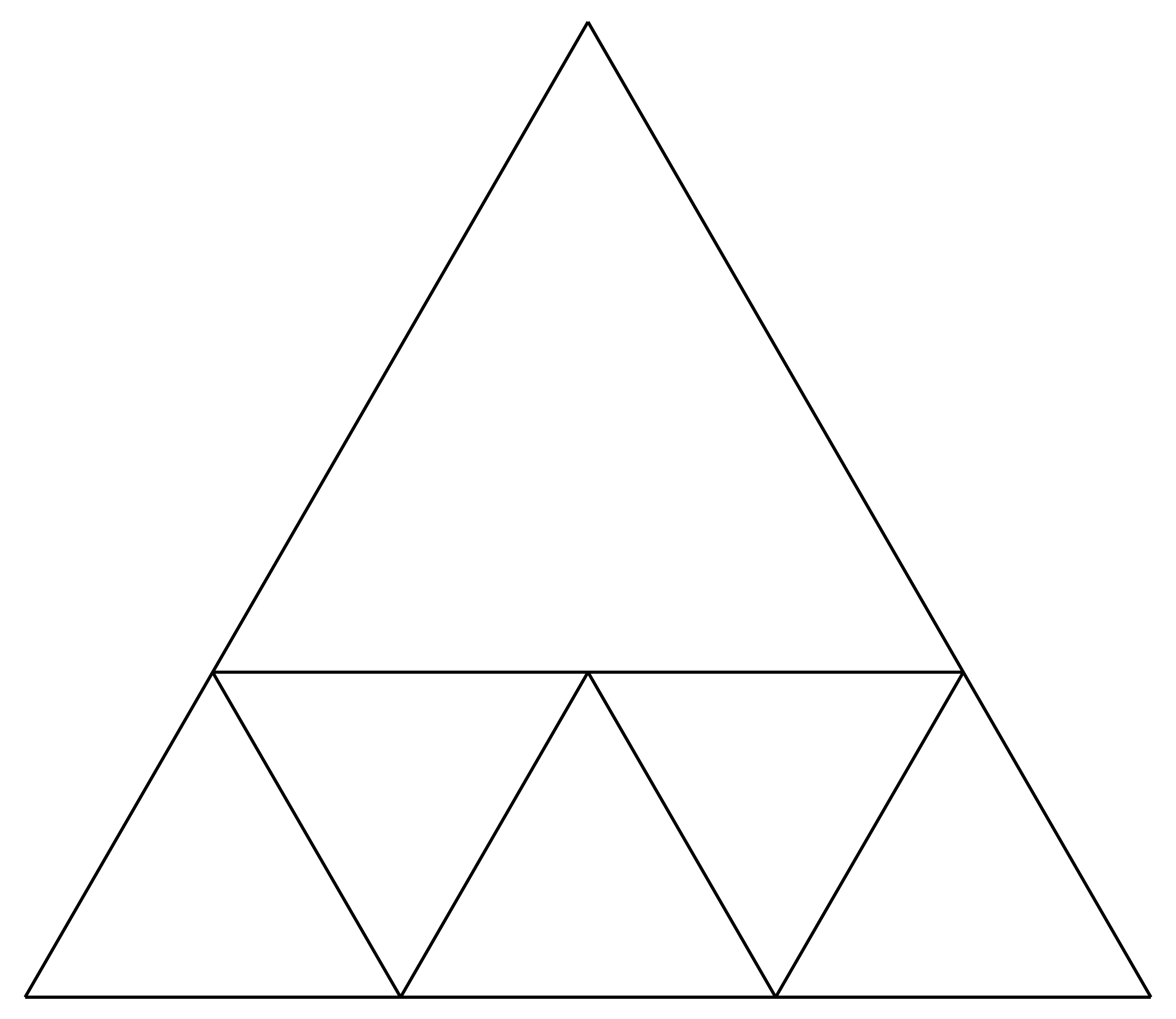}
        \caption{irrep-6 (min.)}
        \label{fig:irrep-6}
    \end{subfigure}
     ~ %add desired spacing between images, e. g. ~, \quad, \qquad, \hfill etc. 
     %(or a blank line to force the subfigure onto a new line)
     \par
         \begin{subfigure}[b]{0.3\textwidth}
        \includegraphics[width=\textwidth]{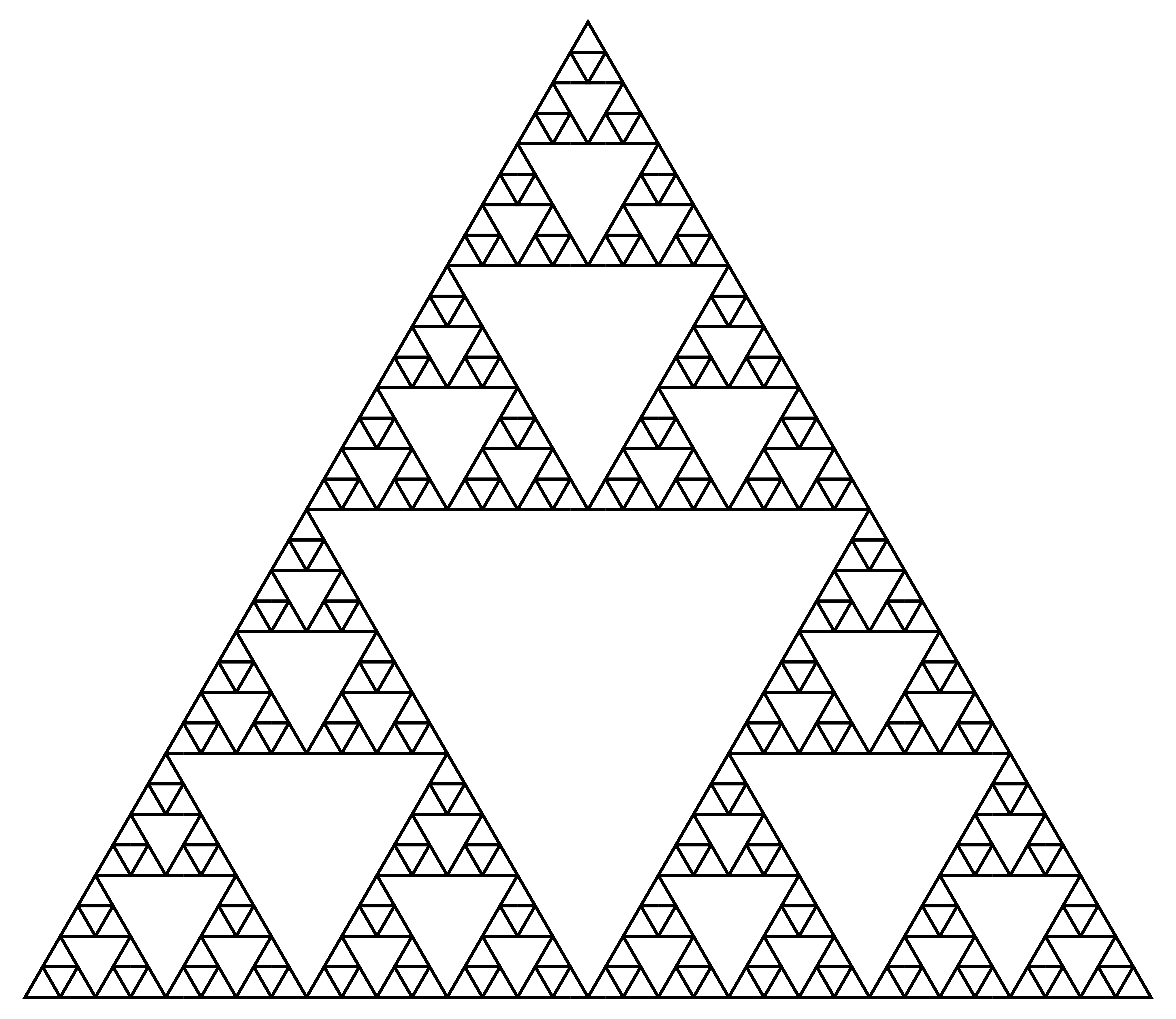}
        \caption{irrep-$\infty$ (Sierpinski's Gasket)}
        \label{fig:sierpinski}
    \end{subfigure}
     ~ %add desired spacing between images, e. g. ~, \quad, \qquad, \hfill etc. 
      %(or a blank line to force the subfigure onto a new line)
      \par
    \begin{subfigure}[b]{0.3\textwidth}
        \includegraphics[width=\textwidth]{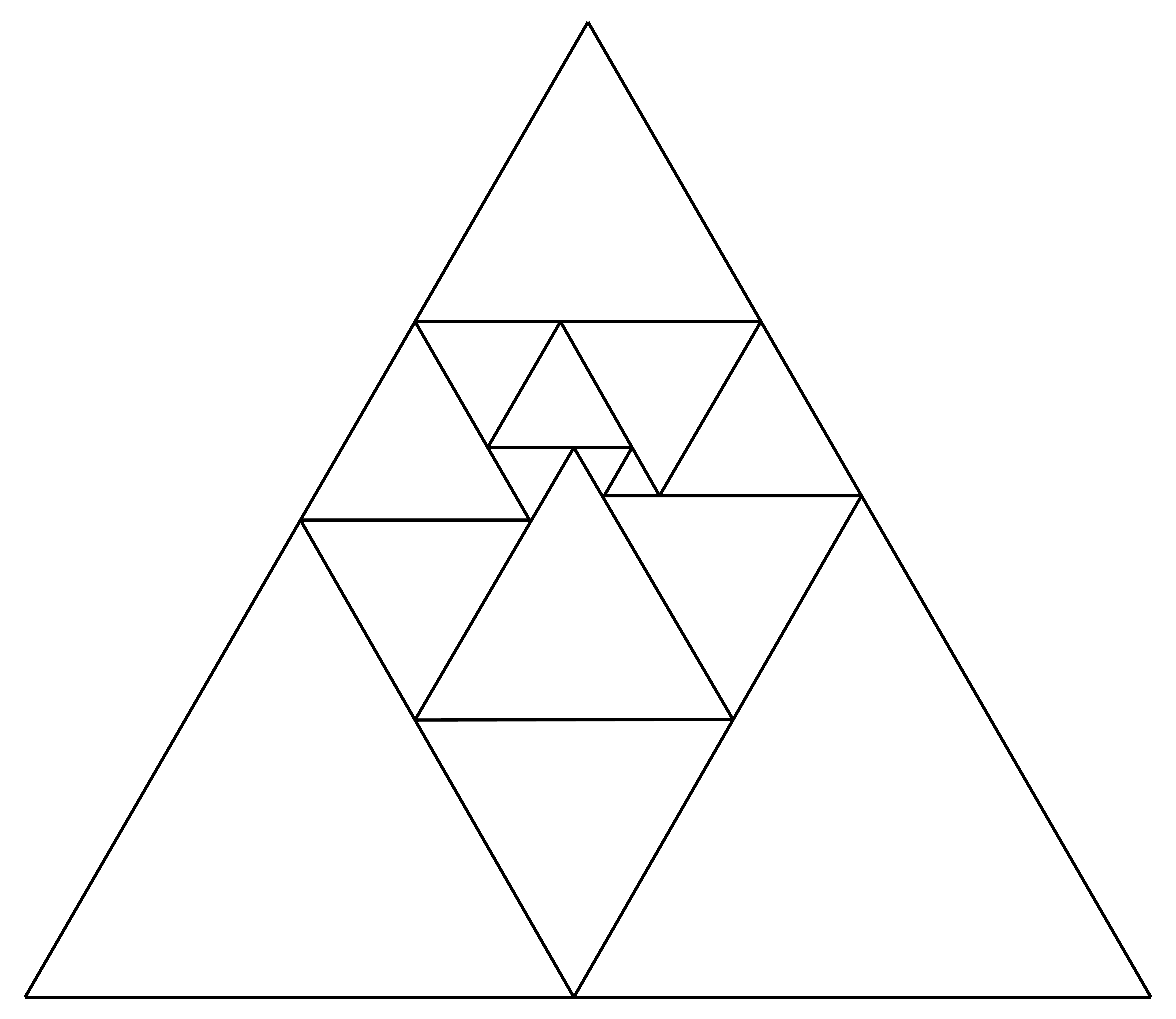}
        \caption{Perfect irrep-15 (min.)}
        \label{fig:perfect irrep-15}
    \end{subfigure}
        ~ %add desired spacing between images, e. g. ~, \quad, \qquad, \hfill etc. 
      %(or a blank line to force the subfigure onto a new line)
      \par
    \begin{subfigure}[b]{0.3\textwidth}
        \includegraphics[width=\textwidth]{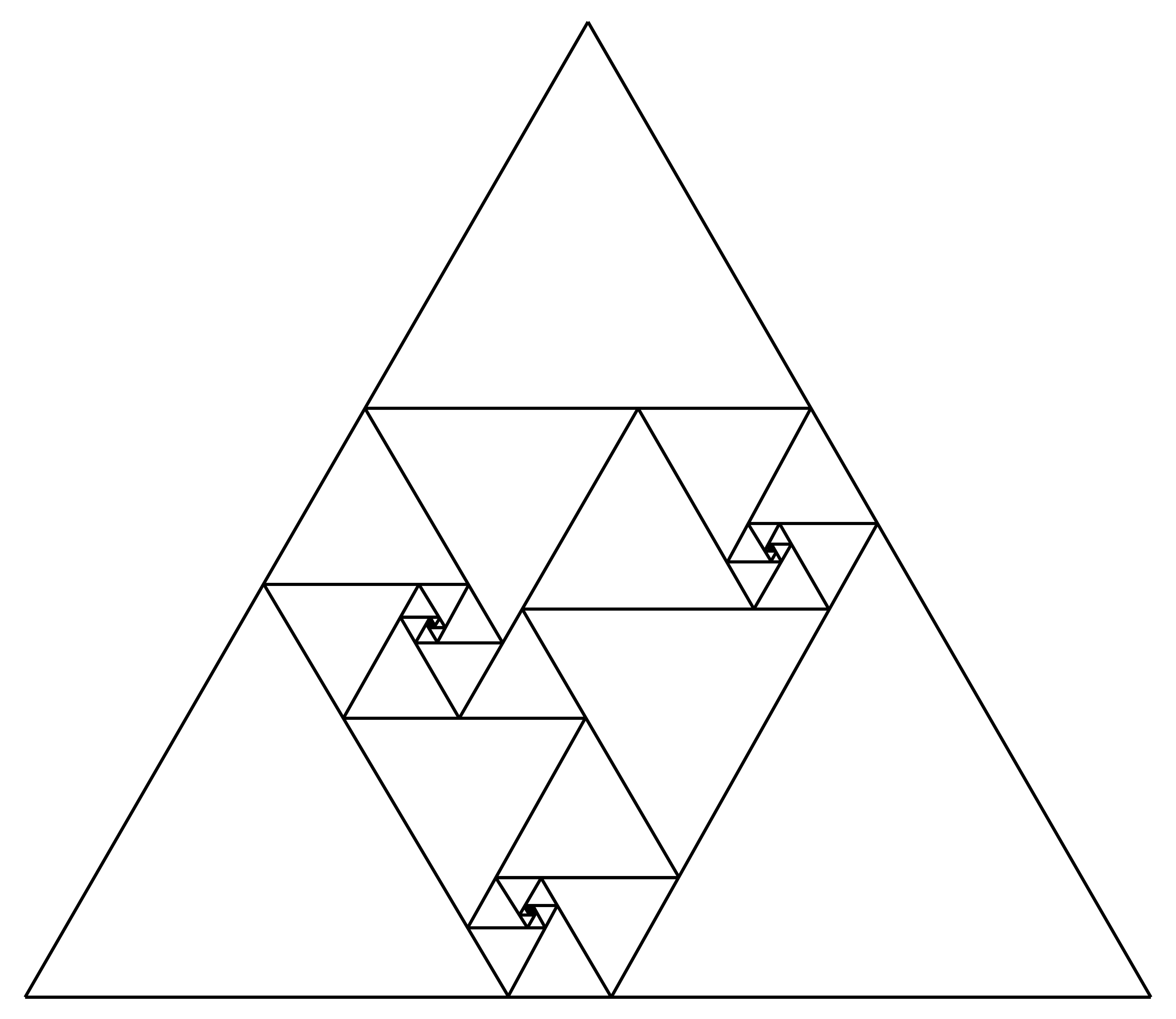}
        \caption{Perfect irrep-$\infty$}
        \label{fig:perfect irrep-inf}
    \end{subfigure}
    ~ %add desired spacing between images, e. g. ~, \quad, \qquad, \hfill etc. 
    %(or a blank line to force the subfigure onto a new line)
    \par
    \end{multicols}
    \caption{Shown are some notable monomorphic tilings of the equilateral triangle: (a) The minimum nontrivial rep-tiling (b) An infinite rep-tiling via infinite iterations of rep-4 (a finite iteration shown for illustrative purposes) (c) The minimum irrep-tiling scheme (d) An example of an asymptotic irrep-tiling tessellation via iterations of rep-4 in non-adjacent tiles, resulting in a fractal geometry (e) The minimum perfect irrep-tiling tessellation (f) A perfect asymptotic irrep-tiling with golden ratio proportions---also a fractal.}
    \label{fig:triangle tilings}
\end{figure}

When the copies within a monomorphic tessellation are all of the same size, they are known as rep-tiles or reptiles, as coined by the mathematician Solomon W. Golomb in a pun for objects---such as reptilian animals---that replicate themselves \cite[pp. 46-58]{gardner2001colossal}. If the tiling requires \textit{n} rep-tiles to generate the specified shape, the shape is said to be rep-\textit{n}. The equilateral triangle has a minimum nontrivial rep-tiling of rep-4 (Figure \ref{fig:rep-4}), as four equilateral triangles of the same size can tessellate to produce a new equilateral triangle \cite{drapal2010enumeration}. Naturally, this rep-4 tessellation can be iterated infinitely many times, resulting in an infinite dissection of the equilateral triangle (i.e. a tiling of rep-$\infty$, Figure \ref{fig:rep-inf}).

The dissection of a shape that contains tiles which are similar but of unequal sizes is said to be composed of irrep-tiles \cite{klaassen1995infinite}. If \textit{n} irrep-tiles are required to generate the intended shape, the shape is said to be irrep-\textit{n}. The minimum number of irrep-tiles to fill an equilateral triangle is six, denoted irrep-6 (Figure \ref{fig:irrep-6})\cite{drapal2010enumeration}. An example of an irrep-$\infty$ tessellation is the well-known Sierpinski's Gasket (Figure \ref{fig:sierpinski}), generated through repeated iterations of the rep-4 tessellation in the nonadjacent tiles; the result is a fractal with an approximate Hausdorff dimension 1.5849 \cite[p. 58]{mccartin2010mysteries}. Tessellations, such as Sierpinski's Gasket, with an infinite number of irrep-tiles that progressively decrease in size can be referred to as asymptotic \cite[p. 143]{n1992fivefold}.

A special type of irrep-tiling, described as ``perfect", occurs if each and every irrep-tile contained within a tessellation is similar but non-congruent (i.e. same shape, but different size and/or orientation) \cite{drapal2010enumeration}. The minimum perfect dissection for the equilateral triangle, first conjectured by W. T. Tutte in 1948 \cite{tutte1948dissection} and later confirmed by A. Dr\'apal and C. H\"am\"al\"ainen in 2003 \cite{drapal2010enumeration}, is irrep-15 (Figure \ref{fig:perfect irrep-15}), the result of which was found through the creative use of techniques relating to Kirchhoff's electrical current laws. A perfect irrep-$\infty$ was found by Bernhard Klaassen in 1995, through the Padovan triangle spiral---a tiling progression of golden ratio proportions \cite{klaassen1995infinite} (Figure \ref{fig:perfect irrep-inf}).

\begin{figure}[!p]
    \centering
    \includegraphics[width=\textwidth]{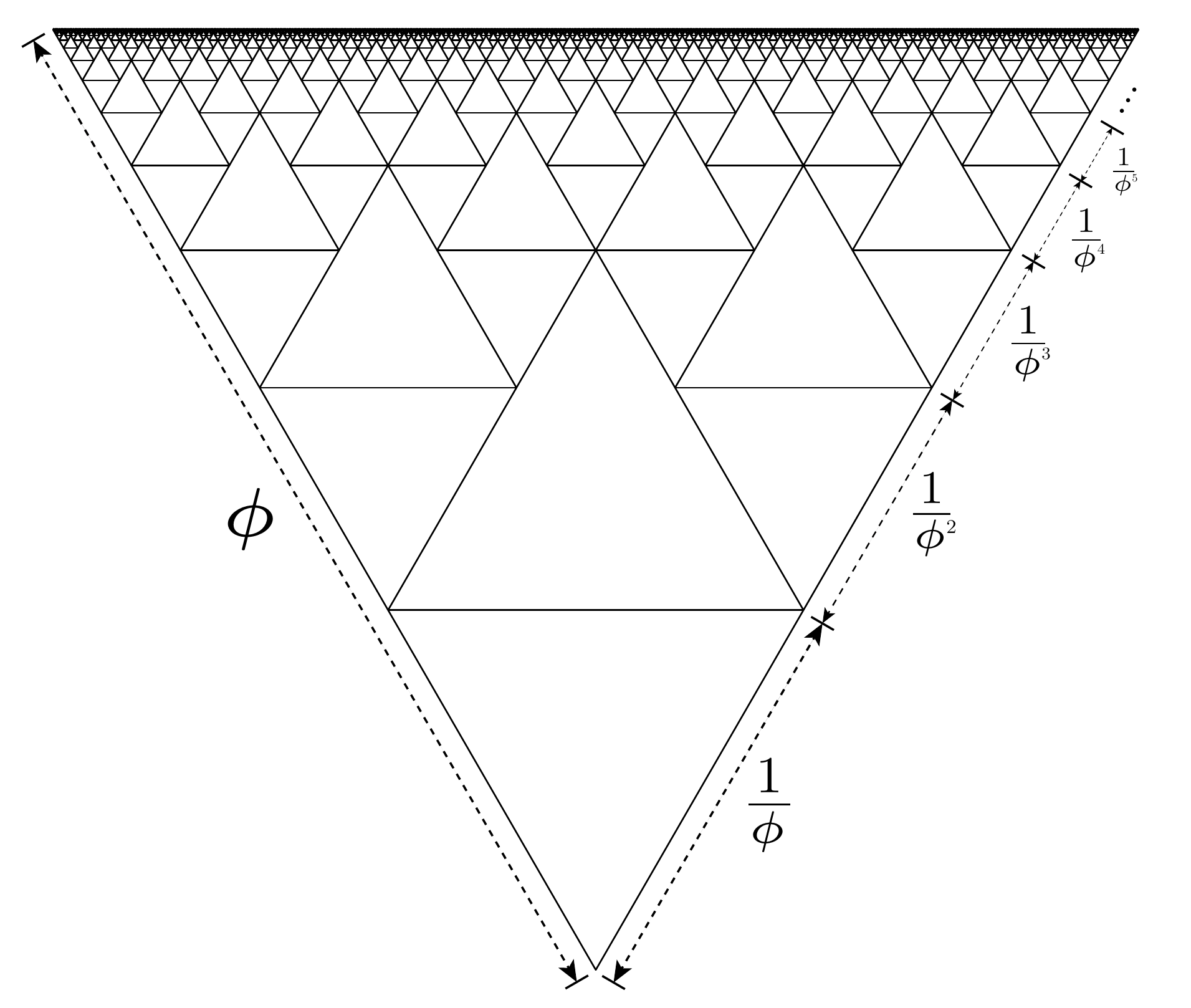}
    \caption{The golden diamond (GD): An asymptotic monomorphic dissection of the equilateral triangle, irrep-$\infty$. Each tile has proportions equal to a power of the golden ratio with respect to the outer triangle. The relationship of dimensions between the facets along an edge and the outer triangle expresses the identity $\phi=\sum_{n=1}^\infty \phi^{-n}$.}
    \label{fig:GD bw}
\end{figure}

\subsection{The Golden Diamond} \label{subsection: golden diamond}
The following is a newly described tessellation of the equilateral triangle through a monomorphic and asymptotic dissection, irrep-$\infty$ (Figure \ref{fig:GD bw}). Each irrep-tile has proportions equal to a power of the golden ratio with respect to the outer triangle's dimensions. In a down-pointed orientation, the figure resembles the profile of a cut diamond with an infinite number of facets, and therefore has been dubbed the golden diamond (GD)\footnote{The name ``golden diamond" is chosen partly for the reason that ``Golden Triangle" has been used in the past to describe a snowflake-like fractal that also employs the golden ratio \cite[p. 17]{walser2001golden}}. As is easily seen, the image is a fractal, exhibiting self-similarity, and an iterative generation process. 

The GD can be constructed in a number of ways. One method is to start with a single up-pointed equilateral triangle and add two similar copies---each smaller by a factor of the golden ratio---upon the ``shoulders" of the first. A shoulder is located at a distance from a base vertex equal to the length of the new generation of triangles. The process is repeated for each new generation of triangles. After an infinite number of iterations, a final down-pointing equilateral triangle can be drawn around the image, intersecting the vertices of the perimeter triangles (Figure \ref{fig:GD construction}).

\begin{figure}[!ht]
     \centering
    \begin{subfigure}[b]{0.3\textwidth}
        \includegraphics[width=\textwidth]{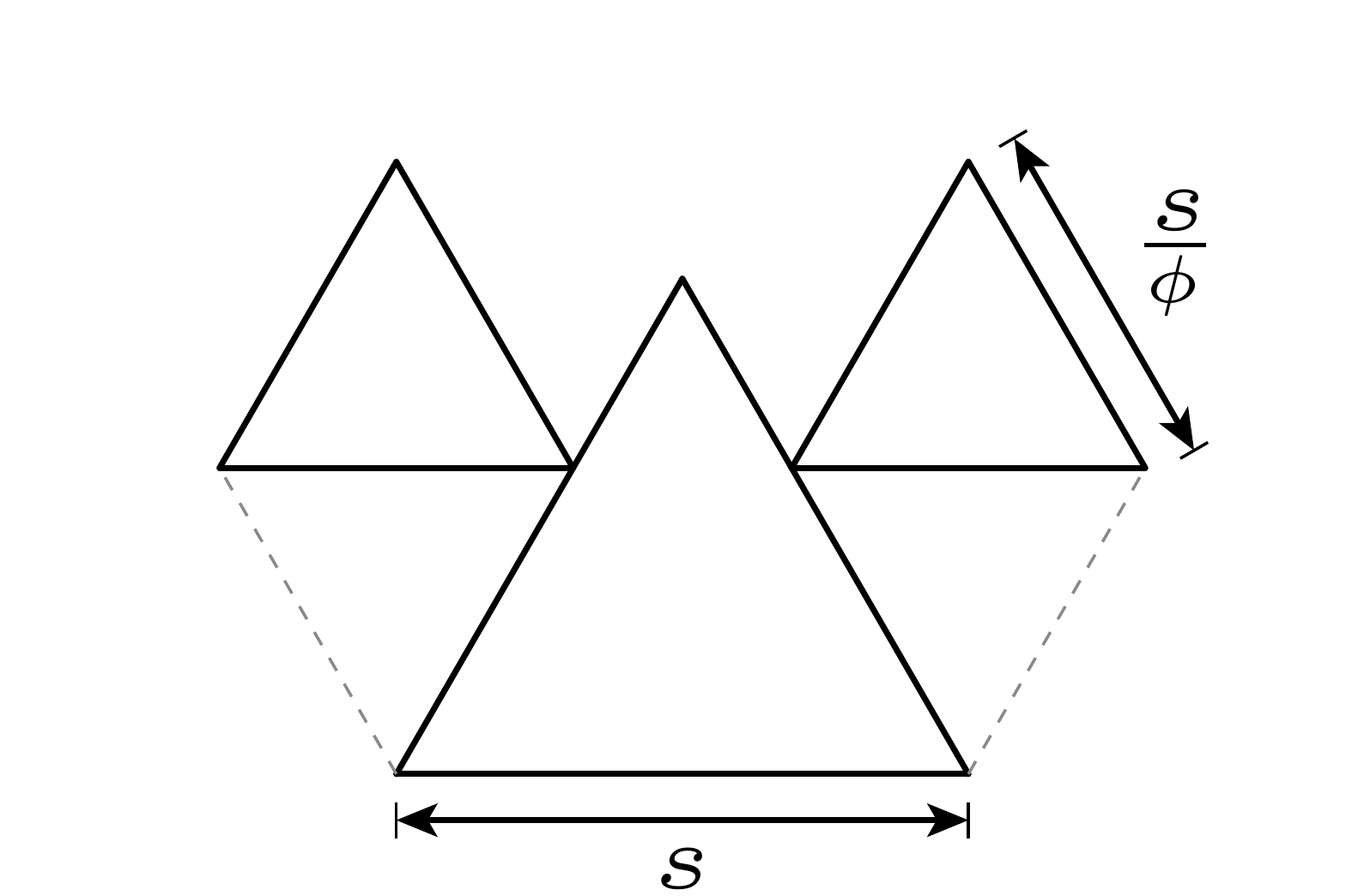}
        \caption{$1^\text{st}$ iteration}
        \label{fig:step 1}
    \end{subfigure}
    ~ %add desired spacing between images, e. g. ~, \quad, \qquad, \hfill etc. 
    %(or a blank line to force the subfigure onto a new line)
    \begin{subfigure}[b]{0.3\textwidth}
        \includegraphics[width=\textwidth]{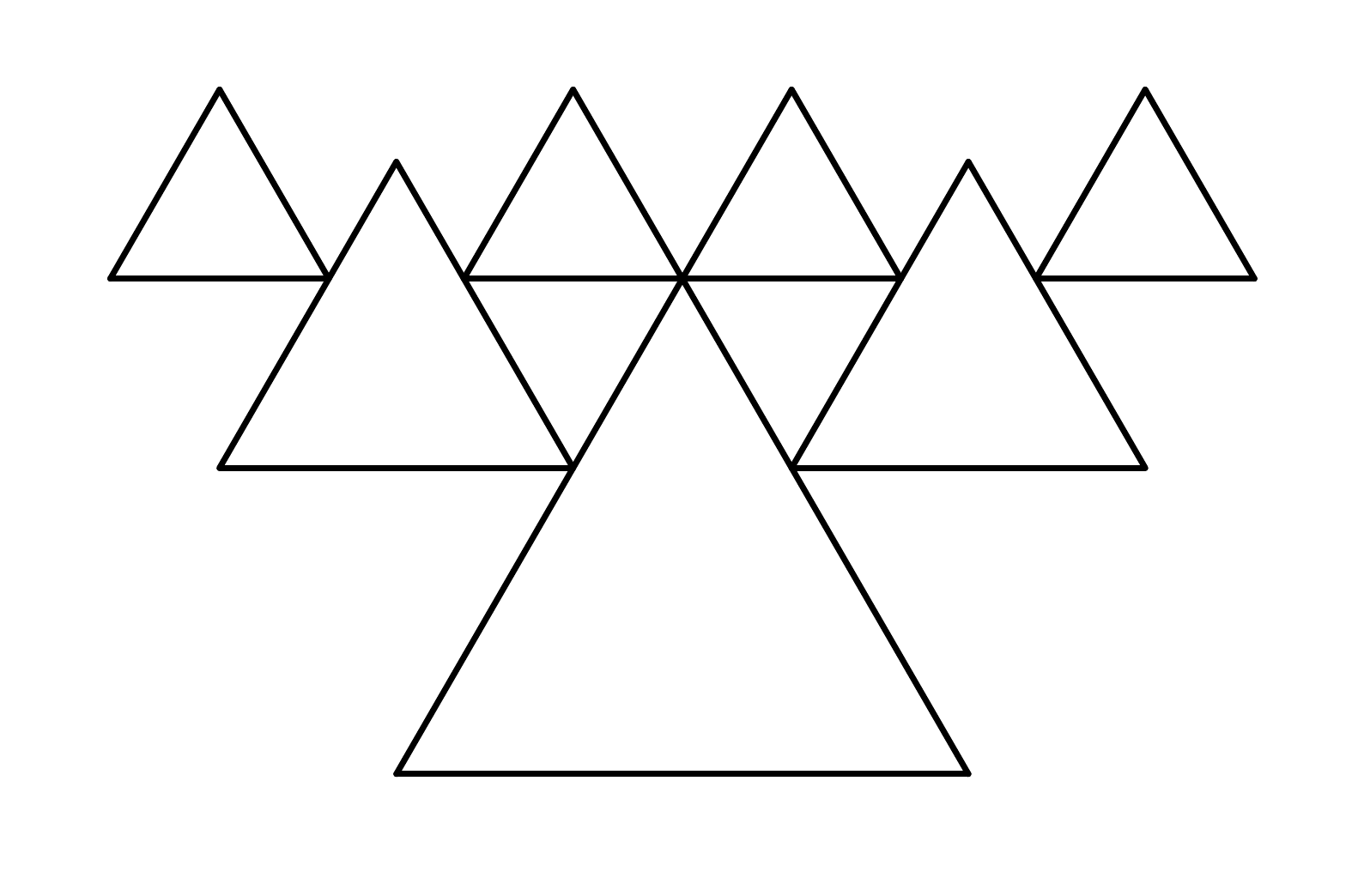}
        \caption{$2^\text{nd}$ iteration}
        \label{fig:step 2}
    \end{subfigure}
    ~ %add desired spacing between images, e. g. ~, \quad, \qquad, \hfill etc. 
    %(or a blank line to force the subfigure onto a new line)
    \begin{subfigure}[b]{0.3\textwidth}
        \includegraphics[width=\textwidth]{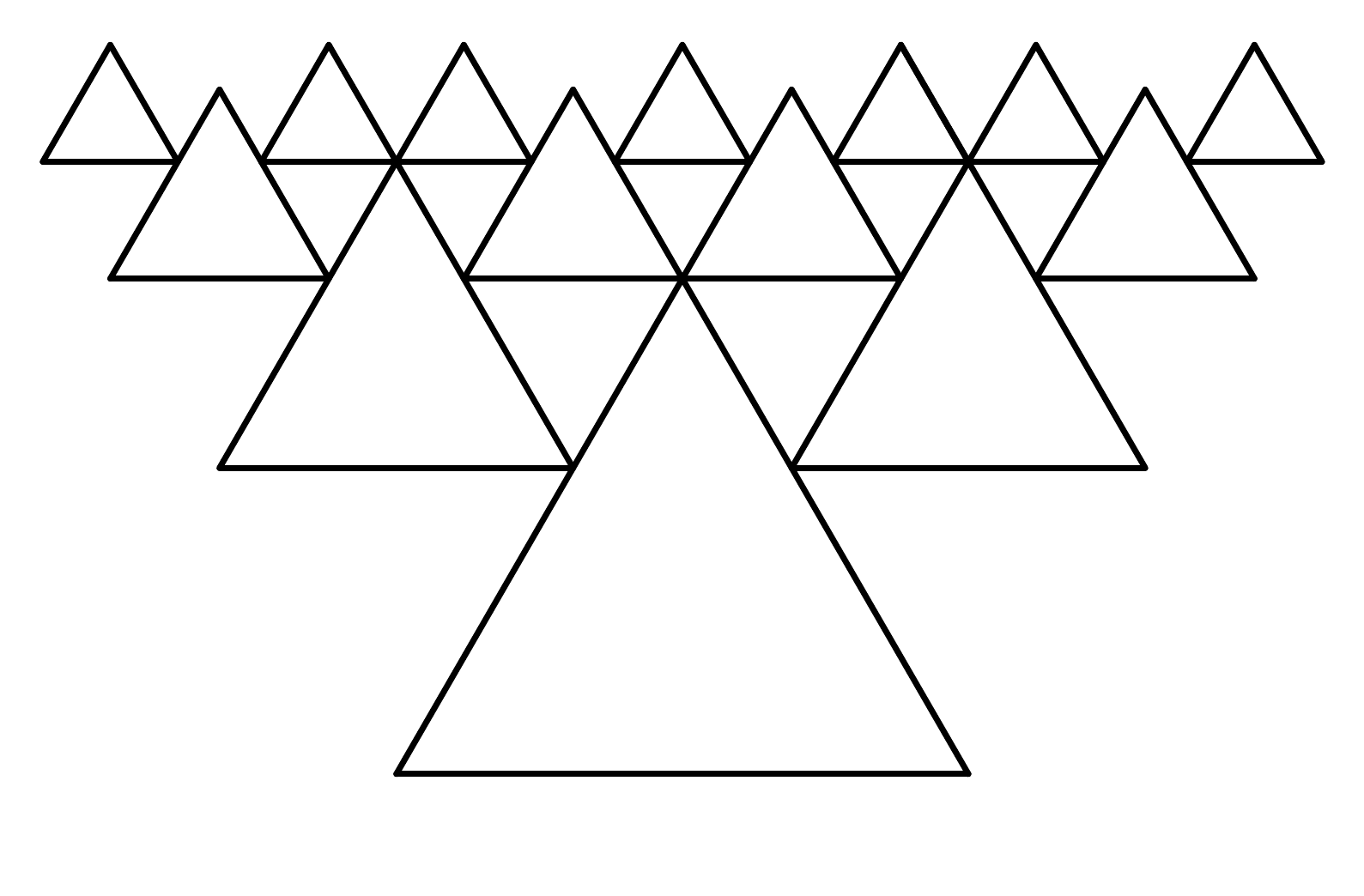}
        \caption{$3^\text{rd}$ iteration}
        \label{fig:step 3}
    \end{subfigure}
    ~ %add desired spacing between images, e. g. ~, \quad, \qquad, \hfill etc. 
    %(or a blank line to force the subfigure onto a new line)
    \begin{subfigure}[b]{0.3\textwidth}
        \includegraphics[width=\textwidth]{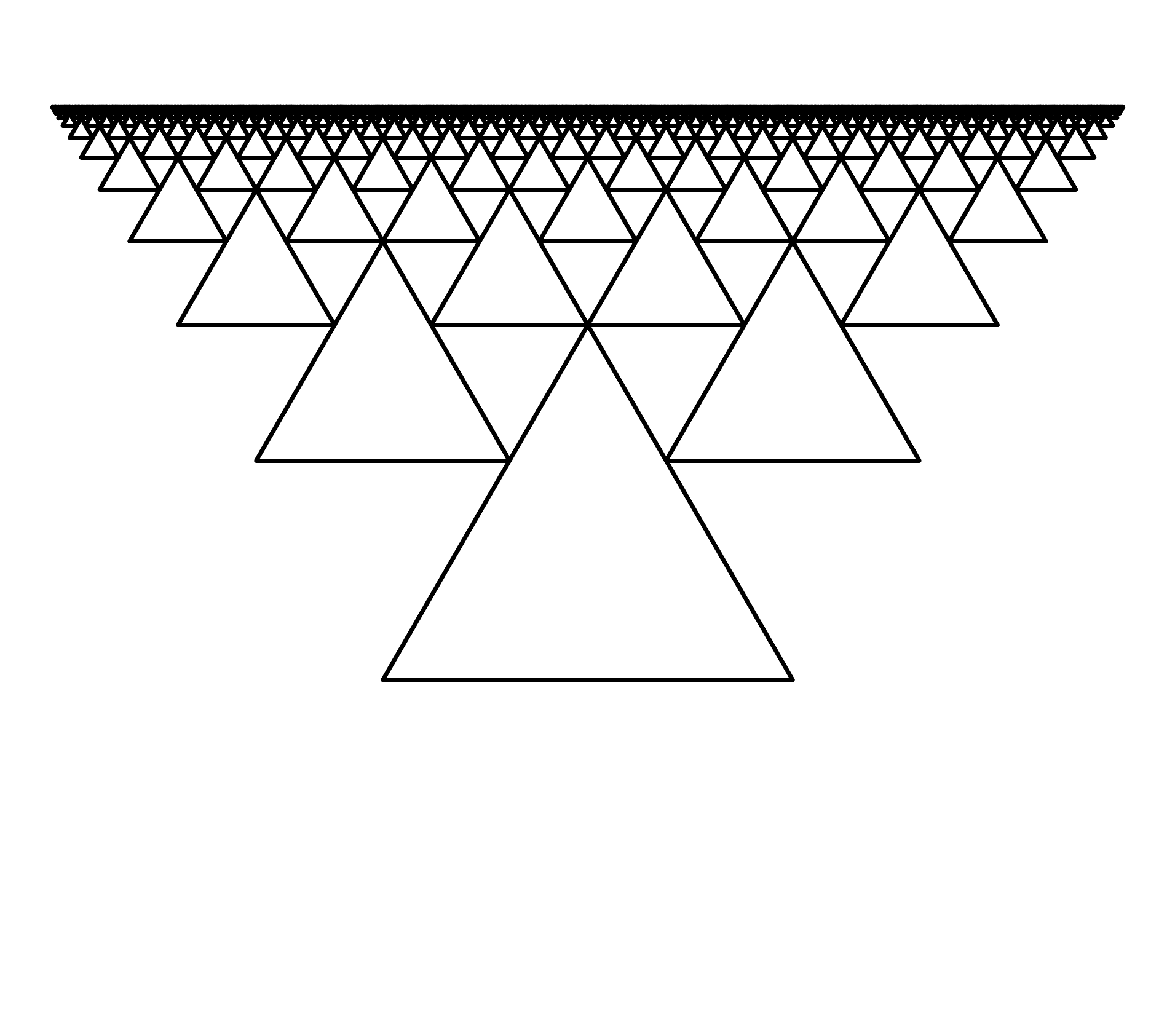}
        \caption{$n^\text{th}$ iteration}
        \label{fig:step n}
    \end{subfigure}
    ~ %add desired spacing between images, e. g. ~, \quad, \qquad, \hfill etc. 
    %(or a blank line to force the subfigure onto a new line)
    \begin{subfigure}[b]{0.3\textwidth}
        \includegraphics[width=\textwidth]{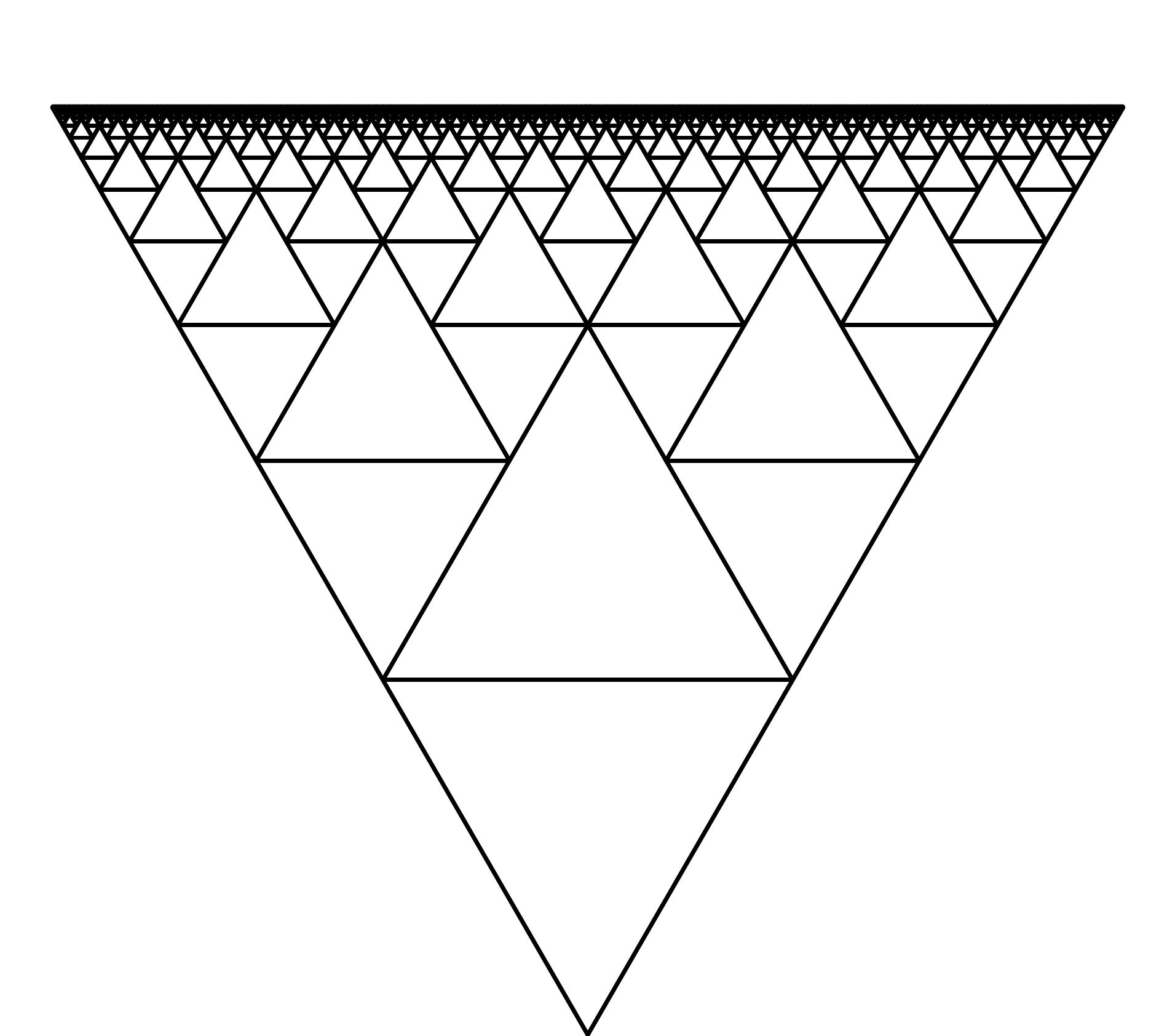}
        \caption{Completed image}
        \label{fig:GD complete}
    \end{subfigure}
    ~ %add desired spacing between images, e. g. ~, \quad, \qquad, \hfill etc. 
    %(or a blank line to force the subfigure onto a new line)    
    \caption{A step by step construction of the golden diamond. Two similar copies of an equilateral triangle---each smaller by a factor of the golden ratio---are stacked upon the ``shoulders" of the first. A shoulder is located at a distance from a base vertex equal to the length of the new generation of triangles. The process is repeated for each new generation of triangles. After an infinite number of iterations, a final equilateral triangle can be drawn around the image, intersecting the vertices of the perimeter triangles.}
    \label{fig:GD construction}
\end{figure}

\subsection{Definitions}

We will now look at how the geometry of the GD is defined and establish some vocabulary to help evaluate its properties.

\begin{definition}
The equilateral triangle that encloses the GD will be referred to as the $outer$ $triangle$; the length of one of its sides is denoted $s$ and has value
\begin{align}
	s=\phi
\end{align}
where $\phi$ is the golden ratio.
\end{definition}

\begin{definition}
The non-intersecting, similar equilateral triangles contained within the outer triangle of the golden diamond are referred to as $facets$.
\end{definition}

\begin{definition}
A set containing all the facets of a particular size is referred to as a $row$. The row with the largest facets is denoted $\text{R}_1$, the row with the second largest facets is $\text{R}_2$, and so on. Row $\text{R}_n$ contains the set of the $n$th largest facets in the GD.
\end{definition}

\vspace{5mm}

\begin{definition}
Each facet contained in a row $\text{R}_n$ is denoted $t_n$ with side length $s_n$. Where 
\begin{align} \label{eq: sides}
	s_n=\frac{1}{\phi^n}
\end{align}
where $\phi$ is the golden ratio.
\end{definition}

\vspace{5mm}

\subsection{The Golden Ratio and Fibonacci Numbers} \label{subsection: The Golden Ratio and Fibonacci Numbers}

As the golden ratio and Fibonacci numbers play an integral role in the geometry of the golden diamond, we will briefly review their definitions and some of their most relevant features.

\vspace{5mm}

The solutions to the quadratic equation
\begin{align*}
	x^2-x-1=0
\end{align*}
are the golden ratio $\phi$ and its conjugate $-\frac{1}{\phi}$, where
\begin{align*}
	\phi&=\frac{1+\sqrt{5}}{2}\approx1.618034...\\
	-\frac{1}{\phi}&=\frac{1-\sqrt{5}}{2}\approx-0.618034...
\end{align*}
Consequently, a number of unique algebraic identities are associated with the golden ratio, such as
\begin{align*}
	\phi^2=\phi+1\ \ \ \ \text{and} \ \ \ \ \frac{1}{\phi}=\phi-1
\end{align*}
or more generally,
\begin{align} \label{eq: phi power recursive relation}
	\phi^{n+1}=\phi^{n}+\phi^{n-1}
\end{align}
for all integer $n$, giving rise to the identity
\begin{align*}
	\phi=\sum_{n=1}^\infty \phi^{-n}.
\end{align*}
Closely related to the golden ratio are the Fibonacci numbers, which are defined as the terms $F_n$ in the recursive sequence with relation $F_{n+1}=F_n+F_{n-1}$ where $F_1=1, F_2=1$. The first few terms of the sequence are $1, 1, 2, 3, 5, 8, 13, 21, \ldots$ The ratio of two consecutive Fibonacci numbers approaches the golden ratio as $n$ approaches infinity, or in formulas
\begin{align*}
	\lim_{n \to \infty} \frac{F_n}{F_{n-1}}=\phi.
\end{align*}
Furthermore, any power of the golden ratio can be written as a linear expression in $\phi$ with Fibonacci number coefficients
\begin{align} \label{eq: fib phi linear relation}
	\phi^n=F_n\phi+F_{n-1}. 
\end{align}
One beneficial consequence of the above is that any linear combination of powers of the golden ratio can be reduced to a linear expression $a\phi+b$ for some values $a$ and $b$---a property that will be utilized often in this paper.

There are hundreds of identities of the Fibonacci numbers---many of them relating to the golden ratio (See \cite{vajda1989fibonacci} for a thorough list). Perhaps, the most prized of these identities is the closed-form expression between the $n$th Fibonacci number and the golden ratio, given by Binet's formula\footnote{A formula named after Jacques Philippe Marie Binet $(1786 - 1856)$, but known to Abraham de Moivre as early as 1730 \cite{knuth1968art}.}
\begin{align}
	F_n=\frac{\phi^n-(-\phi)^{-n}}{\sqrt{5}}, \tag{Binet's Formula}
\end{align}
which, upon evaluation, will always equal a positive integer value for $n \in \mathbbm{Z}^+$.

\subsection{Geometric Properties of the Golden Diamond}

There are many interesting geometric properties of the GD. Most apparent is the visualization it provides of the well-known identity
\begin{align*}
	\phi=\sum_{n=1}^\infty \phi^{-n} 
\end{align*}
found by relating the dimensions of facets along an edge of the GD to the outer triangle as is clearly shown in Figure \ref{fig:GD bw}.

More exotic identities can be found by considering the rows of facets in the fractal. The interlinked rows contained in the golden diamond fractal are composed of pairs of facets with equal size. For example, row $\text{R}_1$ contains one pair of $t_1$ facets, row $\text{R}_2$ contains two pairs of $t_2$ facets, $\text{R}_3$ has four pairs of $t_3$ facets, and $\text{R}_4$ has seven pairs of $t_4$ facets. The number of pairs of facets in each row, starting at $\text{R}_1$, are $1, 2, 4, 7, 12, 20, ...$ A general pattern presents itself:

\vspace{5mm}

\begin{lemma} \label{lemma: R_n}
A row $\text{R}_n$ is the set of $t_n$ facets such that
\begin{align} \label{eq: R_n}
\text{R}_n=2(F_{n+2}-1)t_n
\end{align}
where $F_n$ is the $n$th Fibonacci number. The number of facets in each row is therefore
	\begin{align} \label{eq: |R_n|}
		|\text{R}_n |=2(F_{n+2}-1).
	\end{align}
\end{lemma}

\begin{proof}
A rigorous treatment must wait until \S \ref{subsection: Fib words in GD} where it is provided in the proof of Theorem \ref{theorem: initial palindromic Fibonacci subwords in GD}.
\end{proof}

\vspace{5mm}

We can now use this observation to prove a number of interesting relationships between the golden ratio and Fibonacci numbers.

\vspace{5mm}

\begin{theorem}
	The following relationship between the golden ratio $\phi$ and the $n$th Fibonacci number $F_n$ holds
	\begin{align*} 
	\frac{1}{2}=\sum_{n=1}^{\infty} \frac{F_{n+2}-1}{\phi^{2n+2}}
	\end{align*}
\end{theorem}

\begin{proof}
By Lemma \ref{lemma: R_n}, the number of facets in the $n$th row is $2(F_{n+2}-1)$. As shown in Equation \ref{eq: sides}, the facets in the $n$th row have side length $s_n=\phi^{-n}$ for an outer triangle with side length $s=\phi$. Therefore, the area of a $n$th facet is $A_{t_n}=\frac{\sqrt{3}}{4}\phi^{-2n}$ for an outer triangle with area $A=\frac{\sqrt{3}}{4}\phi^2$. The sum of the areas per row is

\begin{align*}
	A_n=\frac{\sqrt{3}(F_{n+2}-1)}{2\phi^{2n}}
\end{align*}

By definition, the complete tiling of triangles fills the entire space of the outer triangle. Therefore, the sum of the areas $A_n$, as the limit of $n$ approaches infinity, equals the area $A$. In formulas,

\begin{align*}
	A&=\sum_{n=1}^{\infty} A_n \\
	\frac{\sqrt{3}}{4}\phi^2&=\sum_{n=1}^{\infty} \frac{\sqrt{3}(F_{n+2}-1)}{2\phi^{2n}} \\
	\frac{1}{2}&=\sum_{n=1}^{\infty} \frac{F_{n+2}-1}{\phi^{2n+2}}\\
\end{align*}
\end{proof}

\vspace{5mm}

We now turn our attention from evaluating the sum of areas in each row to evaluating the sum of the side lengths of the facets in each row. We find the following elegant relationship between the golden ratio and Fibonacci numbers.

\vspace{5mm}

\begin{figure}[!p]
    \centering
    \includegraphics[width=\textwidth]{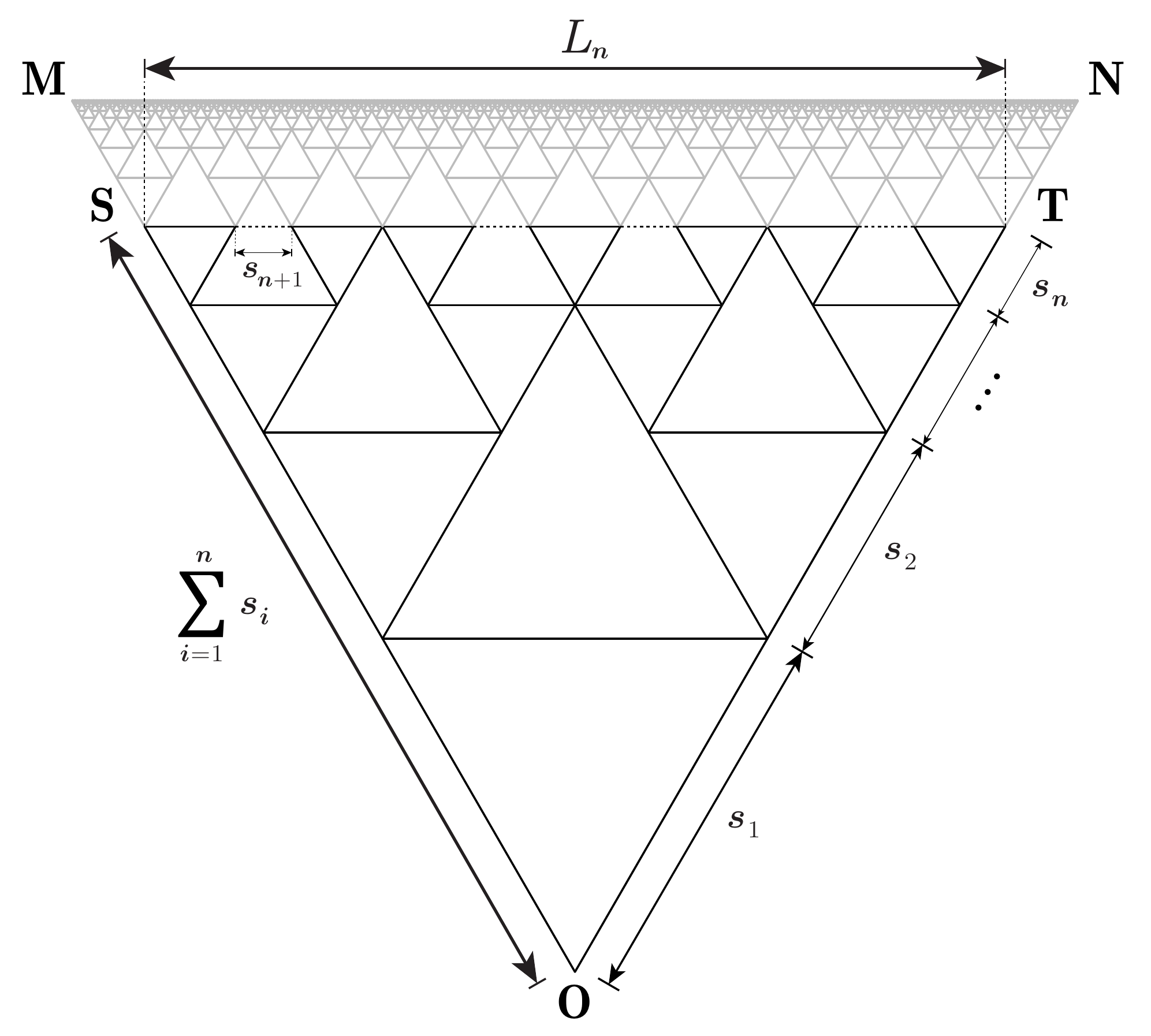}
    \caption{The relationship of sides within a subset of the golden diamond.}
    \label{fig: GD subset identity}
\end{figure}

\begin{theorem} \label{theorem: fib-phi formulas}
	The following relationship between the golden ratio $\phi$ and the $n$th Fibonacci number $F_n$ holds
	\begin{align} \label{eq: row sides 1}
	\lim_{n\to\infty} \frac{F_{n}}{\phi^{n-1}}+\frac{F_{n-1}}{\phi^{n}}=1.
	\end{align}
	Or more generally,
	\begin{align} \label{eq: row sides 2}
	\frac{F_{n+2}-1}{\phi^n}+\frac{F_{n+1}-1}{\phi^{n+1}}=\sum_{i=1}^{n} \frac{1}{\phi^{i}}
	\end{align}
	for $n>0$.
\end{theorem}

\begin{proof}
	We can define the length $L_n$ of the $n$th row in the fractal as the sum of the side lengths $s_n$ shared by the pairs of facets in the given row, plus the sum of the ``gap" lengths occurring between facet pairs. The length of each gap in row $n$ is equal to $g_n$=$s_{n+1}$, the same as the side length of a facet in row $n+1$ (Figure \ref{fig: GD subset identity}). The number of gaps in the $n$th row is equal to the number of facet pairs in the preceding row, $\frac{1}{2}|R_{n-1}|$. Therefore, for a given row, we have length
	
\begin{align} \label{eq: L_n}
	L_n&=\frac{|R_{n}|}{2}s_n+\frac{|R_{n-1}|}{2}g_n \nonumber \\
	&=(F_{n+2}-1)s_n+(F_{n+1}-1)s_{n+1} \nonumber \\
	&=\frac{F_{n+2}-1}{\phi^n}+\frac{F_{n+1}-1}{\phi^{n+1}}.
\end{align}

We notice that $L_n$ can be viewed as the side length $s_{R_n}$ of a relatively large equilateral triangle $t_{R_n}$ within the outer triangle $t$, which shares the same bottom vertex and a proportion of the left and right sides. Let the top left and right vertices of $t$ be denoted M and N, respectively, and the bottom vertex be denoted O. If the top left and right vertices of $t_{R_n}$ formed by side length $L_n$ are labelled S and T, respectively, then

\begin{align*}
	\text{SO} \leq \text{MO} \ \ \ \text{and} \ \ \  \text{TO} \leq \text{NO}.
\end{align*}

It can be observed that the side lengths SO and TO are each equal to the sum of the side lengths $s_1$ through $s_n$ of the facets $t_1$ through $t_n$,

\begin{align} \label{eq: SO TO}
	\text{SO} = \text{TO} = \sum_{i=1}^{n} s_k = \sum_{i=1}^{n} \frac{1}{\phi^n}
\end{align}

Furthermore, $L_n = \text{SO} = \text{TO}$. Therefore, from Equations \ref{eq: L_n} and \ref{eq: SO TO}, we have

\begin{align*}
	\frac{F_{n+2}-1}{\phi^n}+\frac{F_{n+1}-1}{\phi^{n+1}} &= \sum_{i=1}^{n} \frac{1}{\phi^i}
\end{align*}
which is shown in the theorem. Substituting $n-2$ for $n$ and dividing both sides by $\phi$ yields
\begin{align*}
	\frac{F_{n}-1}{\phi^{n-2}}+\frac{F_{n-1}-1}{\phi^{n-1}} &= \sum_{i=1}^{n-2} \frac{1}{\phi^{i}}\\
	\frac{F_{n}-1}{\phi^{n-1}}+\frac{F_{n-1}-1}{\phi^{n}} &= \sum_{i=1}^{n-2} \frac{1}{\phi^{i+1}}.
\end{align*}
Now, letting $n$ approach infinity, we find
\begin{align*}
	\lim_{n\to\infty} \frac{F_{n}-1}{\phi^{n-1}}+\frac{F_{n-1}-1}{\phi^{n}} &= \lim_{n\to\infty} \sum_{i=1}^{n-2}  \frac{1}{\phi^{i+1}}\\
	\lim_{n\to\infty} \frac{F_{n}}{\phi^{n-1}}+\frac{F_{n-1}}{\phi^{n}} &= 1
\end{align*}
which concludes the proof.
\end{proof}

\vspace{5mm}

\begin{lemma} \label{lemma: convergent values}
The following ratios between the Fibonacci number and golden ratio hold in the limit as $n$ approaches infinity\footnote{This result can be used to generate a Beatty sequence. By Beatty's theorem \cite{beatty1926problems}, given two irrational values $\alpha>1$ and $\beta>1$ such that $\frac{1}{\alpha}+\frac{1}{\beta}=1$, two sequences $\mathcal{S}_\alpha$ and $\mathcal{S}_\beta$ can be generated that partition the set of positive integers, sending each positive integer to exactly one of the two sequences---where $\mathcal{S}_\alpha=\{\floor{\alpha}, \floor{2\alpha}, \floor{3\alpha}, \floor{4\alpha}, \dots \}$ and $\mathcal{S}_\beta=\{\floor{\beta}, \floor{2\beta}, \floor{3\beta}, \floor{4\beta}, \dots \}$, with the floor function $\floor{x}$. Here, we have $\frac{1}{\alpha}=\frac{\phi}{\sqrt{5}}=\frac{5+\sqrt{5}}{10}$ and $\frac{1}{\beta}=\frac{1}{\phi\sqrt{5}}=\frac{5-\sqrt{5}}{10}$, resulting in the Beatty sequences $\mathcal{S}_\alpha=\{1,2,4,5,6,8,9,11,12,13,15,16,\dots\}$ and $\mathcal{S}_\beta=\{3,7,10,14,18,21,25,28,32,36,39,\dots\}$. Taking the difference between each pair of consecutive values within each sequence generates two new sequences $\{1,2,1,1,2,1,2,1,1,2,1,1,\dots\}$ and $\{4,3,4,4,3,4,3,4,4,3,4,4,\dots\}$, respectively, both of which exhibit the Fibonacci word pattern---the subject of the following section.}
\begin{align*}
	\lim_{n\to\infty} \frac{F_{n}}{\phi^{n-1}}&=\frac{\phi}{\sqrt{5}}\\
	\lim_{n\to\infty} \frac{F_{n-1}}{\phi^{n}}&=\frac{1}{\phi \sqrt{5}}.
\end{align*}
\end{lemma}

\begin{proof}
Using Binet's formula (\S \ref{subsection: The Golden Ratio and Fibonacci Numbers}), we have
\begin{align*}
	\lim_{n\to\infty} \frac{F_{n}}{\phi^{n-1}}&=\lim_{n\to\infty} \frac{\phi^{n}-(-\phi)^{-n}}{\phi^{n-1}\sqrt{5}}\\
	&=\frac{\phi}{\sqrt{5}}
\end{align*}
and
\begin{align*}
	\lim_{n\to\infty} \frac{F_{n-1}}{\phi^{n}}&=\lim_{n\to\infty} \frac{\phi^{n-1}-(-\phi)^{-(n-1)}}{\phi^{n}\sqrt{5}}\\
	&=\frac{1}{\phi \sqrt{5}}.
\end{align*}
\end{proof}

\vspace{5mm}

This implies that the sum of the base lengths of the facets in the ``last" row---or more accurately, the $n$th row as $n$ approaches infinity---is equal to $\frac{\phi^2}{\sqrt{5}}$, while the sum of the gaps between these facets equals $\frac{1}{\sqrt{5}}$.

\vspace{5mm}

\begin{theorem}
The sum of the facet base lengths $s$ in the $n$th row and the gaps $g$ between these facets within the GD converge to the following values as $n$ approaches infinity
\begin{align*}
\lim_{n\to\infty} \frac{1}{2}\sum_{i=1}^{n} |R_{n}|s_{n} = \frac{\phi^2}{\sqrt{5}}\\
\lim_{n\to\infty} \frac{1}{2}\sum_{i=1}^{n} |R_{n-1}|g_{n} = \frac{1}{\sqrt{5}}.
\end{align*}
\end{theorem}

\begin{proof}
From Theorem \ref{theorem: fib-phi formulas} and Lemma \ref{lemma: convergent values},
\begin{align*}
\lim_{n\to\infty} \sum_{i=1}^{n} \frac{1}{\phi^{i}} &= \lim_{n\to\infty} \frac{|R_{n}|}{2}s_n+\frac{|R_{n-1}|}{2}g_n\\
\phi &= \lim_{n\to\infty} \frac{F_{n+2}-1}{\phi^n}+\frac{F_{n+1}-1}{\phi^{n+1}}\\
&=\lim_{n\to\infty} \frac{F_{n}}{\phi^{n-2}}+\frac{F_{n-1}}{\phi^{n-1}}\\
&=\lim_{n\to\infty} \phi \left (\frac{F_{n}}{\phi^{n-1}}+\frac{F_{n-1}}{\phi^{n}} \right )\\
&=\frac{\phi^2}{\sqrt{5}}+\frac{1}{\sqrt{5}}.
\end{align*}
Therefore,
\begin{align*}
 \lim_{n\to\infty} \frac{|R_{n}|}{2}s_n = \frac{\phi^2}{\sqrt{5}} \ \ \ \ \ \text{and} \ \ \ \ \  \lim_{n\to\infty} \frac{|R_{n-1}|}{2}g_n = \frac{1}{\sqrt{5}}.
\end{align*}
\end{proof}

\subsection{Fibonacci Words}

A particularly interesting structure known as a Fibonacci word will be shown to relate to the golden diamond. Here we give a brief overview of the topic.

A Fibonacci word $w_n$ is a specific sequence of letters typically denoted by $A$'s and $B$'s or 0's and 1's, resulting from the concatenation of two protowords $w_1=A$ and $w_2=AB$ and of the subsequent words that follow \cite{monnerot2013fibonacci}.

\begin{definition}  \label{definition: Fibonacci words}
The $n$th Fibonacci word $w_n$ is a string defined recursively as
\begin{align*}
w_{n+2}=\{w_{n+1}w_{n} \ | \ n \in \mathbbm{Z}^+ : w_1=A, \ w_2=AB\}
\end{align*}
where adjacency indicates concatenation.

\vspace{5mm}

The first few Fibonacci words are as follows
\begin{align}
w_1&=A \nonumber \\
w_2&=AB \nonumber \\
w_3&=ABA \nonumber \\
w_4&=ABAAB \nonumber \\
w_5&=ABAABABA \nonumber \\
w_6&=ABAABABAABAAB \nonumber \\
w_7&=ABAABABAABAABABAABABA \nonumber \\
w_8&=ABAABABAABAABABAABABAABAABABAABAAB \nonumber \\
\vdots \nonumber \\
w_\infty&=ABAABABAABAABABAABABAABAABABAABAABABAABABAABAABABAABABAABAABA...\nonumber
\end{align}
\end{definition}

\vspace{5mm}

\begin{definition} \label{def: inf fib word}
The word $w_\infty$ is the infinite Fibonacci word. \cite[A005614]{sloane2003line}
\end{definition}

\vspace{5mm}

\begin{theorem} [Fibonacci Word Substitution Theorem] \label{theorem: fib word substitution rule}
Given any Fibonacci word $w_n$, its successor $w_{n+1}$ can be generated by the substitution morphism
	\begin{align*}
		A&\to AB \\
		B&\to A.
	\end{align*}
\end{theorem}

\vspace{5mm}

The infinite Fibonacci word is often referred to as the prototypical example of a Sturmian word, named after Jacques Charles Fran\c{c}ois Sturm \cite{cassaigne2008extremal}. Sturmian words can be defined by a cutting sequence, in which a line of positive irrational slope generates an infinite sequence via alternating intersections with the integer grid lines in two dimensions (Figure \ref{fig: Cutting Sequence}) (For more on Sturmian words, see \cite{lothaire2002algebraic}.)

\vspace{5mm}

\begin{theorem} [Fibonacci Cutting Sequence]
The cutting sequence for a line of slope $\frac{1}{\phi}$ generates the infinite Fibonacci word.
\end{theorem}

\vspace{5mm}

\begin{figure}[ht]
    \centering
    \includegraphics[width=0.4\textwidth]{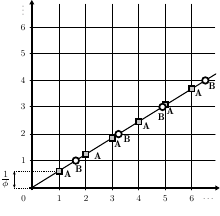}
    \caption{A line of slope $\frac{1}{\phi}$ generates the infinite Fibonacci word via the pattern in which it intersects horizontal and vertical integer grid lines. In general, such a pattern resulting from a line of an arbitrary irrational slope is known as a cutting sequence.}
    \label{fig: Cutting Sequence}
\end{figure}

\begin{theorem} [Fibonacci Word Letters] \label{theorem: fib word nth letter rule}
The $n$th letter in the infinite Fibonacci word is either an $A$ or $B$ according to the following rule
\begin{align*}
      A & \ \ \text{if} \ \ \lfloor(n+1)\phi \rfloor - \lfloor n\phi \rfloor - 1= 1\\
      B & \ \ \text{if} \ \ \lfloor(n+1)\phi \rfloor - \lfloor n\phi \rfloor - 1= 0 
\end{align*}
where $\lfloor x \rfloor$ is the floor function of x.
\end{theorem}

\vspace{5mm}

We will be working closely with the Fibonacci words, in addition to subwords of various type, defined as follows.

\begin{definition}
A Fibonacci $subword$ is any consecutive string of letters contained within the infinite Fibonacci word.
\end{definition}

\begin{definition} \label{def: initial fib subword}
An $initial$ $Fibonacci$ $subword$ is a Fibonacci subword obtained by truncating a proper Fibonacci word.
\end{definition}

\begin{definition} \label{def: palindromic fib subword}
A Fibonacci subword is $palindromic$ if it reads the same forward and backward.
\end{definition}

\vspace{5mm}

The Fibonacci words are a host to many interesting properties, such as the following examples.

\vspace{5mm}

\begin{lemma} \label{lemma: fib word letters}
The ratio of $A$'s to $B$'s in a word $w_n$ is $\frac{F_n}{F_{n-1}}$ where $F_n$ is the $n$th Fibonacci number. Therefore, the ratio of $A$'s to $B$'s in the infinite Fibonacci word approaches the golden ratio $\phi$.
\end{lemma}

\begin{proof}
The property follows naturally from the fact that $w_1$ has one $A$ and $w_2$ has both one $A$ and one $B$. Upon concatenating the words to form $w_3=ABA$, the $A$'s and $B$'s have increased to the next Fibonacci numbers in quantity, $F_3=2$ and $F_2=1$, respectively. In $w_4=ABAAB$, the $A$'s have increased to $F_4=3$ and the $B$'s to $F_3=2$. The pattern continues in this manner.
\end{proof}

\vspace{5mm}

The pattern in the Fibonacci word is not periodic but does have quasiperiods \cite{mignosi1992repetitions}.  The word has four distinct subwords of length three: $AAB$, $ABA$, $BAA$ and $BAB$, and subwords containing $BB$ or $AAA$ never occur. Additionally, any subword of finite length occurs infinitely often in the Fibonacci word \cite{mignosi1992repetitions}. The concatenation of two successive Fibonacci words is ``almost commutative" in the sense that $w_{n+1}=w_{n}w_{n-1}$ and $w_{n-1}w_{n}$ differ only by their last two letters \cite{pirillo1997fibonacci}. The number $0.010010100...$, whose decimals are built with the digits of the Fibonacci word (with $w_1=0$ and $w_2=01$), is transcendental. Many more properties are known (See \cite{mignosi1992repetitions}\cite{DELUCA1995307}\cite{pirillo1997fibonacci}).

\subsection{Fibonacci Words in the Golden Diamond} \label{subsection: Fib words in GD}

The Fibonacci word pattern has been used to create a series of curves known as the Fibonacci word fractals \cite{monnerot2013fibonacci} via drawing instructions defined by the infinite word. Unlike the Fibonacci word fractals, which were designed, here we will see that the Fibonacci words appear \textit{naturally} within the golden diamond fractal.

The occurrence of facets $t_k$ and $t_{k+1}$ within their respective rows $\text{R}_k$ and $\text{R}_{k+1}$ alternates in the Fibonacci word pattern, such that any two consecutive GD rows together are isomorphic to initial palindromic Fibonacci subwords (Figure \ref{fig: Fibonacci word GD}).

\vspace{5mm}

\begin{figure}[!b]
    \centering
    \includegraphics[width=\textwidth]{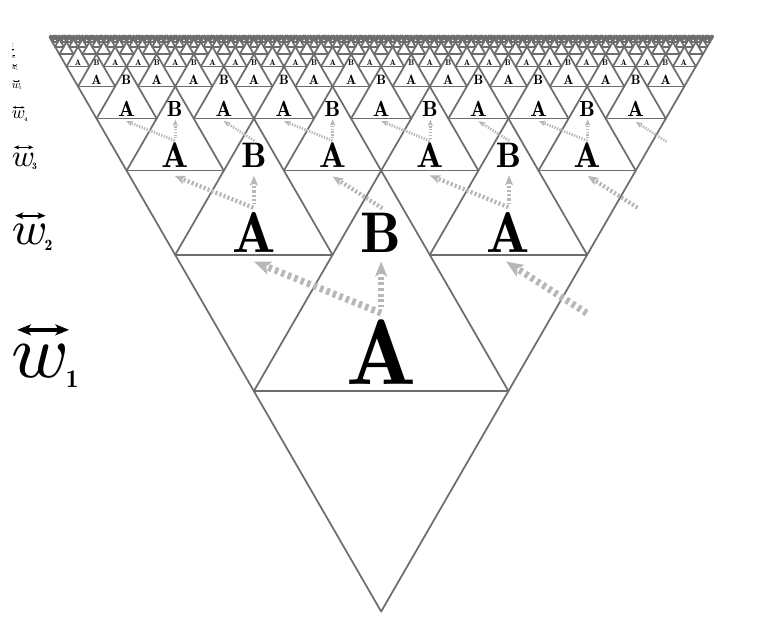}
    \caption{The union of any two consecutive rows in the golden diamond is isomorphic to an initial palindromic Fibonacci subword. The dashed arrows indicate the substitution rule: A$\to$AB, B$\to$A.}
    \label{fig: Fibonacci word GD}
\end{figure}

\begin{figure}[!p]
    \centering
    \includegraphics[width=\textwidth]{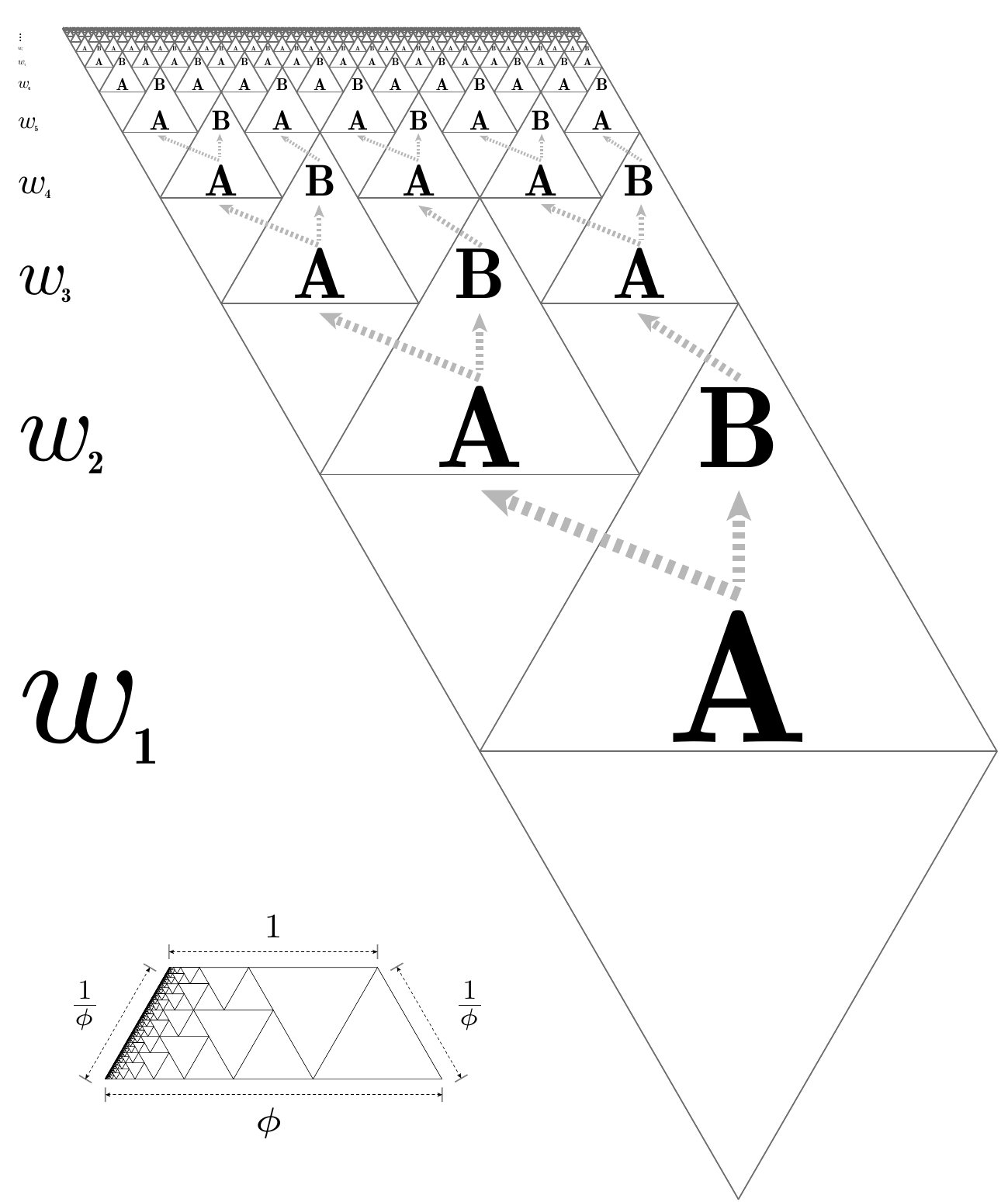}
    \caption{The golden trapezoid (GT): This trapezoid contains all the Fibonacci words within its rows of triangles. The golden diamond---shown here as a subset of the trapezoid---contains an infinite series of initial palindromic Fibonacci subwords.}
    \label{fig: Fibonacci Trapezoid}
\end{figure}

\begin{theorem} \label{theorem: palindromic Fibonacci subwords GD}
The union of any two consecutive rows in the golden diamond is isomorphic to an initial palindromic Fibonacci subword.
\end{theorem}

\begin{proof}
Let any pair of facets in a row of the GD represent the letter A and the apex of such a pair---or equivalently, the ``gap" created in the successive row by such a pair---represent the letter B. As indicated in Figure \ref{fig: Fibonacci word GD} by dashed arrows, the letters in each successive row can be defined by the Fibonacci word substitution rule, A$\to$AB, B$\to$A (Theorem \ref{theorem: fib word substitution rule}). The pattern is observed to continue by induction. The rule for generating the right-most facets is explained in the following theorem. As the GD is symmetric about its vertical axis, the resulting sequence of letters in each row are isomorphic to initial palindromic Fibonacci subwords (Definitions \ref{def: initial fib subword} and \ref{def: palindromic fib subword}).
\end{proof}

\vspace{5mm}

If we wish to recover all the Fibonacci words in the fractal, we can augment the golden diamond by augmenting the right end of each row with an additional pair of facets, spaced properly to account for the preceding row; or equivalently, we can simply remove one of the two largest GD self-similar proper subsets to obtain the same image. The resulting geometry is dubbed the golden trapezoid (GT) and is shown in Figure \ref{fig: Fibonacci Trapezoid}.

\vspace{5mm}

\begin{theorem} \label{theorem: Fibonacci words in golden trapezoid}
The union of any two consecutive rows in the golden trapezoid is isomorphic to a Fibonacci word.
\end{theorem}

\begin{proof}
The proof follows the same argument for Theorem \ref{theorem: palindromic Fibonacci subwords GD}. As was indicated in Figure \ref{fig: Fibonacci word GD} by dashed arrows, we can geometrically induce the Fibonacci word substitution rule (Theorem \ref{theorem: fib word substitution rule}), generating successive row patterns, isomorphic to the Fibonacci words.
\end{proof}

\vspace{5mm}

It is easily observed and well-known \cite{monnerot2013fibonacci} that truncating a Fibonacci word by its final two letters results in a palindromic subword; although, the proof of this fact is not particularly obvious. Here, however, the proof presents itself geometrically.

\vspace{5mm}

\begin{theorem} \label{theorem: initial palindromic Fibonacci subwords in GD}
The truncating of a proper Fibonacci word $w_n$ for $n>2$ by its last two letters produces an initial palindromic Fibonacci subword $\overset\leftrightarrow{w}_n$.
\end{theorem}

\begin{proof}
The golden diamond, which contains the set of initial palindromic Fibonacci subwords, is a proper subset of the golden trapezoid and is obtained by removing the final pair of facets in each row of the GT---and consequently the ``gap", denoted by a letter B, found near the end of each row. Therefore, a row of the GD defined as a subset of a GT row is isomorphic to an initial palindromic Fibonacci subword defined by the removal of the final two letters in the associated proper Fibonacci word.

This provides the proof for Lemma \ref{lemma: R_n}, which stated that $F_{n+2}-1$ pairs of facets are found in each row of the GD. As the $n$th row of the GT contains $F_{n}$ pairs of facets---or letter A's by Lemma \ref{lemma: fib word letters}---the $n$th row of the GD contains $F_{n+2}-1$ pairs of facets, due to the offset of rows in the GD with respect to the GT, where row $n$ of the GD corresponds to row $n+2$ of the GT.
\end{proof}

\vspace{5mm}

Other such proofs for patterns in Fibonacci words can be easily obtained from the geometric properties of the GD and GT. Decomposing any given Fibonacci word into palindromic subwords is easily done by observing intersections of self-similar GD subsets through the row of the GT associated with the Fibonacci word of interest. As each self-similar proper GD subset contains palindromic subwords, the corresponding rows that intersect a given GT row will result in a palindromic subword within a proper Fibonacci word.

\section{Perspective Projection} \label{section: perspective projection}
A remarkable property of the golden diamond is that it is isomorphic to a perspective projection of a field of infinitely many vertical triangles in the plane. By assuming the top two vertices of the fractal as vanishing points, we can connect perspective lines to the base vertices of the triangular facets and illustrate this perspective space (Figure \ref{fig: golden diamond Perspective}). In this manner, we create a grid-like plane upon which a field of identical triangles stand together in unending rows. For illustrative purposes, we can extrapolate from this geometry and expand our triangles into square-based pyramids (Figure \ref{fig: golden diamond Perspective - Pyramids}). Furthermore, we could represent a field of infinitely many octahedra, standing on point---shapes that, from the viewer's perspective, never overlap nor reveal the gaps between---simply a view of infinitely many polyhedra and nothing more (Figure \ref{fig: golden diamond Perspective - Octahedra}). 

\begin{figure}[!ht]
     \centering
    \begin{subfigure}[b]{0.29\textwidth}
        \includegraphics[width=\textwidth]{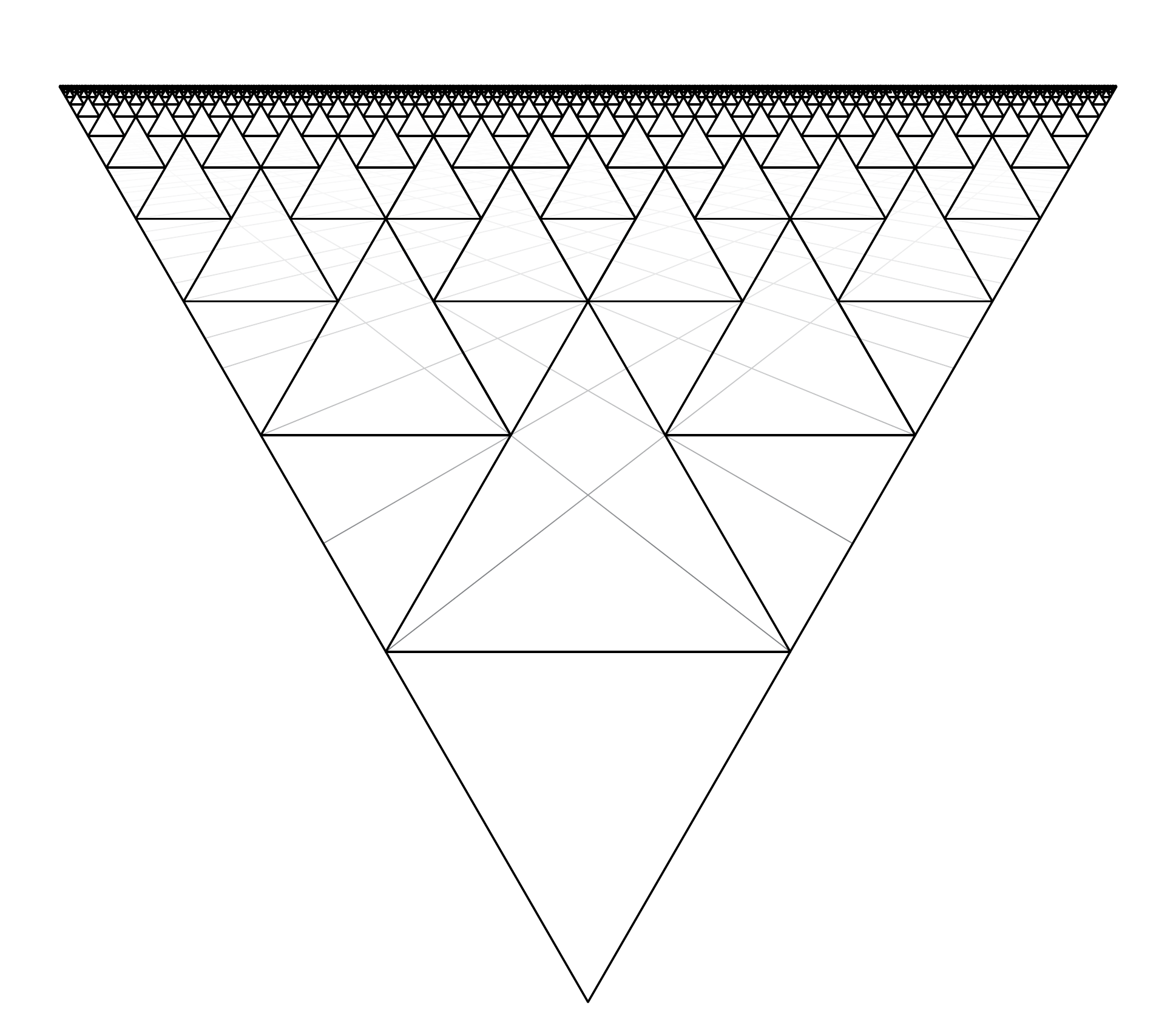}
        \caption{}
        \label{fig: golden diamond Perspective}
    \end{subfigure}
   %\par
     \begin{subfigure}[b]{0.29\textwidth}
        \includegraphics[width=\textwidth]{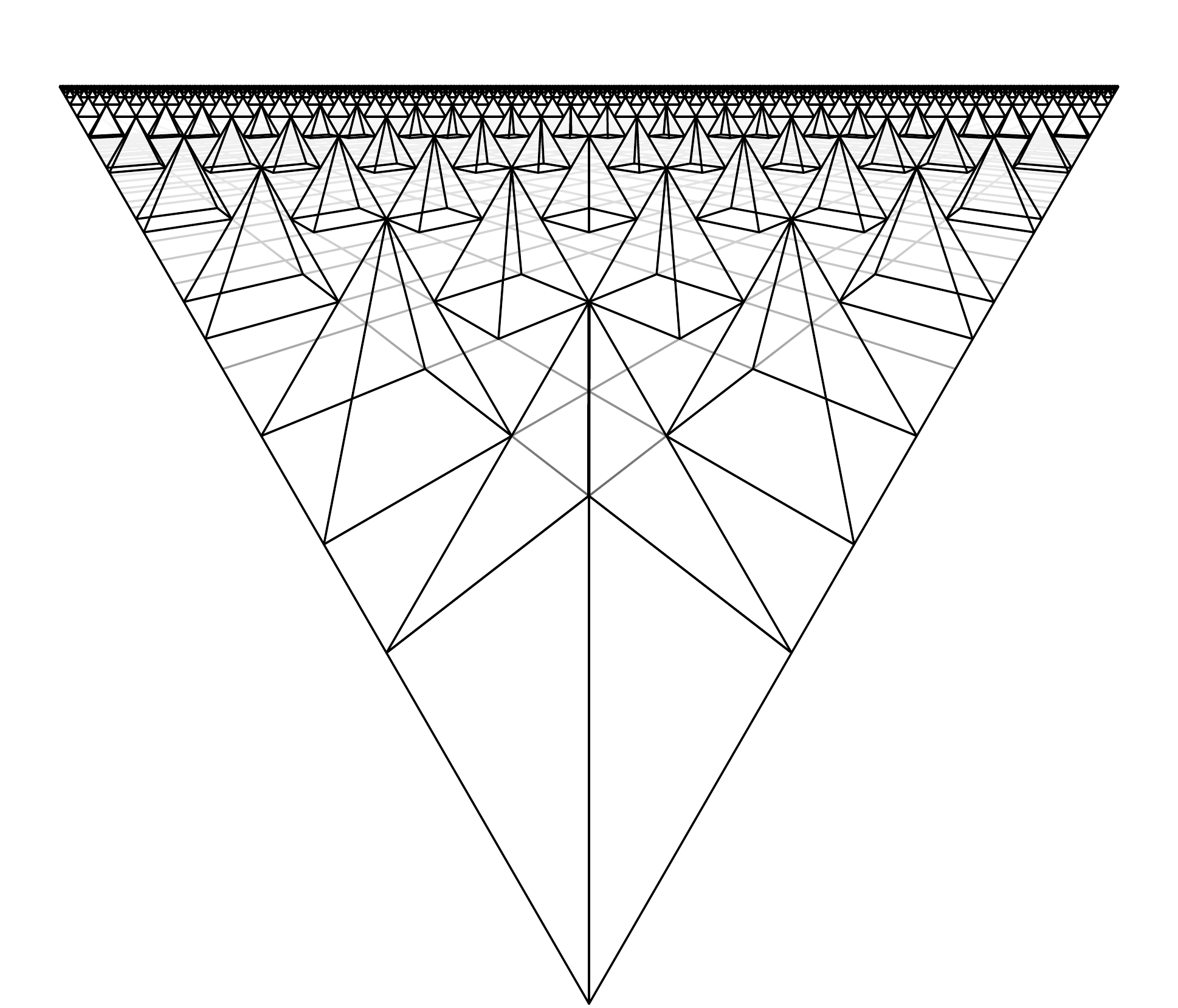}
        \caption{}
        \label{fig: golden diamond Perspective - Pyramids}
    \end{subfigure}
    %\par
    \begin{subfigure}[b]{0.29\textwidth}
        \includegraphics[width=\textwidth]{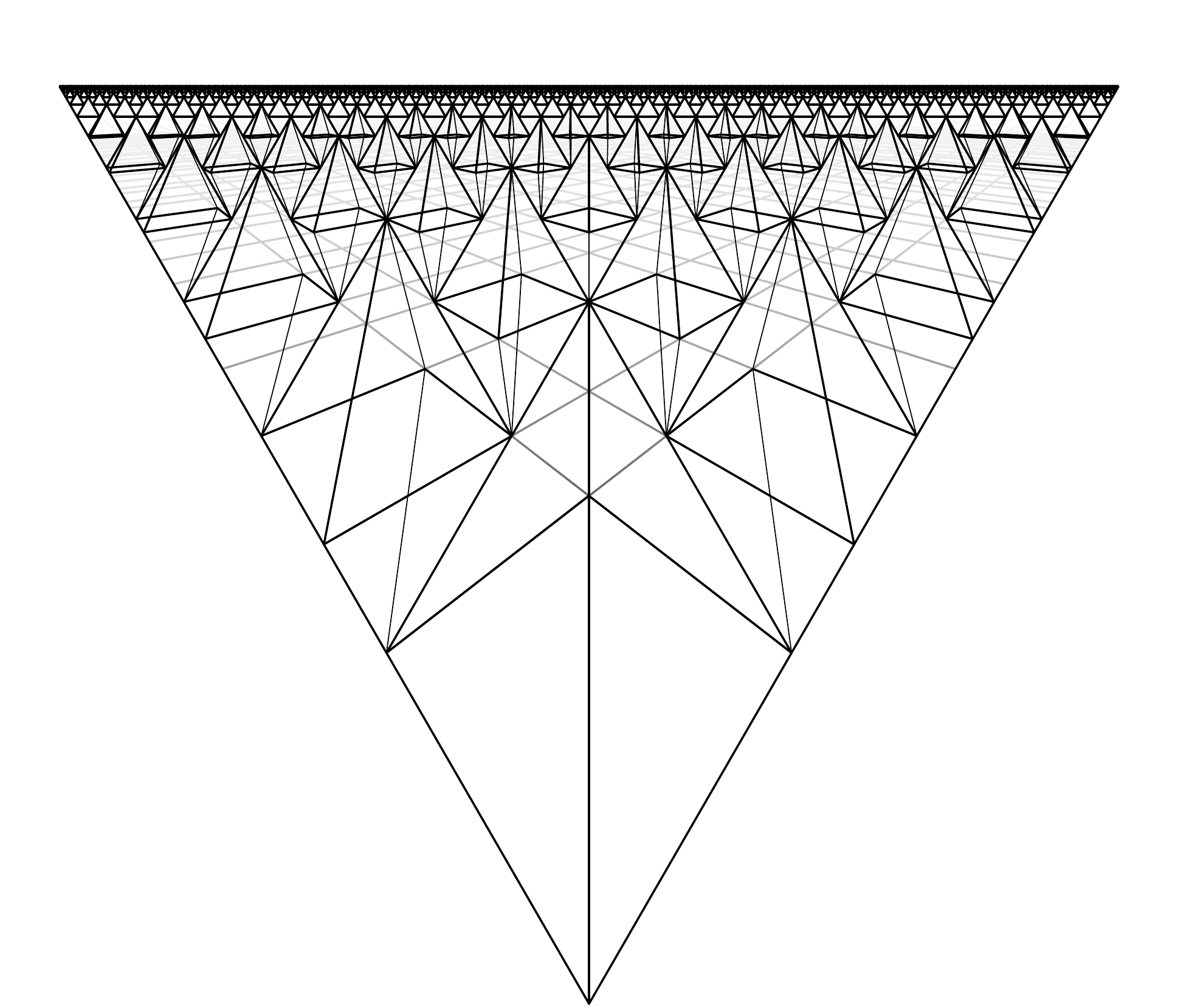}
        \caption{}
        \label{fig: golden diamond Perspective - Octahedra}
    \end{subfigure}

    %%%%%%
    \begin{subfigure}[b]{\textwidth}
   	\includegraphics[width=\textwidth]{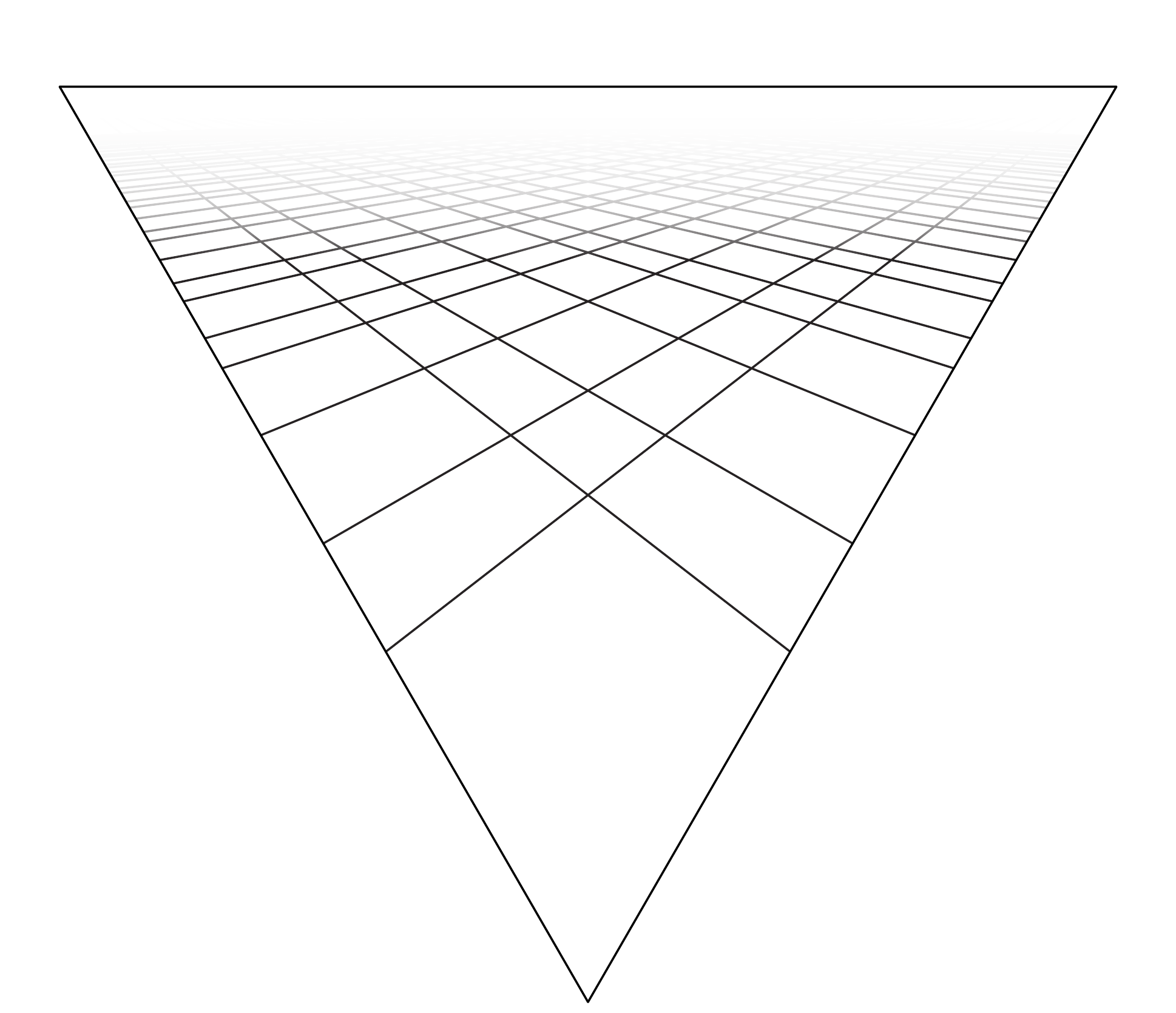}
	\caption{}
        \label{fig: Phinary Domain Perspective}
    \end{subfigure}
    \caption{(a) The golden diamond with ``perspective lines" connecting base vertices of facets to vanishing points. (b) Extending a line down from each triangle apex to adjacent grid intersections produces the image of a field, in which infinite non-overlapping pyramids extend toward the horizon. (c) Illustrated with octahedra. (d) The perspective grid lines.}
    \label{fig: Phinary perspective figures}
\end{figure}

At this stage, it is worth reflecting on the peculiarity of an abstract geometry like the golden diamond---natural in its discovery---being capable of depicting a perspective projection of some other abstract space. What is it actually showing? We will now enter the ``window" offered by this fractal, studying the perspective space and the unusual features associated with it. An illustration of the co-image is shown in Figure \ref{fig: Phinary Domain Perspective}.

\subsection{The Phinary Domain} \label{subsection: phinary domain}

Here we are interested in determining the grid upon which the pyramids/octahedra of Figures  \ref{fig: golden diamond Perspective - Pyramids}/ \ref{fig: golden diamond Perspective - Octahedra} sit. For reasons that will become clear, we will refer to this space as the phinary domain. By inspection, we recognize that the grid is composed of two different interval lengths; one interval is associated with the base side-length of a pyramid and the other interval spans the distance between two neighboring pyramids. We want to find the proportionality of those two intervals and quantify the pattern in which they occur. To determine the ratio between the two values, we can employ an invariant quantity associated with perspective projections, known as the cross-ratio \cite[pp. 163-164]{schneider2002geometric}. It is by this method that we prove the following.

\vspace{5mm}

\begin{figure}[ht]
   \centering
   \includegraphics[width=\textwidth]{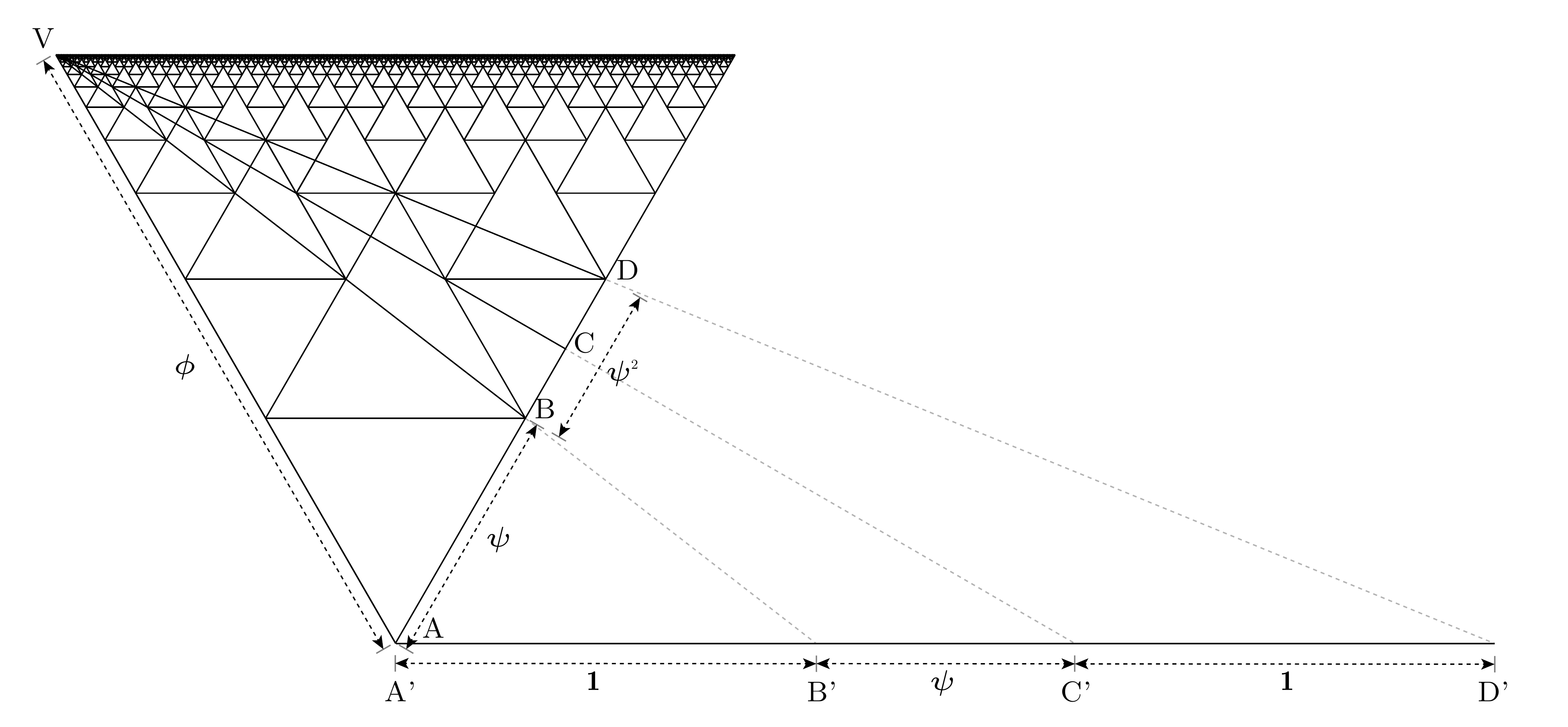}
    \caption{Projective rays VA, VB, VC, and VD produce the first three projected interval markings, AB, BC, and CD. Through the use of the cross-ratio, we can determine the ``true" interval values, A'B', B'C', and B'D', as they occur in the phinary domain.}
    \label{fig: Projective Cross Ratio}
\end{figure}

\begin{lemma} \label{lemma: phinary cross ratio}
The gridlines of the phinary domain are spaced out in nonuniform intervals alternating between two lengths that differ by a factor of the golden ratio, in a pattern defined by the Fibonacci word.
\end{lemma}

\begin{proof}
Given four collinear points A, B, C, and D, and another set of independently collinear points A', B', C', and D'---with the requirement that lines connecting dual points, i.e. lines through AA', BB', CC', and DD', intersect at a mutual locus---we may define an invariant quantity known as the cross-ratio, such that

\begin{align}
	\frac{\text{AC} \cdot \text{BD}}{\text{BC} \cdot \text{AD}} = \frac{\text{A'C'} \cdot \text{B'D'}}{\text{B'C'} \cdot \text{A'D'}}. \label{eq: cross-ratio}
\end{align}

As shown in Figure \ref{fig: Projective Cross Ratio}, we consider the rays passing from points A, B, C, and D to their vanishing point V. The lengths AB and CD correspond to the projected edge lengths of two neighboring pyramids in Figure \ref{fig: golden diamond Perspective - Pyramids}; the value of BC defines the projected distance between the two pyramids. We wish to use these intervals to determine the ``actual" distances between points in the phinary domain. We have defined a side of the golden diamond, VA, to equal the golden ratio, $\phi$, such that

\begin{align*}
	\text{AB} = \psi, \ \text{BD} = \psi^2,  \ \text{and} \ \text{BC} = \text{CD} = \frac{\psi^2}{2}
\end{align*}

where $\psi=\frac{1}{\phi}$. Therefore,

\begin{align*}
	\text{AC} = \frac{\phi}{2} \ \text{and} \ \text{AD} = 1.
\end{align*}
All pyramids in the phinary domain have equal dimensions. Therefore, if we define the dual points A', B', C', and D', we may say that

\begin{align*}
	\text{A'B'} = \text{C'D'} \ \text{and} \ \text{A'C'} = \text{B'D'}.
\end{align*}
We are interested in the ratio $\frac{\text{A'B'}}{\text{B'C'}}$. To find it, we can set up the cross-ratio as defined by Equation \ref{eq: cross-ratio}, such that

\begin{align}
	\frac{\frac{\phi}{2} \cdot \psi^2}{\frac{\psi^2}{2} \cdot 1} &= \frac{(\text{A'C'})^2}{\text{B'C'} \cdot \text{A'D'}} \nonumber \\
	\phi &= \frac{(\text{A'C'})^2}{\text{B'C'} \cdot \text{A'D'}} \nonumber \\
	\text{A'D'} &= \psi \frac{(\text{A'C'})^2}{\text{B'C'}}. \label{eq: A'D' 1} 
\end{align}
Furthermore,
\begin{align} 
	\text{A'D'} &= \psi \frac{(\text{A'B'+B'C'})^2}{\text{B'C'}} \nonumber \\
	 &= \psi \left (\frac{(\text{A'B'})^2}{\text{B'C'}} + 2 \text{A'B'} + \text{B'C'} \right) \nonumber \\
	\text{A'D'} &= \phi \frac{(\text{A'B'})^2}{\text{B'C'}}. \label{eq: A'D' 2}
\end{align}
Setting Equations \ref{eq: A'D' 1} and \ref{eq: A'D' 2} equal yields

\begin{align*}
	\psi \frac{(\text{A'C'})^2}{\text{B'C'}} &= \phi \frac{(\text{A'B'})^2}{\text{B'C'}} \\
	\phi \text{A'B'} &= \text{A'C'} \\
	&= \text{A'B'} + \text{B'C'}.
\end{align*}
And therefore,

\begin{align*}
	\frac{\text{A'B'}}{\text{B'C'}} &= \phi.
\end{align*}

Figure \ref{fig: Projective Cross Ratio} shows that we may conveniently project VA, VB, VC, and VD to a horizontal line and recover the same relationship we have just defined by labelling the intersections A', B', C' and D'.

Finally, we wish to consider the sequential pattern by which these intervals progress toward the ``horizon". We may extend our previous method by defining the infinite set of points A, B, C, D, E, F, ... that are generated in exactly the same manner as before, by projecting rays from vanishing point V through the base vertices of all pyramids. Here we note that each ray must pass through a corresponding base vertex of a facet in the ``back" row of the golden diamond. As the intervals between facet base vertices in any given row is defined by the Fibonacci word pattern (Theorem \ref{theorem: palindromic Fibonacci subwords GD}), by induction, the intervals in the phinary domain are sequentially arranged in the very same pattern.
\end{proof}

\vspace{5mm}

If we assume the GD has side lengths equal to $\phi$, we can define the grid intervals in the phinary domain to have values of 1 and $\frac{1}{\phi}$, as shown in Figure \ref{fig: Projective Cross Ratio}. In this way, we can assign a numerical value to each of the grid lines (Figure \ref{fig: Phi grid numbers}).

\begin{figure}[ht]
   \centering
   \includegraphics[width=0.4\textwidth]{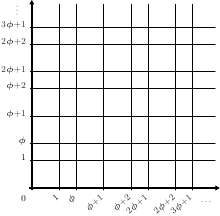}
    \caption{Numerical values are assigned to each grid line in the phinary domain with respect to its distance from the origin.}
    \label{fig: Phi grid numbers}
\end{figure}

These values come, of course, from progressively adding the letters in the infinite Fibonacci word. If we let $\psi=\frac{1}{\phi}$ and use the letters 1 and $\psi$ to spell the infinite Fibonacci word, we have
\begin{align*}
1\psi11\psi1\psi11\psi11\psi1\psi11\psi1\psi11\psi11\psi1\psi11\psi11\psi1\psi11\psi1\psi11\psi11\psi1\psi11\psi1\psi1 \dots
\end{align*}
Evaluating the sum at each successive letter yields
\begin{align*}
1&=1\\
\phi&=1+\psi\\
\phi+1&=1+\psi+1\\
\phi+2&=1+\psi+1+1\\
2\phi+1&=1+\psi+1+1+\psi\\
2\phi+2&=1+\psi+1+1+\psi+1\\
3\phi+1&=1+\psi+1+1+\psi+1+\psi\\
&\vdots
\end{align*}

\subsection{Phinary Numbers} \label{subsection: phinary numbers}

By continuing to define the numbers at each grid line, we establish an infinite set of ordered values that we will refer to as the $phinary$ $numbers$. We will denote the set of phinary numbers by the symbol $\mathbbm{Z}^+_\Phi$, alluding to a relationship with the set of positive integers $\mathbbm{Z}^+$. Each value can be written in the form $a\phi+b$ for $a,b \in \mathbbm{Z}^+$. The first few values are as follows:

\vspace{5mm}

\begin{center}
\centering
\begin{tabular}{r r r r r r r r r r r r r r r r r r}
  \multicolumn{1}{c}{$\mathbbm{Z}^+$}& \multicolumn{1}{c}{$\mathbbm{Z}^+_\Phi$}& 	&	&\multicolumn{1}{c}{$\mathbbm{Z}^+$} & \multicolumn{1}{c}{$\mathbbm{Z}^+_\Phi$}& 		&& \multicolumn{1}{c}{$\mathbbm{Z}^+$} & \multicolumn{1}{c}{$\mathbbm{Z}^+_\Phi$} & 		&\multicolumn{1}{c}{$\mathbbm{Z}^+$} & \multicolumn{1}{c}{$\mathbbm{Z}^+_\Phi$} & 		&& \multicolumn{1}{c}{$\mathbbm{Z}^+$} & \multicolumn{1}{c}{$\mathbbm{Z}^+_\Phi$}& \\
  \hline
  1 	& 1 			& 			&& 12 & $4\phi+4$ 	&			&& 23 & $9\phi+5$& 			& 34 & $13\phi+8$ & 		 	&& 45 & $17\phi+11$\\
  2 	& $\phi$ 		& 			&& 13 & $5\phi+3$ 	& 			&& 24 & $9\phi+6$&			& 35 & $13\phi+9$ &				&& 46 & $17\phi+12$\\
  3 	& $\phi+1$	&			&& 14 & $5\phi+4$ 	& 			&& 25 & $9\phi+7$&			& 36 & $14\phi+8$ &				&& 47 & $18\phi+11$\\
  4 	& $\phi+2$ 	& 			&& 15 & $6\phi+3$ 	&			&& 26 & $10\phi+6$& 		& 37 & $14\phi+9$ &				&& 48 & $18\phi+12$\\
  5 	& $2\phi+1$	&		 	&& 16 & $6\phi+4$ 	& 			&& 27 & $10\phi+7$&		& 38 & $14\phi+10$ &			&& 49 & $19\phi+11$\\
  6 	& $2\phi+2$	& 			&& 17 & $6\phi+5$ 	&			&& 28 & $11\phi+6$&		& 39 & $15\phi+9$ &				&& 50 & $19\phi+12$\\
  7 	& $3\phi+1$ 	& 			&& 18 & $7\phi+4$ 	&			&& 29 & $11\phi+7$&		& 40 & $15\phi+10$ &			&& 51 & $19\phi+13$\\
  8 	& $3\phi+2$ 	& 			&& 19 & $7\phi+5$ 	& 			&& 30 & $11\phi+8$&		& 41 & $16\phi+9$ &				&& 52 & $20\phi+12$\\
  9 	& $3\phi+3$ 	& 			&& 20 & $8\phi+4$ 	&			&& 31 & $12\phi+7$&		& 42 & $16\phi+10$ &			&& 53 & $20\phi+13$\\
  10 	& $4\phi+2$ 	& 			&& 21 & $8\phi+5$ 	& 		 	&& 32 & $12\phi+8$& 		& 43 & $16\phi+11$ & 			&& 54 & $21\phi+12$\\
  11 	& $4\phi+3$ 	& 			&& 22 & $8\phi+6$ 	& 			&& 33 & $12\phi+9$&		& 44 & $17\phi+10$ &			&& 55 & $21\phi+13$\\ \\
\end{tabular}
\end{center}

The phinary numbers correspond to the positive integer values of the base-phi positional numbering system, also known as the golden ratio base or phinary, which was first introduced by George Bergman in 1957\footnote{Bergman was a twelve year old junior high student at the time of publication.} \cite{10.2307/3029218}. Base-phi employs an irrational radix by means of the golden ratio, such that any real number can be defined as a sum of powers of phi. For example, the first five positive integers can be written in the following manner:

\begin{center}
\centering
\begin{tabular}{r r r}
  \multicolumn{1}{c}{base-10}	&	\multicolumn{1}{c}{base-$\phi$}		&	\multicolumn{1}{c}{powers of $\phi$}\\
  \hline
  1	&	$1_\Phi$			&	$\phi^0$	\\
  2	&	$10.01_\Phi$		&	$\phi^1+\phi^{-2}$	\\
  3	&	$100.01_\Phi$		&	$\phi^2+\phi^{-2}$	\\
  4	&	$101.01_\Phi$		&	$\phi^2+\phi^0+\phi^{-2}$	\\
  5	&	$1001.1001_\Phi$	&	$\phi^3+\phi^0+\phi^{-1}+\phi^{-4}$	
\end{tabular}
\end{center}

Rational numbers take on repeating decimal values \cite{eggan1966decimal}. Below are some notable rational and irrational numbers in the golden ratio base.

\begin{center}
\centering
\begin{tabular}{r r r}
  \multicolumn{1}{c}{$n$}	&	\multicolumn{1}{c}{base-$\phi$}	\\
  \hline
  $\frac{1}{2}$	&	$0.01001001001001001001001001\dots_\Phi$\\
  $\frac{1}{3}$	&	$0.00101000001010000010100000\dots_\Phi$\\
  $\pi$		&	$100.010010101001000101010100\dots_\Phi$	\\  
  $e$			&	$100.000010000100100000000100\dots_\Phi$	\\
  $\sqrt{2}$	&	$1.01000001010010100100000001\dots_\Phi$	\\
  $\sqrt{5}$	&	$10.1_\Phi$
\end{tabular}
\end{center}

Our interest, however, is with the phinary numbers---the values occurring to the left of the radix point. Naturally, we can extend our scope to incorporate negative values, thereby defining the set of \textit{phinary integers} $\mathbbm{Z}_\Phi$. These values appear to have avoided a thorough investigation in the literature and will constitute the topic for the majority of this paper. Below are the first few phinary numbers, expanded into powers of the golden ratio.

\begin{center}
\centering
\begin{tabular}{r r r}
  \multicolumn{1}{c}{$\mathbbm{Z}^+_\Phi$}	&	\multicolumn{1}{c}{base-$\phi$}		&	\multicolumn{1}{c}{powers of $\phi$}\\
  \hline
  1			&	$1_\Phi$			&	$\phi^0$\\
  $\phi$		&	$10_\Phi$			&	$\phi^1$\\
  $\phi+1$		&	$100_\Phi$		&	$\phi^2$\\
  $\phi+2$		&	$101_\Phi$		&	$\phi^2+\phi^0$	\\
  $2\phi+1$	&	$1000_\Phi$		&	$\phi^3$\\
  $2\phi+2$	&	$1001_\Phi$		&	$\phi^3+\phi^0$\\
  $3\phi+1$	&	$1010_\Phi$		&	$\phi^3+\phi^1$	\\
  $3\phi+2$	&	$10000_\Phi$		&	$\phi^4$
\end{tabular}
\end{center}

Although similar to binary, notice that no consecutive values of 1 will appear in the base-phi representation of a phinary number. This is called the \textit{standard form} for a base-phi number, as alternative representations are possible \cite{10.2307/3029218}. For example, $3\phi+2$ equals $\phi^4$, $\phi^3+\phi^2$, and $\phi^3+\phi^1+\phi^0$ and therefore can be written either as $10000_\Phi$, $1100_\Phi$, or $1011_\Phi$. However, as it turns out, the omitting of consecutive powers of the golden ratio, as is enforced by the standard form, results in a single, unique representation for each phinary number---a property that is true for representing the real numbers, as well \cite{10.2307/3029218}. This provides a convenient means for writing out an ordered list of the phinary numbers: proceed as if writing binary values, but skip any values which contain consecutive 1's.

There are several important properties to note about the phinary numbers. To begin with, each value can be derived from the Fibonacci words.

\vspace{5mm}

\begin{theorem} \label{theorem: phi number initial fib subword}
Every phinary number is equal to the sum of letters in an initial Fibonacci subword with letters $1$ and $\psi$.
\end{theorem}

\begin{proof}
This is, of course, how we came to define the phinary numbers.
\end{proof}

\vspace{5mm}

\begin{lemma}
The phinary numbers $\mathbbm{Z}^+_\Phi$ form a set of ordinals.
\end{lemma}

\begin{proof}
The phinary numbers are well-ordered, which is easily seen, as they form a totally ordered set $(\mathbbm{Z}^+_\Phi, \leq)$. That is, they are transitive, i.e. for any elements $p, q, r \in \mathbbm{Z}^+_\Phi$, if $p > q$ and $q > r$, then $p > r$; they are well-founded, i.e. every nonempty subset has a least element; and trichotomy holds on any two phinary numbers $p$ and $q$, such that exactly one of the statements, $p<q$, $q<p$, $p=q$, is true.
\end{proof}

\vspace{5mm}

\begin{theorem} \label{theorem: phi power fib word}
A power of the golden ratio, $\phi^n$ for $n \geq 0$, is equal to the sum of letters in a proper Fibonacci word $w_{n+1}$ with letters $1$ and $\psi$.
\end{theorem}

\begin{proof}
This follows from Lemma \ref{lemma: fib word letters}, which stated that a proper Fibonacci word $w_{n+1}$ was composed of $F_{n+1}$ instances of its first letter and $F_{n}$ of its other letter. With an alphabet of $1$ and $\psi$, this equates to 
\begin{align*}
F_{n+1}\cdot1+F_{n}\cdot\psi&=\psi(F_{n+1}\phi+F_{n})\\
&=\psi \phi^{n+1}\\
&= \phi^n
\end{align*}
by Equation \ref{eq: fib phi linear relation} in \S \ref{subsection: The Golden Ratio and Fibonacci Numbers}.
\end{proof}

\vspace{5mm}

Consequently, the $n$th phinary number equals a power of the golden ratio whenever $n$ is a Fibonacci number.

\vspace{5mm}

\begin{center}
\centering
\begin{tabular}{r r | r r}
  \multicolumn{1}{c}{$\text{F}_{k+2}$}	&	\multicolumn{1}{c|}{$n\in\mathbbm{Z}^+$}		&	\multicolumn{1}{c}{$\phi^k$}		& 	\multicolumn{1}{c}{$\mathbbm{Z}^+_\Phi$}	\\
  \hline
  $\text{F}_2$		&	1	&	$\phi^0$	&	1			\\
  $\text{F}_3$		&	2	&	$\phi^1$	&	$\phi$		\\
  $\text{F}_4$		&	3	&	$\phi^2$	&	$\phi+1$		\\
  $\text{F}_5$		&	5	&	$\phi^3$	&	$2\phi+1$		\\
  $\text{F}_6$		&	8	&	$\phi^4$	&	$3\phi+2$		\\
  $\text{F}_7$		&	13	&	$\phi^5$	&	$5\phi+3$		\\
  $\text{F}_8$		&	21	&	$\phi^6$	&	$8\phi+5$		\\
  $\text{F}_9$		&	34	&	$\phi^7$	&	$13\phi+8$	\\
  $\text{F}_{10}$	&	55	&	$\phi^8$	&	$21\phi+13$	\\
\end{tabular}
\end{center}

We find that the Fibonacci numbers relate to the positive integers in the same manner as the powers of the golden ratio relate to the phinary numbers. For example, if we consider a phinary analogue to the Fibonacci sequence, we recover the powers of phi. 

\vspace{5mm}

\begin{lemma} \label{lemma: phinary fib numbers}
Let T be a recursive sequence, such that
\begin{align*}
\text{T}_{n+2}= \left \{\text{T}_{n+1}+\text{T}_{n} \ | \ (\mathbbm{Z}^+_\Phi)^2 \to \mathbbm{Z}^+_\Phi : \ \text{T}_1=1, \ \text{T}_2=\phi, \ n \in \mathbbm{Z}^+ \right \}.
\end{align*}
Yielding
\begin{align*}
\text{T}=\{\phi^0, \phi^1, \phi^2, \phi^3, \phi^4, \phi^5, \phi^6, \phi^7, ...\}.
\end{align*}
\end{lemma}
\begin{proof}
This, of course, follows naturally from the recursive property of the golden ratio powers (\S \ref{subsection: The Golden Ratio and Fibonacci Numbers}, Equation \ref{eq: phi power recursive relation}). As $T_1=\phi^0=1$ and $T_2=\phi^1=\phi$, then $T_2=\phi^1+\phi^0=\phi^2$. And in general, $T_{n+2}=\phi^{n+1}+\phi^{n}=\phi^{n+2}$.
\end{proof}

\vspace{5mm}

By viewing the phinary numbers as a set of ordinal numbers, we observe that the powers of the golden ratio serve the role as the ``Fibonacci numbers" within this set.

\vspace{5mm}

\begin{theorem} \label{theorem: fib phi iso}
The subset of the phinary numbers containing the nonnegative integer powers of the golden ratio is isomorphic to the subset of the positive integers containing the unique Fibonacci numbers.
\begin{align*}
\{\mathbbm{Z}^+_\Phi, \text{F}_\Phi\} \cong \{\mathbbm{Z}^+, F\} 
\end{align*}
where $\displaystyle F=\bigcup_{i=2}^\infty F_i$ is the set of ordinary Fibonacci numbers, without $F_1$, and $\displaystyle \text{F}_\Phi=\bigcup_{i=0}^\infty \phi^i$ is the set of phinary Fibonacci numbers.
\end{theorem}

\begin{proof}
\begin{align*}
F= \left \{ \bigcup_{i=2}^\infty F_i \ | \ F_{n+2}=F_{n+1}+F_n, \  F_1=1, \ F_2=1 \right \}.
\end{align*}
By Lemma \ref{lemma: phinary fib numbers},
\begin{align*}
\text{F}_\Phi=\bigcup_{i=0}^\infty \phi^i = \left \{ \bigcup_{i=0}^\infty T_i \ | \ \text{T}_{n+1}=\text{T}_n+\text{T}_{n-1}, \  \text{T}_1=1, \ \text{T}_2=\phi \right \}.
\end{align*}
\end{proof}

\begin{remark}
Noteworthy is the fact that unlike the natural Fibonacci numbers $F$, whose ratio of consecutive values $\frac{F_n}{F_{n-1}}$ approaches $\phi$ only when $n$ goes to infinity, the ratio of any two consecutive phinary Fibonacci numbers always equals $\phi$.
\end{remark}

\subsection{The Phinary Domain Revisited}

Equipped with the phinary numbers, we will now revisit the phinary domain to define it properly and note some of its features.

\vspace{5mm}

\begin{definition}
The phinary domain is the Cartesian plane of nonnegative phinary integers, which can be written as $\mathbbm{Z}_\Phi^*\times \mathbbm{Z}_\Phi^*$.
\end{definition}

\vspace{5mm}

\begin{figure}[!p]
   \centering
   \includegraphics[width=\textwidth]{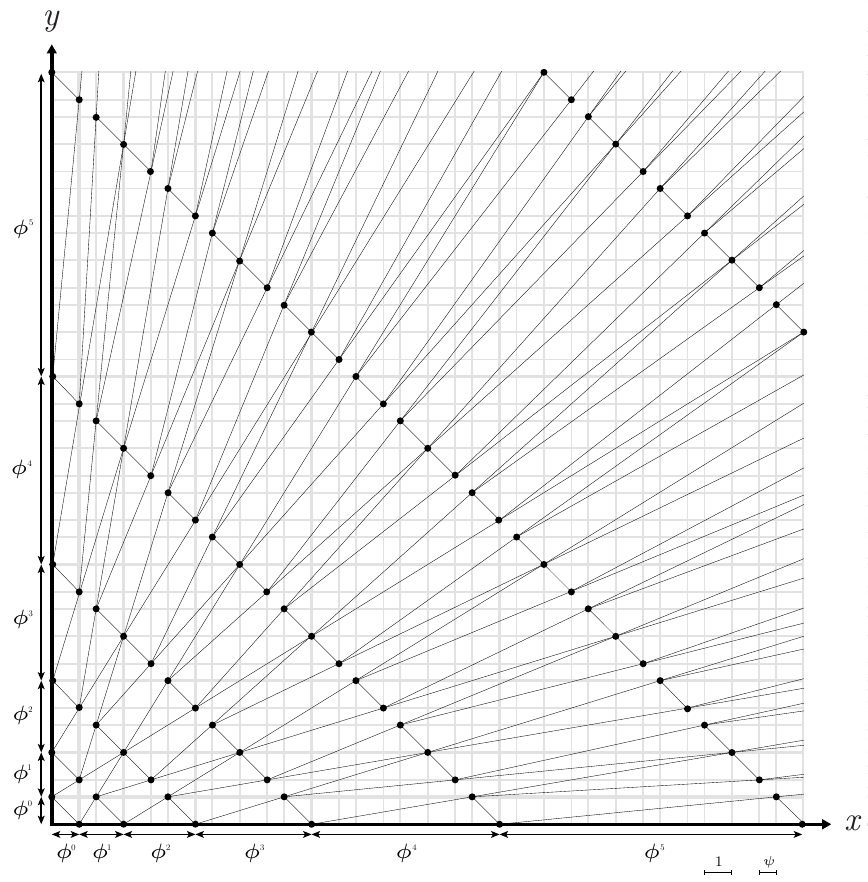}
    \caption{The golden diamond co-image in the phinary domain. The base vertices of the standing triangles are indicated by dark points, located at coordinates $(x,y)$, such that $x+y = \sum_{i=0}^{n}\phi^i$, where $n$ is determined by the diagonal within which the point is located; the diagonal closest to the origin represents $n=0$ and further diagonals increment for greater values of $n$. The lines indicate the ``shadow" of each standing triangle as projected from the perspective focal point.}
    \label{fig: Phi Grid - GD}
\end{figure}

\begin{lemma} \label{lemma: GD points}
The vertices of the standing triangles in the phinary domain have coordinates equal to $(x, y)$ where
\begin{align*}
x+y = \sum_{i=0}^{n}\phi^i 
\end{align*}
for positive integer $n$.
\end{lemma}

\begin{proof}
See Figure \ref{fig: Phi Grid - GD}.
\end{proof}

\vspace{5mm}

\begin{theorem} \label{theorem: infinite phinary number}
	The number of intervals in each axis of the phinary domain tends toward the number of letters in the infinite palindromic Fibonacci subword.
	\begin{align*}
		\text{Number of intervals} = \lim_{n \to \infty} \| w_n  \| - 2 = \lim_{n \to \infty} F_{n+3} - 2
	\end{align*}
\end{theorem}

\begin{proof}
	The rays extending from a vanishing point, i.e. from one of the two top GD vertices, through each of the facet's base vertices in the ``final", i.e. limiting, row of the GD define the interval markings for one of the phinary domain coordinate axes. As the $n$th row of the GD contains $F_{n+3}-2$ intervals between triangle base vertices, by induction, we can define the limit of the number of axis intervals to tend toward $\displaystyle{\lim_{n \to \infty} F_{n+3} - 2}$.
\end{proof}

\subsection{Geometry}
To be thorough, we will calculate the geometric requirements that would allow for such a projection in three dimensional space. To begin with, we will review perspective projection and look at the particular case of a projected geometric series in two dimensions.

When a set of consecutive intervals along a line, increasing in distance by a geometric progression, are projected onto an orthogonal line, the resulting image will also contain intervals defined by a geometric progression but with the reciprocal common ratio (Figure \ref{fig: 2D perspective}).

\vspace{5mm}

\begin{figure}[ht]
	\centering
  	\includegraphics[width=\textwidth]{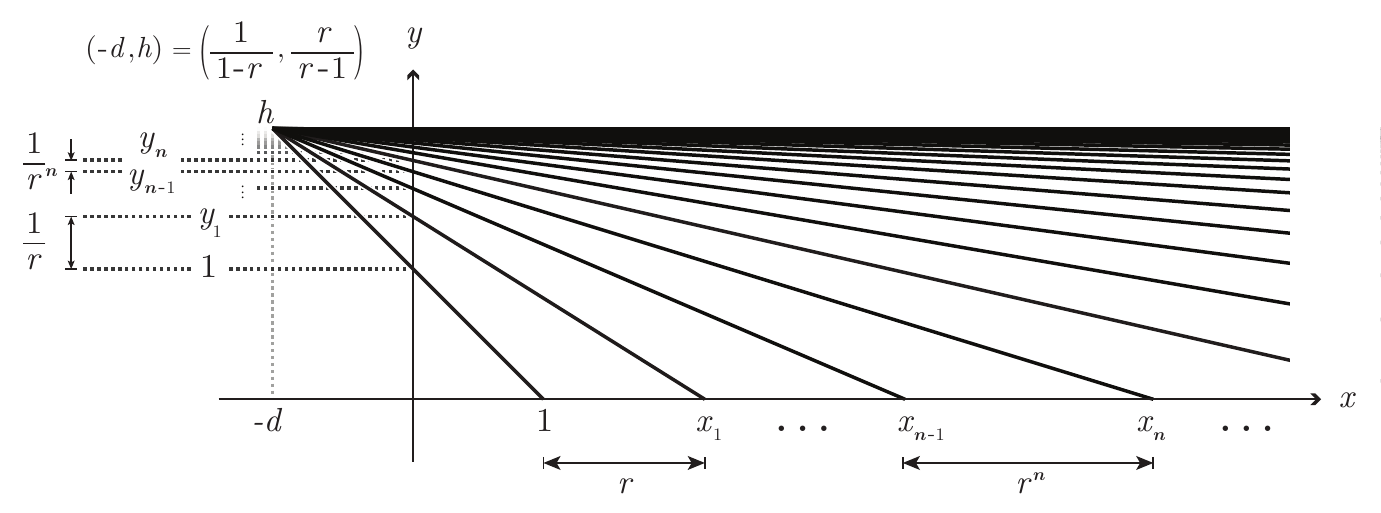}
   	\caption{A special case of a two-dimensional perspective projection from the $x$-axis onto the $y$-axis. The points $(1,0), (x_1,0), (x_2,0), \dots , (x_n,0), \dots$ are mapped onto the points $(0,1), (0,y_1), (0,y_2), \dots ,(0,y_n), \dots$, respectively, via the focal point at $(-d, h)=(\frac{1}{1-r}, \frac{r}{r-1})$.}
    	\label{fig: 2D perspective}
\end{figure}

\begin{lemma} \label{lemma: reciprocal geometric progression projection}
The image $S'=(0,y)$ of a point $S=(x,0)$ with $x=\sum_{i=0}^n r^i$, is defined by $y=\sum_{i=0}^n r^{-i}$ when mapped by a perspective projection through the focal point $(-d, h)=(\frac{1}{1-r}, \frac{r}{r-1})$, with $r \in \mathbbm{R}, r>0, n \geq 0$.
\end{lemma}

\begin{proof}
	Consider the projection of a series of points along a horizontal axis $x$ onto a vertical axis $y$ via the focal point $(-d, h)$ as shown in Figure \ref{fig: 2D perspective}. We are interested in a projection with the following properties: The points $(1,0), (x_1,0), (x_2,0), \dots , (x_n,0), \dots$ are mapped onto the points $(0,1), (0,y_1), (0,y_2), \dots ,(0,y_n), \dots$, respectively, where
	\begin{align*}
		x_1&=1+r\\
		y_1&=1+\frac{1}{r}
	\end{align*}
	for some positive $r \in \mathbbm{R}$. The ray from the focus to $(1,0)$ intersects the $x$-axis at a 45\textdegree\ angle. Therefore,
	\begin{align}
		h=d+1. \label{h0}
	\end{align}
	Given the second pair of points, $(x_1,0)$ and $(0,y_1)$, we use the similar triangles formed by the ray connecting them to find
	\begin{align}
		\frac{h}{d+x_1}&=\frac{y_1}{x_1} \nonumber \\
		h&=\frac{y_1(d+x_1)}{x_1} \nonumber \\
		h&=\frac{(1+\frac{1}{r})(d+1+r)}{1+r}. \label{h1}
	\end{align}
	Setting Equations \ref{h0} and \ref{h1} equal to each other yields
	\begin{align}
		d+1&=\frac{(1+\frac{1}{r})(d+1+r)}{1+r} \nonumber \\
		(d+1)(1+r)&=(1+\frac{1}{r})(d+1+r) \nonumber \\
		(r-\frac{1}{r})d&=1+\frac{1}{r} \nonumber \\
		d&=\frac{1}{r-1} \label{d}
	\end{align}
	and therefore,
	\begin{align} \label{h}
		h=\frac{r}{r-1}.
	\end{align}
	We now consider a new point $(x_n,0)$ along the $x$-axis such that
	\begin{align*}
		x_n=\sum_{i=0}^n r^i.
	\end{align*}
	We wish to find the value $y_n$ that corresponds to the $y$-coordinate of the projected point $(0,y_n)$ associated with $(x_n,0)$. As before, we set up an equation via similar triangles with
	\begin{align*}
		\frac{h}{d+x_n}&=\frac{y_n}{x_n}. \\
	\end{align*}
	Substituting in $d$ and $h$ from Equations \ref{d} and \ref{h} yields
	\begin{align*}
		\frac{y_n}{x_n}&=\frac{\frac{r}{r-1}}{\frac{1}{r-1}+x_n}\\
		&=\frac{r}{1+(r-1)x_n}\\
		y_n&=\frac{r}{\frac{1}{x_n}+r-1}.
	\end{align*}
	As $x_n$ is a geometric series, we can substitute in $x_n=\frac{1-r^n}{1-r}$ to find
	\begin{align}
		y_n&=\frac{r}{\frac{1-r}{1-r^n}+r-1}  \nonumber \\
		&=\frac{(1-r^n)r}{1-r+(1-r^n)(r-1)} \nonumber \\
		&=\frac{r-r^{n+1}}{r^n-r^{n+1}} \nonumber \\
		&=\frac{r^n-1}{r^n-r^{n-1}}. \label{eq}
	\end{align}
	Rewriting Equation \ref{eq} in terms of $\frac{1}{r}$ becomes
	\begin{align*}
		y_n&=\frac{(\frac{1}{r})^{-n}-1}{(\frac{1}{r})^{-n}-(\frac{1}{r})^{1-n}} \\
		&=\frac{1-(\frac{1}{r})^n}{1-\frac{1}{r}}
	\end{align*}
	which equals
	\begin{align*}
		y_n=\sum_{i=0}^n r^{-i}.
	\end{align*}
\end{proof}

\vspace{5mm}

As we will prove, the GD is the projected image of a field of infinitely many upright triangles in the first quadrant of the Cartesian plane (Figure \ref{fig: Image Plane 3D pyramids}). Our strategy for producing the GD through a perspective projection is to find the conditions required to project the grid lines of the two-dimensional phinary domain onto a plane such that the resulting image features the geometric progression of the fractal. This is demonstrated in the following theorem.

\vspace{5mm}

\begin{figure}[!ht]
    \centering
    \includegraphics[width=\textwidth]{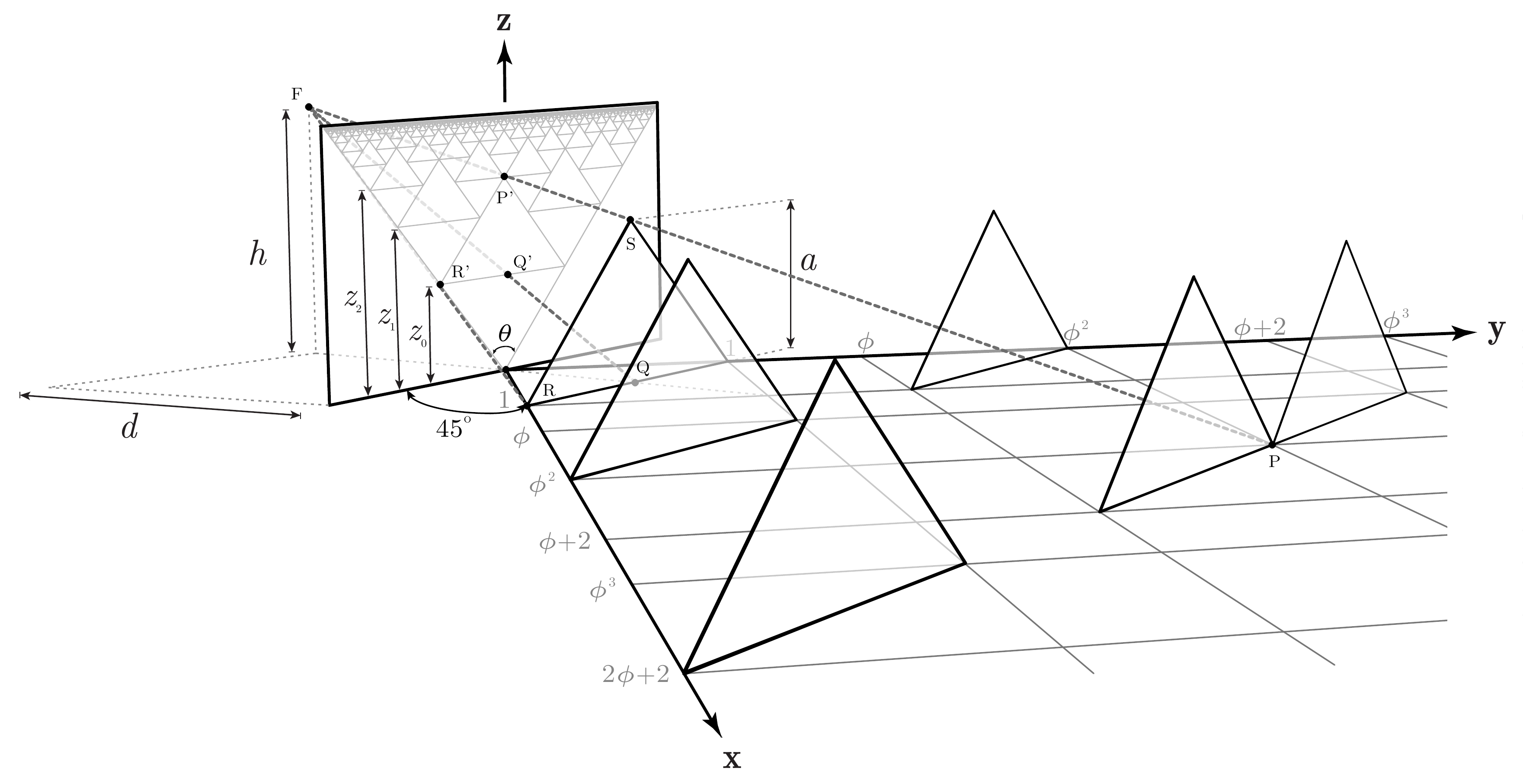}
    \caption{Standing triangles in the first quadrant of the Cartesian plane are projected onto an orthogonal plane, generating the golden diamond via a perspective projection.}
    \label{fig: Image Plane 3D pyramids}
\end{figure}

\begin{theorem} \label{theorem: GD perspective projection plane}
The GD is the image of the perspective projection $\mathbbm{R}^3 \to \mathbbm{R}^2$ onto the plane $x+y=0$ through the focal point $(-d_x,-d_y,h)=(-\frac{\phi}{2},-\frac{\phi}{2},\frac{\phi}{\sqrt{2}})$ of an infinite set of isosceles right triangles with altitudes $a=\frac{1}{\sqrt{2}}$, base lengths of $\sqrt{2}$, and base vertices $(u_1,v_1,0)$ and $(u_2,v_2,0)$ with
\begin{align*}
u_1+v_1=u_2+v_2=\sum_{i=0}^{n}\phi^i
\end{align*}
for all $n\in \mathbbm{Z}, n\geq0$.
\end{theorem}

\begin{proof}
To begin with, we consider the case outlined in Lemma \ref{lemma: reciprocal geometric progression projection}, in which a series of collinear points, spaced via a geometric progression of common ratio $r$ are mapped by a perspective projection onto an image featuring the reciprocal geometric progression with ratio $\frac{1}{r}$. This special case is a good starting point, as the points in the phinary domain spaced by powers of $\phi$ are mapped onto points spaced by powers of $\frac{1}{\phi}$. The lemma defines the focal point at a height above the domain and distance from the projection plane---which we will denote as $h^*$ and $d^*$, respectively---with $h^*=\frac{r}{r-1}$ and $d^*=\frac{1}{r-1}$. However, as can be seen in Figure \ref{fig: Image Plane 3D pyramids}, the perspective projection defined in the lemma for $\mathbbm{R}^2 \to \mathbbm{R}$ will correspond to a two-dimensional cross section, in which we map the points from the line $y=x$ in the Cartesian plane onto the $z$-axis. Therefore, we may use $r=\phi$ as our ratio but must scale the dimensions of all intervals by a factor of $\frac{1}{\sqrt{2}}$, yielding
\begin{align}
h^*&=\frac{r}{\sqrt{2}(r-1)}=\frac{\phi}{\sqrt{2}(\phi-1)}=\frac{\phi^2}{\sqrt{2}}\\
d^*&=\frac{1}{\sqrt{2}(r-1)}=\frac{1}{\sqrt{2}(\phi-1)}=\frac{\phi}{\sqrt{2}}.
\end{align}
As this is a perspective projection, we are assured that lines in the domain which are parallel to the projection plane will be mapped onto horizontal lines in the image, as desired.

At this stage, we must consider a final condition that has not been addressed by the parameters of the lemma. Referring to the points in Figure \ref{fig: Image Plane 3D pyramids}, we require that the ray, which sends point P to P', intersects S, the apex of the nearest triangle to the origin, as is evident by the geometry of the GD. To account for this, let the two-dimensional cross section spanned by the line $y=x$, in the $x$-$y$ coordinate  plane, and the $z$-axis be named the $t$-$z$ coordinate plane. In this plane, we have the focal point F* with $(t,z)$-coordinates equal to $(-d^*, h^*)=(-\frac{\phi}{\sqrt{2}},\frac{\phi^2}{\sqrt{2}})$, and point $\text{P}=(\phi^2\sqrt{2},0)$. Point S, with unknown altitude $a$, has coordinates $(\frac{1}{\sqrt{2}},a)$. The line containing the focal point F$^*$ and point P has the equation
\begin{align} \label{eq: z-t eq}
z=-\frac{1}{\phi^2}t+\sqrt{2}.
\end{align}
Inputting the point S$=(\frac{1}{\sqrt{2}},a)$ into this equation gives
\begin{align*}
a&=-\frac{1}{\phi^2\sqrt{2}}+\sqrt{2}\\
&=\frac{\phi}{\sqrt{2}}.
\end{align*}
This result implies that each standing triangle in the domain has an altitude of $\frac{\phi}{\sqrt{2}}$, and by Equation \ref{eq: z-t eq}, P' has a $z$-value of $\sqrt{2}$, which results in the GD image with a height of $\frac{\phi^2}{\sqrt{2}}$. These three values are a factor of $\phi$ greater than what would be expected from how we have defined the GD. If we instead define the standing triangles to have altitudes of $a=\frac{1}{\sqrt{2}}$, then we must scale $h^*$ by $\frac{1}{\phi}$, as well. This scaling preserves the common ratio $\frac{1}{\phi}$ in the GD image while reducing the overall size, yielding
\begin{align*}
h=\frac{h^*}{\phi}=\frac{\phi}{\sqrt{2}}.
\end{align*} 
Therefore, we have
\begin{align*}
h&=\frac{\phi}{\sqrt{2}}\\
d&=\frac{\phi}{\sqrt{2}}\\
a&=\frac{1}{\sqrt{2}}.
\end{align*} 
This gives $(x,y,z)$-coordinates of the focal point as F$=(-\frac{\phi}{2},-\frac{\phi}{2},\frac{\phi}{\sqrt{2}})$.
The coordinates of the base vertices of the standing triangles come from Lemma \ref{lemma: GD points}, such that the sum of the $x$ and $y$ coordinates of each equals $\sum_{i=0}^{n}\phi^i$ for the same nonnegative integer $n$. By the values defined here, the remaining points of the GD can be easily found to map as desired. 
\end{proof}
\begin{remark}
Note that this projection results in a GD image with $\theta=90^\circ$ as indicated in Figure \ref{fig: Image Plane 3D pyramids}, producing an outer triangle with sides $\phi$, $\phi$, $\sqrt{2}\phi$. The standing triangles in the plane will also be isosceles right triangles, with sides 1, 1, $\sqrt{2}$. This implies that the equilateral triangle version of the GD that we have been using in this paper, although aesthetically appealing, is apparently not the most natural choice if one defines the GD projection in the above manner; for our purposes, however, this angle scaling is inconsequential.
\end{remark}

\section{Phinary Arithmetic} \label{section: operations}

\subsection{The Hyperoperation}
In this section, we explore the arithmetic properties of phinary values. In order to do so, we first review the concepts behind the familiar arithmetical operations. The ordinary integers $\mathbbm{Z}$ form a ring and are convenient for their closure under the successor operation, addition, multiplication, exponentiation, and further polyation operations; these specific operations being defined generally by the hyperoperations, a class of operators developed through the independent work of Albert Bennett, Wilhelm Ackermann, R. L. Goodstein, and others \cite{Bennett_1915}\cite{ackermann1928hilbertschen}\cite{goodstein1947transfinite}. Various notation is found in the literature, and we will define the hyperoperation as follows:

\begin{definition} \label{definition: hyperoperation}
Let $\otimes_n : (\mathbbm{S})^3 \to \mathbbm{S}$ for set $\mathbbm{S}$ be a hyperoperation defined recursively through the successor operator $\text{S}(b) : \mathbbm{S} \to \mathbbm{S}$ and predecessor operator $\text{P}(b) : \text{P}(\text{S}(b))=b$ for all $a,b \in \mathbbm{S}$ as follows:
\begin{align*}
a \otimes_n b = \begin{cases} 
      \text{S}(b) & \text{if} \ n=0  \\
      a & \text{if} \ n=1 \ \text{and} \ b=0 \\
      0 & \text{if} \ n=2 \ \text{and} \ b=0 \\
      1 & \text{if} \ n \geq 3 \ \text{and} \ b=0 \\
      a \otimes_{n-1} \left(a \otimes_{n} \text{P}(b) \right) & \text{otherwise}
   \end{cases}
\end{align*}
The inverse operations are defined in terms of negative $n$ values, such that
\begin{align*}
a\otimes_{n}b\otimes_{\!-n}b &=a \\
a\otimes_{\!-n}b\otimes_{n}b &=a
\end{align*}
\end{definition}

\vspace{5mm}

If we wish to generate the ordinary operators of addition, multiplication, and so on, we simply define the successor operation as $\text{S}(b) := b+1$, as demonstrated below. 

\begin{definition}
The ordinary operators are defined via the hyperoperation and the successor operation $\text{S}(b) := b+1$, and are denoted as follows:
\begin{align*}
\odot_n := \left \{\otimes_n \ | \ \text{S}(b) := b+1, \ b \in \mathbbm{S} : \mathbbm{S} \subseteq \mathbbm{R} \right \}
\end{align*}
For $0 < n \leq 3$, we use the typical names and symbolic notation:

\vspace{5mm}

\begin{center}
\begin{tabular}{  m{4cm}  m{4cm} m{3cm}} 
  (addition) & $a \odot_1 b =: a+b$ & ``$a$ $plus$ $b$" \\ 
  (multiplication) & $a \odot_2 b =: a \cdot b$ or $ab$ & ``$a$ $times$ $b$"\\
  (exponentiation) & $a \odot_3 b =: a^b$ & ``\textit{a to the power of b}"
\end{tabular}
\end{center}

\vspace{5mm}

For $3 \leq n < 0$, we use the conventional inverse operations:

\vspace{5mm}

\begin{center}
\begin{tabular}{  m{4cm}  m{4cm} m{3cm}} 
  (subtraction) & $a \odot_{\!-1} b =: a-b$ & ``$a$ minus $b$" \\ 
  (division) & $a \odot_{\!-2} b =: a / b$ or $\displaystyle{\frac{a}{b}}$ & ``$a$ divided by $b$"\\
  (radicalization) & $a \odot_{\!-3} b =: \sqrt[\leftroot{-2}\uproot{2}b]{a}$ & ``the $b$th root of $a$"
\end{tabular}
\end{center}

\vspace{5mm}

As is customary, higher orders of $n$ are given the names tetration ($n=4$), pentation ($n=5$), hexation ($n=6$), and so on.
\end{definition}

\vspace{5mm}

Notice that the above ordinary operations are defined for an arbitrary subset of the real numbers, denoted $\mathbbm{S}$. To be useful, however, they are best employed on a set with a group structure, otherwise closure and other desirable properties are not ensured. Of course, on the ring of integers $\mathbbm{Z}$, $\odot_{n}$ for $n \geq -1$, properties of closure, associativity, distributivity, etc. are well-defined in the conventional sense. Likewise, for the field of rational numbers $\mathbbm{Q}$, similar such properties hold for $n \geq -2$. Furthermore, the field of reals $\mathbbm{R}$ enjoys operations for all $n$. The details of such properties on $\mathbbm{Z}, \mathbbm{Q},$ and $\mathbbm{R}$ are, of course, well-known and outside the scope of this paper.

We wish to use operations such as addition, multiplication, exponentiation, etc. on phinary values, but if we naively proceed with the current operations, we will find that they all refuse closure.

\vspace{5mm}

\begin{theorem} \label{theorem: phinary nonclosure}
The phinary integers $\mathbbm{Z}_\Phi$ are not closed under the ordinary operators $\odot_n$ for all $n$.
\end{theorem}

\begin{proof}
The operator $\odot_0$ equals the successor operator S, and all other operators $\odot_n$ are recursively defined through it. For example,
\begin{align*}
a \odot_1 b &:= a + b \\
&= \text{S}(a + \text{P}(b)) \\
&= \text{S}(\text{S}(a + \text{P}(\text{P}(b)))) \\
&= \text{S}(\text{S}(\text{S}(a + \text{P}(\text{P}(\text{P}(b)))))) \\
& \ \ \vdots
\end{align*}
Recalling that the predecessor P is defined by the successor function, i.e. $\text{P}(\text{S}(b))=b$, and assuming we know the predecessor of $b$ or some further predecessor of its predecessor, we ought to be able to eventually solve for $a+b$. At which point, all higher-order hyperoperations between $a$ and $b$ can then be carried out.

However, we show that $\mathbbm{Z}_\Phi$ is not closed under $\text{S}(b) := b+1, \ \forall b \in \mathbbm{Z}_\Phi$ through a counterexample---namely, $\text{S}(1), \ 1 \in \mathbbm{Z}_\Phi$:
\begin{align*}
\text{S}(1) = 1+1 = 2 \notin \mathbbm{Z}_\Phi.
\end{align*}
Infinitely other such counterexamples exist. As the successor operation fails closure on $\mathbbm{Z}_\Phi$, it follows by induction that all ordinary operations will fail to provide closure, as well.
\end{proof}

\subsection{Functions} \label{subsection: functions}

Although the set of phinary numbers is not closed under the natural operators, closure does hold for a subset of values. We can define a number of functions that map the set of phinary numbers to itself.

\vspace{5mm}

\begin{lemma} \label{lemma: p plus power}
Addition between a phinary number $p$ and a power of the golden ratio with a value greater than or equal to $p$ returns a phinary number, with the exception of $1+1$ and $\phi+\phi$.
\begin{align*}
\left \{f : p \mapsto p+\phi^n, \ \mathbbm{Z}^+_\Phi \to \mathbbm{Z}^+_\Phi \ | \ p \in \mathbbm{Z}^+_\Phi, \ n \in \mathbbm{Z}^+ : p \leq \phi^n, \ \{n,p\} \neq \{0,1\} \lor \{1,\phi\} \right \} 
\end{align*}
\end{lemma}

\begin{proof}
By Theorem \ref{theorem: phi power fib word}, every power of the golden ratio $\phi^n$ is equal to the sum of letters in a proper Fibonacci word $w_{n+1}$ with letters $1$ and $\psi$. And by definition, a Fibonacci word is the concatenation of the two preceding words (Definition \ref{definition: Fibonacci words}). Therefore, we can write
\begin{align}
	\phi^{n+2} \Leftrightarrow w_{n+3} 	&= w_{n+2}w_{n+1}\nonumber\\
							&= (w_{n+1}w_{n})(w_{n}w_{n-1})\nonumber\\
							&= (w_{n+1}w_{n})((w_{n-1}w_{n-2})w_{n-1})\nonumber\\
							&= w_{n+1}(w_{n}w_{n-1})w_{n-2}w_{n-1}\nonumber\\
							&= w_{n+1}w_{n+1}w_{n-2}w_{n-1}\label{eq: word expansion 1}
\end{align}
As any phinary number is equal to the sum of letters in an initial Fibonacci subword (Definition \ref{def: initial fib subword}, Theorem \ref{theorem: phi number initial fib subword}), Equation \ref{eq: word expansion 1} demonstrates that $w_{n+1}w_{n+1}$ is a possible initial subword, for $n>1$, as $\nexists w=11$ or $1\psi1\psi$. Let $w_s$ be an initial subword that succeeds $w_{n+1}$ and preceeds or is equal to  $w_{n+1}w_{n+1}$. In formulas,
\begin{align}
w_{n+1} \preceq w_s \preceq w_{n+1}w_{n+1}. \label{eq: word ineq}
\end{align}
Taking the sum of the letters in each of the above words of Equation \ref{eq: word ineq} corresponds to an inequality for some phinary number $q$, such that
\begin{align*}
\phi^n \leq q \leq 2\phi^n
\end{align*}
for $n>1$. Therefore, there exists some phinary number $q=p+\phi^n$, where $0 \leq p \leq \phi^n$, $p \in \mathbbm{Z}^+_\Phi$ and $n>1$, with the addition of $q=1+\phi$, which is clearly permitted.
\end{proof}

\vspace{5mm}

\begin{lemma}
The following sum between phinary numbers returns a phinary number.
\begin{align*}
\left \{f: p \mapsto p+2\phi^{n}\sum_{i=0}^k \phi^{-3i}, \ \mathbbm{Z}^+_\Phi \to \mathbbm{Z}^+_\Phi \ | \ p \in \mathbbm{Z}^+_\Phi, \ k,n \in\mathbbm{Z}^+ : 0 \leq p \leq 2\phi^{n-3(k+1)}, \ n>3(k+1) \right \}
\end{align*}
\end{lemma}

\begin{proof}
This proof continues from the previous one, where we can extend our initial derivation, such that
\begin{align}
	\phi^{n+2} \Leftrightarrow w_{n+3} 	&= w_{n+2}w_{n+1}\nonumber\\
							&= (w_{n+1}w_{n})(w_{n}w_{n-1})\nonumber\\
							&= (w_{n+1}w_{n})((w_{n-1}w_{n-2})w_{n-1})\nonumber\\
							&= w_{n+1}(w_{n}w_{n-1})w_{n-2}w_{n-1}\nonumber\\
							&= w_{n+1}w_{n+1}w_{n-2}w_{n-1}\nonumber\\
							&= w_{n+1}w_{n+1}w_{n-2}(w_{n-2}w_{n-3})\nonumber\\
							&= w_{n+1}w_{n+1}w_{n-2}w_{n-2}w_{n-3}\label{eq: word expansion 2}\\
							&= w_{n+1}w_{n+1}w_{n-2}w_{n-2}(w_{n-4}w_{n-5})\nonumber\\
							& \ \ \ \ \ \ \ \ \ \ \ \vdots\nonumber\\
							&= w_{n+1}w_{n+1}w_{n-2}w_{n-2}w_{n-5}w_{n-5}w_{n-8}w_{n-8}\dots\label{eq: word expansion 3}	
\end{align}
By the same reasoning used in Lemma \ref{lemma: p plus power}, Equation \ref{eq: word expansion 2} permits other opportunities for adding phinary numbers. As $w_{n+1}w_{n+1}$ and $w_{n+1}w_{n+1}w_{n-2}w_{n-2}$ both represent phinary numbers for $n>2$, so too does any initial subword $w_t$ in between. That is,
\begin{align*}
w_{n+1}w_{n+1} \preceq w_s \preceq w_{n+1}w_{n+1}w_{n-2}w_{n-2},
\end{align*}
which implies that there is some phinary number $r$, such that
\begin{align*}
					2\phi^n \leq 	r \leq 	2\phi^n+2\phi^{n-3}.
\end{align*}
for $n>3$. As before, this implies that there exists a phinary number $r=p+2\phi^n$, where $0 < p \leq 2\phi^{n-3}$, $p \in \mathbbm{Z}^+_\Phi$ and $n>3$. Continuing in this manner, from Equation \ref{eq: word expansion 3}, we compile a list of possible sums over the phinary numbers
\begin{align*}
		\{f: p \mapsto p+2\phi^{n} \ | \ 0 \leq p \leq 2\phi^{n-3}, \ n>3 \}\\
		\{f: p \mapsto p+2\phi^{n}(1+\phi^{-3}) \ | \ 0 \leq p \leq 2\phi^{n-6}, \ n>6 \}\\
		\{f: p \mapsto p+2\phi^{n}(1+\phi^{-3}+\phi^{-6}) \ | \ 0 \leq p \leq 2\phi^{n-9}, \ n>9 \}\\
		\vdots  \ \ \ \ \ \ \ \ \ \ \ \ \ \ \ \ \ \ \ \ \ \ \ \ \ \ \ \ \ \ \ \ \ \ \ \ \ \ \ \\
		\left \{f: p \mapsto p+2\phi^{n}\sum_{i=0}^k \phi^{-3i}, \ \mathbbm{Z}^+_\Phi \to \mathbbm{Z}^+_\Phi \ | \ p \in \mathbbm{Z}^+_\Phi, \ k,n \in\mathbbm{Z}^+ : 0 \leq p \leq 2\phi^{n-3(k+1)}, \ n>3(k+1) \right \}.
\end{align*}
\end{proof}

\begin{remark}
Note, the above formulas only account for a select set of sums between phinary number pairs. However, for our purposes it will be enough. Nevertheless, a more detailed study of these relationships is worth pursuing in future research.
\end{remark}

\vspace{5mm}

This validates the claim that phinary numbers can be written as sums of powers of the golden ratio.

\vspace{5mm}

\begin{corollary} \label{corollary: sums of powers}
The sum between two powers of the golden ratio equals a phinary number, with the exception of $1+1$ and $\phi+\phi$.
\end{corollary}

\begin{proof}
This follows directly from Lemma \ref{lemma: p plus power}.
\end{proof}

\vspace{5mm}

\begin{theorem} \label{theorem: sums of phi pow}
All phinary integers can be written as nonnegative powers of the golden ratio or as sums of such powers.
\begin{align*}
p = \left \{ \sum_{i=0}^n a_i\phi^i \ | \ \forall p \in \mathbbm{Z}^+_\Phi, \ n \in \mathbbm{Z}^+ : a_i=0 \ \lor \ 1 \right \}
\end{align*}
\end{theorem}

\begin{proof}
The theorem is trivially true for phinary numbers of the form $\phi^n$. For the values of $q \in \mathbbm{Z}^+_\Phi$ that lie between powers of the golden ratio, we turn to Lemma \ref{lemma: p plus power}, which proved that for any $ \phi^{n-1} < q < \phi^n$ we may say that $q = \phi^{n-1} + p$ such that $p<\phi^{n-2}$ when $n>2$, and $p=\phi^0$ when $n=2$. Therefore, by induction, the theorem holds. This is, of course, known and was demonstrated in \S \ref{subsection: phinary numbers} in terms of the base-phi number system.
\end{proof}

\vspace{5mm}

\begin{corollary} \label{corollary: phi mult}
The phinary integers are closed under multiplication with powers of the golden ratio.
\begin{align*}
\left \{f : p \mapsto \phi^n \!\cdot p, \ \mathbbm{Z}^+_\Phi \to \mathbbm{Z}^+_\Phi \ | \ p \in \mathbbm{Z}^+_\Phi, \ n \in \mathbbm{Z}^+ \right \}
\end{align*}
\end{corollary}

\begin{proof}
As proved in the previous theorem, all phinary integers are powers of phi or sums of powers of phi. Therefore, multiplication by any power of the golden ratio will also yield a phinary integer. 

This can also be seen in the Fibonacci word substitution morphism $A\to AB$ and $B\to A$ (Theorem \ref{theorem: fib word substitution rule}) that results from multiplying a phinary integer by phi. That is, any phinary integer is an initial Fibonacci subword with letters $1$ and $\psi$. Multiplying a phinary integer by the golden ratio sends each $1$ to $\phi$ and $\psi$ to $1$, such that
\begin{align*}
1 &\to 1\psi \\
\psi &\to 1
\end{align*}
which results in another initial Fibonacci subword and therefore another phinary integer. As multiplication by a power of phi is simply repeated multiplication of phi, the proof holds.
\end{proof}

\vspace{5mm}

Fortunately, there is another way to proceed toward operator closure on phinary values. The definition begins with the introduction of a $numerical$ $count$.

\subsection{Numerical Count}

Clearly, the positive integers and phinary numbers have a different sense of ``rhythm." The positive integers are pervasively homogenous and stand in stark contrast to the phinary numbers, which exhibit an inhomogeneous, yet highly-structured configuration. Here we define this sense of rhythm by the concept of a \textit{numerical count}. We define their respective counts in the following manner:

\begin{definition}
	Let $C$ be a numerical count such that a periodic or quasiperiodic pattern exists between elements of a set $S$. The count may be associated with a set of ordinals $\mathbbm{O}$, in which case a subscript denotes the pattern. It is said that $C$ iterates on $S$, written formally as
	\begin{align} 
		C=\mathbin{\prec}(S)_{\mathbbm{O}}.
	\end{align}
\end{definition}

\begin{theorem} \label{definition: natural count}
The natural count $\mathbin{\prec}(A)_{\mathbbm{Z}^+}$ is a periodic sequence upon one element $A$ of arbitrary type, such that
\begin{align*}
\mathbin{\prec}(A)_{\mathbbm{Z}^+}= \{A,A,A,A,A,A,A,A,A,...\}.
\end{align*}
\end{theorem}

\begin{remark}
The natural count is a formal definition for the the trivial pattern of uniform successor. We define it here only to contrast it with the phinary count, which we present in the following theorem.
\end{remark}

\vspace{5mm}

\begin{theorem} \label{definition: phinary count}
The phinary count $\mathbin{\prec}(A,B)_\Phi$ is an aperiodic binary sequence defined by the infinite Fibonacci word upon two elements $A$ and $B$ of arbitrary type, such that
\begin{align*}
\mathbin{\prec}(A,B)_\Phi = \{A,B,A,A,B,A,B,A,A,...\}.
\end{align*}
The $n$th member or \emph{tally} of the phinary count on elements $A$ and $B$ is
\begin{align*}
\mathbin{\prec}_n(A,B)_\Phi = \begin{cases} 
      A & \text{if} \ \lfloor(n+1)\phi \rfloor - \lfloor n\phi \rfloor - 2= 1\\
      B & \text{if} \ \lfloor(n+1)\phi \rfloor - \lfloor n\phi \rfloor - 2= 0,
   \end{cases}
\end{align*}
where $\lfloor x \rfloor$ is the floor function of x.
\end{theorem}

\begin{proof}
The $n$th tally follows from Theorem \ref{theorem: fib word nth letter rule} for the $n$th letter of the infinite Fibonacci word.
\end{proof}

\subsection{The Phinary Successor}

The phinary numbers exhibit a different notion of successor than the positive integers. Unlike the latter, the phinary successor acts locally on a number in $\mathbbm{Z}^+_\Phi$, due to the aperiodic pattern of the infinite Fibonacci word. This means that we cannot determine the successor of a number $p \in \mathbbm{Z}^+_\Phi$ without knowing every value that precedes $p$.

\vspace{5mm}

\begin{lemma} \label{lemma: phinary successor}
The phinary successor $\text{S}_\Phi : \mathbbm{Z}^+_\Phi \to \mathbbm{Z}^+_\Phi$ is a local unary operation defined by the phinary count, such that the $n$th successor of zero, for $n \in \mathbbm{Z}^+$, is defined by the $n$th phinary tally on values 1 and $\psi$.
\begin{align*}
\text{S}^n(0)_\Phi &= \sum_{i=1}^n\mathbin{\prec}_i(1,\psi)_\Phi \\
\mathbbm{Z}^+_\Phi &= \bigcup_{i=1}^\infty \text{S}^i(0)_\Phi.
\end{align*}
Where $\text{S}^n(p)$ is the $n$th recursive evaluation of S on $p$ with $\text{S}^0(p)=p$, e.g. $\text{S}^3(p)_\Phi = \text{S}(\text{S}(\text{S}(p)))_\Phi$.
\end{lemma}

\begin{proof}
The differences between consecutive phinary numbers is nonconstant and alternates aperiodically between values of 1 and $\psi$, where $\psi=\frac{1}{\phi}$. The pattern is defined by the phinary count, i.e. infinite Fibonacci word:
\begin{align*}
\text{S}(0)_\Phi &:= 0 + \mathbin{\prec}_1(1,\psi)_\Phi \\
\text{S}(\text{S}(0))_\Phi &:= 0 + \mathbin{\prec}_1(1,\psi)_\Phi + \mathbin{\prec}_2(1,\psi)_\Phi\\
\text{S}(\text{S}(\text{S}(0)))_\Phi &:= 0 + \mathbin{\prec}_1(1,\psi)_\Phi + \mathbin{\prec}_2(1,\psi) + \mathbin{\prec}_3(1,\psi)_\Phi \\
\text{S}(\text{S}(\text{S}(\text{S}(0))))_\Phi &:= 0 + \mathbin{\prec}_1(1,\psi)_\Phi + \mathbin{\prec}_2(1,\psi)_\Phi + \mathbin{\prec}_3(1,\psi)_\Phi + \mathbin{\prec}_4(1,\psi)_\Phi \\
\text{S}(\text{S}(\text{S}(\text{S}(\text{S}(0)))))_\Phi &:= 0 + \mathbin{\prec}_1(1,\psi)_\Phi + \mathbin{\prec}_2(1,\psi)_\Phi + \mathbin{\prec}_3(1,\psi)_\Phi + \mathbin{\prec}_4(1,\psi)_\Phi + \mathbin{\prec}_5(1,\psi)_\Phi
\end{align*}
\begin{align*}
&\vdots \\ 
\text{S}^n(0)_\Phi =& \sum_{i=1}^n\mathbin{\prec}_i(1,\psi)_\Phi,
\end{align*}
which yields
\begin{align*}
\text{S}(0)_\Phi &:= 0 + 1 = 1\\
\text{S}(\text{S}(0))_\Phi &:= 0 + 1 + \psi = \phi \\
\text{S}(\text{S}(\text{S}(0)))_\Phi &:= 0 + 1 + \psi + 1 = \phi + 1\\
\text{S}(\text{S}(\text{S}(\text{S}(0))))_\Phi &:=  0 + 1 + \psi + 1 + 1 = \phi + 2\\
\text{S}(\text{S}(\text{S}(\text{S}(\text{S}(0)))))_\Phi &:= 0 + 1 + \psi + 1 + 1 + \psi = 2\phi + 1.
\end{align*}
Therefore,
\begin{align*}
&\vdots \\
\mathbbm{Z}^+_\Phi =& \bigcup_{i=1}^\infty \text{S}^i(0)_\Phi.
\end{align*}
\end{proof}

\vspace{5mm}

\begin{theorem} \label{theorem: phinary number successor}
The successor of an arbitrary phinary number is given by
\begin{align*}
\text{S}(p)_\Phi = \left\{p + (1-\psi)(\lfloor(n+1)\phi \rfloor - \lfloor n\phi \rfloor - 1) + \psi \ | \ n \in \mathbbm{Z}^+ : p:=\text{S}^{n-1}(0)_\Phi, \ n>1 \right\}.
\end{align*}
\end{theorem}

\begin{proof}
From Lemma \ref{lemma: phinary successor}, we can state that for some positive integer $n>1$ and phinary number $p$,
\begin{align*}
p &= \text{S}^{n-1}(0)_\Phi \\
\text{S}(p)_\Phi &= \text{S}^{n}(0)_\Phi\\
&= \sum_{i=1}^n\mathbin{\prec}_i(1,\psi)_\Phi \\
&= p + \mathbin{\prec}_n(1,\psi)_\Phi.
\end{align*}
As shown in Lemma \ref{lemma: phinary successor}, the difference between consecutive phinary numbers alternates between 1 or $\psi$ according to the phinary count. From Definition \ref{definition: phinary count}, $q=\lfloor(n+1)\phi \rfloor - \lfloor n\phi \rfloor - 1$ equals either 1 or 0 for a given $n$. We can modify the equation, such that
\begin{align*}
Xq+Y = \begin{cases} 
      1 & \text{if} \ q= 1\\
      \psi & \text{if} \ q= 0.
   \end{cases}
\end{align*}
Solving for X and Y yields
\begin{align*}
X&=1-\psi \\
Y&=\psi.
\end{align*}
Therefore,
\begin{align*}
\mathbin{\prec}_n(1,\psi)_\Phi = (1-\psi)(\lfloor(n+1)\phi \rfloor - \lfloor n\phi \rfloor - 1) + \psi.
\end{align*}
\end{proof}

\vspace{5mm}

We can extend the notion of phinary predecessor/successor to the phinary integers, such that the relationship on the positive phinary integers, i.e. phinary numbers, is ``mirrored" on the negative phinary integers.

\vspace{5mm}

\begin{theorem}
The successor of an arbitrary phinary integer is given by
\begin{align*}
\text{S}(p)_\Phi = \begin{cases}
	p + (1-\psi)(\lfloor(n+1)\phi \rfloor - \lfloor n\phi \rfloor - 1) + \psi & \text{if} \ \ p \geq 0 \\
	p + (1-\psi)(\lfloor n\phi \rfloor - \lfloor (n-1)\phi \rfloor - 1) + \psi & \text{if} \ \ p < 0
	\end{cases}
\end{align*}
where $p:=\text{S}^{n-1}(0)_\Phi$ for $n>0$ and $p:=-\text{S}^{1-n}(0)_\Phi$ for $n \leq 0$, $n \in \mathbbm{Z}$.
\end{theorem}

\begin{proof}
If $n>0$, then $p \geq 0$ and we recover the successors of the phinary numbers $\mathbbm{Z}^+_\Phi$, proven in Theorem \ref{theorem: phinary number successor}, with the addition of $n=1, p=0$, such that $\text{S}(0)_\Phi=0 + (1-\psi)(\lfloor(1+1)\phi \rfloor - \lfloor (1)\phi \rfloor - 1) + \psi=1$, as desired.
For a negative phinary integer $q$, we want to define its successor such that $\text{S}(q)_\Phi=-\text{P}(-q)_\Phi$. That is, the successors of negative phinary integers should ``mirror" the positive phinary integers by equaling the negative of the predecessor of the positive phinary value with the same magnitude. Let $p$ be a positive phinary integer, such that
\begin{align*}
\text{S}(p)_\Phi=p+\Delta p_n
\end{align*}
where $\Delta p_n=(1-\psi)(\lfloor(n+1)\phi \rfloor - \lfloor n\phi \rfloor - 1) + \psi$. Therefore, the predecessor of $p$ is
\begin{align*}
\text{P}(p)_\Phi=p-\Delta p_{n-1}
\end{align*}
where $\Delta p_{n-1}=(1-\psi)(\lfloor n\phi \rfloor - \lfloor (n-1)\phi \rfloor - 1) + \psi$. Finally, if $q=-p$ and $\text{S}(q)_\Phi=-\text{P}(-q)_\Phi$, then
\begin{align*}
\text{S}(q)_\Phi&=-\text{P}(-q)_\Phi\\
&=-\text{P}(p)_\Phi\\
&=-p+\Delta p_{n-1}\\
&=q+\Delta p_{n-1}\\
&=q + (1-\psi)(\lfloor n\phi \rfloor - \lfloor (n-1)\phi \rfloor - 1) + \psi.
\end{align*}
\end{proof}

\begin{lemma} \label{lemma: successor iso}
The phinary successor operator on the phinary integers is isomorphic to the ordinary successor operator on the ordinary integers.
\begin{align*}
(\mathbbm{Z}_\Phi,S_\Phi) \cong (\mathbbm{Z},S)
\end{align*}
\end{lemma}

\begin{proof}
As the successor operation simply iterates through the elements of a well-ordered set, an isomorphism exists for any successor operation that has been defined to traverse a particular ordinal set.
\end{proof}

\begin{remark}
The above lemma will be useful, as it allows us to count or index a sequence with phinary numbers---a feature that we utilize in \S \ref{The Fibonacci diatomic Recurrence Relation}.
\end{remark}

\subsection{Phinary Operators} \label{subsection: phinary operators}

With the definition of the phinary successor we can now define a hyperoperation for phinary values; it follows in the exact same manner as that for the ordinary integers and rationals. For completeness, we will define the full suite of operators that mirror the ordinary operators. However, in this paper, we will be interested in only a small subset of these operations, as is necessary for the applications that follow.

\begin{definition} \label{definition: phinary operators}
The phinary operators are defined by the hyperoperation such that
\begin{align*}
\circledast_n := \left \{ \otimes_n \ | \ (\mathbbm{Z}^+_\Phi)^3 \to \mathbbm{Z}^+_\Phi : a,b \in \mathbbm{Z}^+_\Phi, \  a \otimes_0 b = \text{S}(b)_\Phi \right \}.
\end{align*}
For $0 < n \leq 3$, we use the following names and symbolic notation:

\vspace{5mm}

\begin{center}
\begin{tabular}{  m{5cm}  m{3.5cm} m{3cm}} 
  (phinary addition) & $a \circledast_1 b =: a \dagger b$ &``$a$ $dagger$ $b$" \\ 
  (phinary multiplication) & $a \circledast_2 b =: a \star b$ &``$a$ $star$ $b$" \\
  (phinary exponentiation) & $a \circledast_3 b =: a^{\underline{b}}$ &``$a$ \textit{to the bar of} $b$"
\end{tabular}
\end{center}

\vspace{5mm}

For $3 \leq n < 0$, we define the following inverse operations:

\vspace{5mm}

\begin{center}
\begin{tabular}{  m{5.2cm}  m{3.5cm} m{3.2cm} } 
  (phinary subtraction) & $a \circledast_{\!-1} b =: a \rightharpoondown b$ &``$a$ $hook$ $b$" \\ 
  (phinary division) & $a \circledast_{\!-2} b =:  a \sslash b$ or $\displaystyle{\efrac{a}{b}}$ &``$a$ $stripe$ $b$"\\ 
  (phinary radicalization) & $a \circledast_{\!-3} b =: \sqrt[\leftroot{-2}\uproot{2}\underline{b} ]{a}$ &``\textit{the bth bar-root of a}"
\end{tabular}
\end{center}

\vspace{5mm}

The higher orders of $n$ are given the names phinary tetration ($n=4$), phinary pentation ($n=5$), phinary hexation ($n=6$), and so on.
\end{definition}

\vspace{5mm}

\begin{theorem} \label{theorem: operation iso}
The phinary operators on the phinary integers are isomorphic to the ordinary operators on the ordinary integers.
\begin{align*}
(\mathbbm{Z}_\Phi,\circledast_n) \cong (\mathbbm{Z},\odot_n)
\end{align*}
\end{theorem}

\begin{proof}
Both the phinary operators and ordinary operators are defined by the hyperoperation (Definition \ref{definition: hyperoperation}), differing only by their respective successor operations. As shown in Lemma \ref{lemma: successor iso}, the phinary and natural successor operators are isomorphic to one another, with respect to their constituent sets. Moreover, as both sets have the same cardinality, i.e. they are both countably infinite, we can therefore define a ring isomorphism between the two sets.
\end{proof}

\begin{remark}
Consequently, group properties are inherited from the normal integers, such that the phinary integers form a ring under the phinary operations. One could go on to define the field of phinary rational numbers, but as it is outside the scope of this paper, it will not be addressed here (See the concluding remarks in \S \ref{section: concluding remarks} for more on this point).
\end{remark}

\vspace{5mm}

\begin{corollary}
The phinary integers $\mathbbm{Z}_\Phi$ are closed under the phinary operators $\circledast_n$, for $n\geq0$.
\end{corollary}

\begin{proof}
This follows naturally from Theorem \ref{theorem: operation iso}.
\end{proof}

\vspace{5mm}

A helpful way to think of phinary operations on the phinary numbers is through the above isomorphism. For example, if one wants to evaluate $(\phi+2) \dagger (3\phi+1)$, it is easy to make the following observation: The value $(\phi+2)$ is the fourth phinary number and $(3\phi+1)$ is the seventh phinary number. Therefore, phinary addition between these two values will yield the fifteenth phinary number $6\phi+3$, as $(\phi+2) \dagger (3\phi+1)=6\phi+3$ is isomorphic to $4+11=15$. In the same manner, one can quickly deduce the results of any phinary operation on any phinary values.

A somewhat surprising feature of these operations is that phinary addition between powers of phi is equivalent to ordinary addition between the same values. This property will become useful in the applications of the next section.

\vspace{5mm}

\begin{corollary} \label{corollary: addition iso}
Between powers of the golden ratio, phinary addition is isomorphic to addition, with the exception of $1+1$ and $\phi+\phi$.
\begin{align*}
(\Phi^n,\dagger) \cong (\Phi^n,+)
\end{align*}
where $\Phi^n$ is the set of all nonnegative powers of the golden ratio.
\end{corollary}

\begin{proof}
Phinary numbers are derived from the phinary count (Lemma \ref{lemma: phinary successor}) such that they are sums of Fibonacci subwords of letters 1 and $\psi$. Moreover, powers of the golden ratio occur as proper Fibonacci words, which we defined in Lemma \ref{lemma: p plus power}.  Therefore, according to Lemma \ref{lemma: fib word letters}, the ratio of 1s to $\psi$s in the sum equalling $\phi^n$ is $\frac{F_{n+1}}{F_{n}}$. For example,
\begin{center}
\begin{tabular}{ m{2cm} m{0.5cm} m{0.5cm} |  m{8cm} m{0.5cm} m{0.5cm}}
$ \ \ \ \ \ \ \ \ \ \  \phi^n$ & 1's & $\psi$'s & $ \ \ \ \ \ \ \ \ \ \ \ \ \ \ \ \ \ \ \ \ \ \ \ \ \ \ \ \ \ \ \ \ \ \ \ \ \ \ \ \phi^n$ & 1's & $\psi$'s \\
\hline 
$\phi^0=1$  & $F_1$ & $F_0$ & $\phi^3=1+\psi+1+1+\psi$ & $F_4$ & $F_3$   \\
$\phi^1=1+\psi$  & $F_2$ & $F_1$ & $\phi^4=1+\psi+1+1+\psi+1+\psi +1$ & $F_5$ & $F_4$  \\
$\phi^2=1+\psi +1$ & $F_3$ & $F_2$ & $\phi^5=1+\psi+1+1+\psi+1+\psi +1+1+\psi+1+1+\psi$ & $F_6$ & $F_5$
\end{tabular}
\end{center}

More succinctly,
\begin{align*}
\phi^n = F_{n+1}+F_n\psi.
\end{align*}
Ordinary addition between two positive phinary integers is equivalent to finding the sum of 1's and $\psi$'s in the two associated strings. This sum, however, does not necessarily correspond to a phinary number, as the total number of 1's and $\psi$'s may not correspond to an initial Fibonacci subword. A unique feature of the words associated with the powers of phi is that the strings of any phinary integer greater than such a power will equal the string of the power concatenated with an initial Fibonacci subword. Therefore, when we use phinary addition between two phi powers, we may always choose the smallest associated string of the pair and concatenate it to the end of the largest one; the sum of the resulting string's letters will always equal another phinary number.
\end{proof}

\vspace{5mm}

In this paper, we will be largely concerned with phinary addition in conjunction with ordinary addition and multiplication to demonstrate the applications presented. Therefore, we will discuss some matters with respect to these operations but will not cover the details of other operation overlap.

\vspace{5mm}

\begin{lemma} \label{lemma: nondistributivity}
Ordinary multiplication is not distributive over phinary addition or subtraction.
\begin{align*}
p(q \dagger r) &\neq pq \dagger pr\\
p(q \rightharpoondown r) &\neq pq \rightharpoondown pr
\end{align*}
for all $p, q, r \in \mathbbm{Z}^+_\Phi : n \in \mathbbm{Z}^+, p=\phi^n$.
\end{lemma}

\begin{proof}
By Corollary \ref{corollary: phi mult}, the phinary numbers are closed under multiplication with powers of the golden ratio. However, given a function such as $\phi^n(q \dagger r)$, distributivity is not ensured. For example, let $n=2, q=2\phi+2,$ and $r=\phi+2$, such that
\begin{align*}
\phi^n(q \dagger r) &= \phi^2((2\phi+2) \dagger (\phi+2))\\
&=\phi^2(4\phi+2)\\
&=10\phi+6.
\end{align*}
However, 
\begin{align*}
\phi^2q \dagger \phi^2r &= \phi^2(2\phi+2) \dagger \phi^2(\phi+2)\\
&=(6\phi+4) \dagger (4\phi+3)\\
&=10\phi+7.
\end{align*}
The reason for the discrepancy is most easily understood by viewing how the associated strings for $q$ and $r$ behave under multiplication. The strings are
\begin{equation*}
q \Longleftrightarrow 1 \psi 1 1 \psi 1 \ \ \ \ \ \ \text{and} \ \ \ \ \ \  r: \Longleftrightarrow1 \psi 1 1.
\end{equation*}
As $r$ is the fourth positive phinary integer, $q \dagger r$ results in the fourth successor of $q$, yielding the string
\begin{align*}
q \dagger r \Longleftrightarrow 1 \psi 1 1 \psi 1 \bm{\psi 1 1 \psi}.
\end{align*}
Multiplying by $\phi^2$ yields
\begin{align} \label{eq: phi2q dagger r}
\phi^2(q \dagger r) \Longleftrightarrow \phi^2 \phi \phi^2 \phi^2 \phi \phi^2 \bm{\phi \phi^2 \phi^2 \phi}=1\psi11\psi1\psi11\psi11\psi1\psi1\bm{1\psi1\psi11\psi11\psi}.
\end{align}
Now, we consider multiplying the strings of $q$ and $r$ independently, with
\begin{align*} 
\phi^2q \Longleftrightarrow \phi^2 \phi \phi^2 \phi^2 \phi \phi^2=1\psi11\psi1\psi11\psi11\psi1\psi1
\end{align*}
and
\begin{align*}
\phi^2r \Longleftrightarrow \phi^2 \phi \phi^2 \phi^2=1\psi11\psi1\psi11\psi1.
\end{align*}
As $\phi^2r=4\phi+3$ is the eleventh positive phinary integer, $\phi^2q \dagger \phi^2r$ results in the eleventh successor of $\phi^2q$, such that
\begin{align} \label{eq: phi2q dagger phi2r}
\phi^2q \dagger \phi^2r \Longleftrightarrow 1\psi11\psi1\psi11\psi11\psi1\psi1\bm{1\psi1\psi11\psi11\psi1}.
\end{align}
Comparing Equations \ref{eq: phi2q dagger r} and \ref{eq: phi2q dagger phi2r} demonstrates the unequal values that result from each computation. A proof for phinary subtraction follows the same reasoning.
\end{proof}

\section{Number-Theoretic Applications} \label{section: applications}

We will now turn to some applications of the phinary numbers with respect to number-theoretic trees and recurrence relations.

\subsection{The Calkin-Wilf and  Stern-Brocot Recurrence Trees}

In number theory, two closely-related binary trees exist that share a remarkable property: the vertices of each tree correspond one-for-one to the positive rational numbers. They are known independently as the Stern-Brocot (SB) and Calkin-Wilf (CW) trees (Figures \ref{fig: Calkin-Wilf Tree} and \ref{fig: Stern-Brocot Tree}), and we will refer to them in general as \textit{recurrence trees}. The trees differ only by the order in which the rational numbers appear in the nodes---the Stern-Brocot tree having the special property that its rational numbers appear in order of ascending value from left to right. It was discovered independently in 1858 by Moritz Stern, a number theorist, and in 1861 by Achille Brocot, a French clockmaker, who developed the tree to calculate optimal approximations of complicated gear ratios \cite{brocot1861calcul}. More recently, in 2000, Neil Calkin and Herbert Wilf presented the so-called Calkin-Wilf tree \cite{calkin2000recounting}. However, the tree was introduced as early as 1996, by Jean Berstel and Aldo de Luca as the Raney tree \cite{berstel1997sturmian}, having employed some findings by George N. Raney \cite{Raney:1973aa}. Perhaps, the earliest suggestion of these recurrence trees dates back to 1619 in Johannes Kepler's $Harmonices$ $Mundi$, where he presents a similar construction in studying a relationship between harmonic tones in music and the orbital velocities of the planets \cite{kepler1968harmonices}.

The methods of generating the values in each tree are detailed in Definitions \ref{definition: calkin-wilf} and \ref{definition: stern-brocot}. First, however, we will observe a relationship between these trees and the golden diamond fractal.

\begin{figure}[ht]
   \centering
   \includegraphics[width=\textwidth]{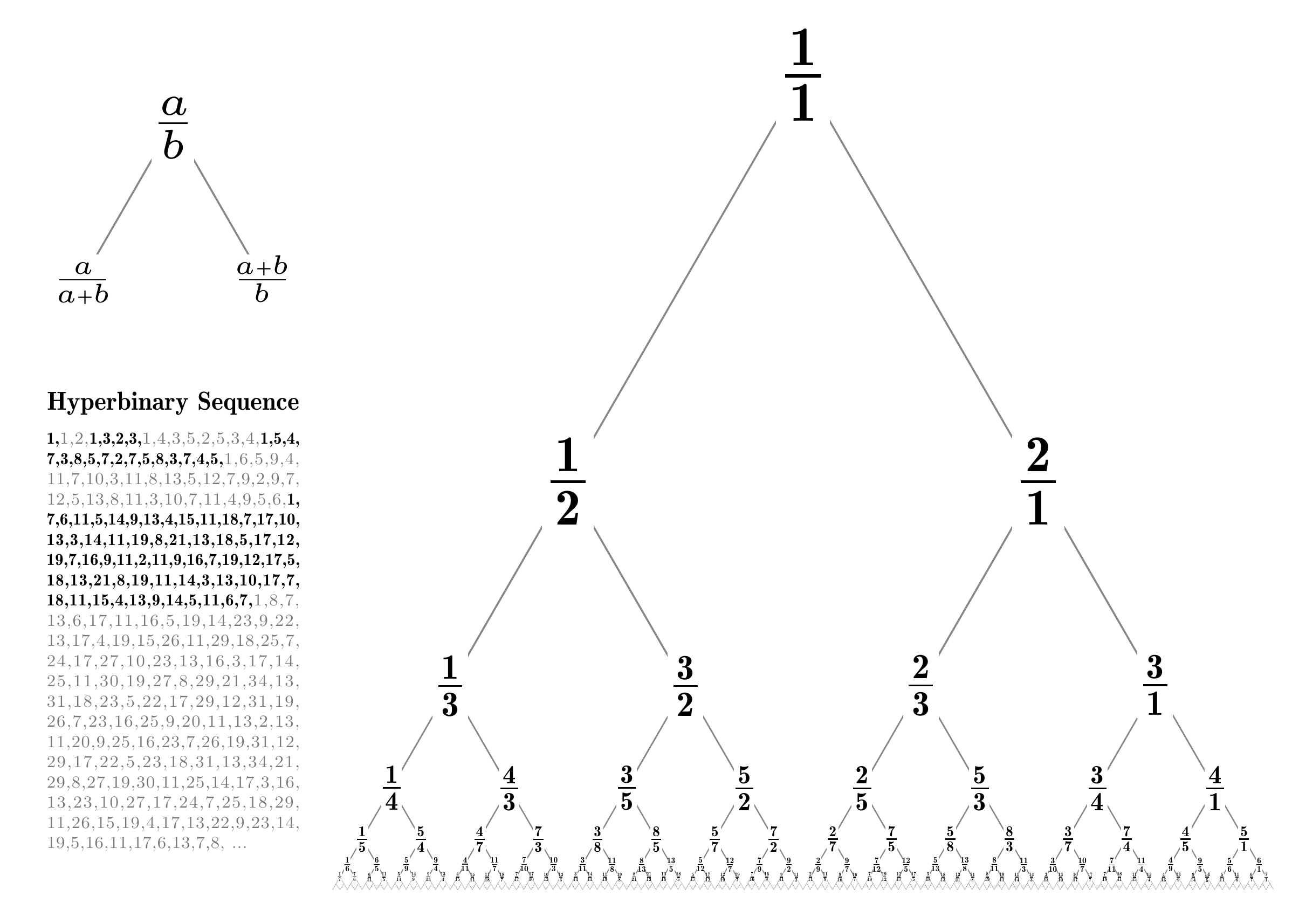}
    \caption{The Calkin-Wilf Tree}
    \label{fig: Calkin-Wilf Tree}
\end{figure}

\begin{figure}[ht]
   \centering
   \includegraphics[width=\textwidth]{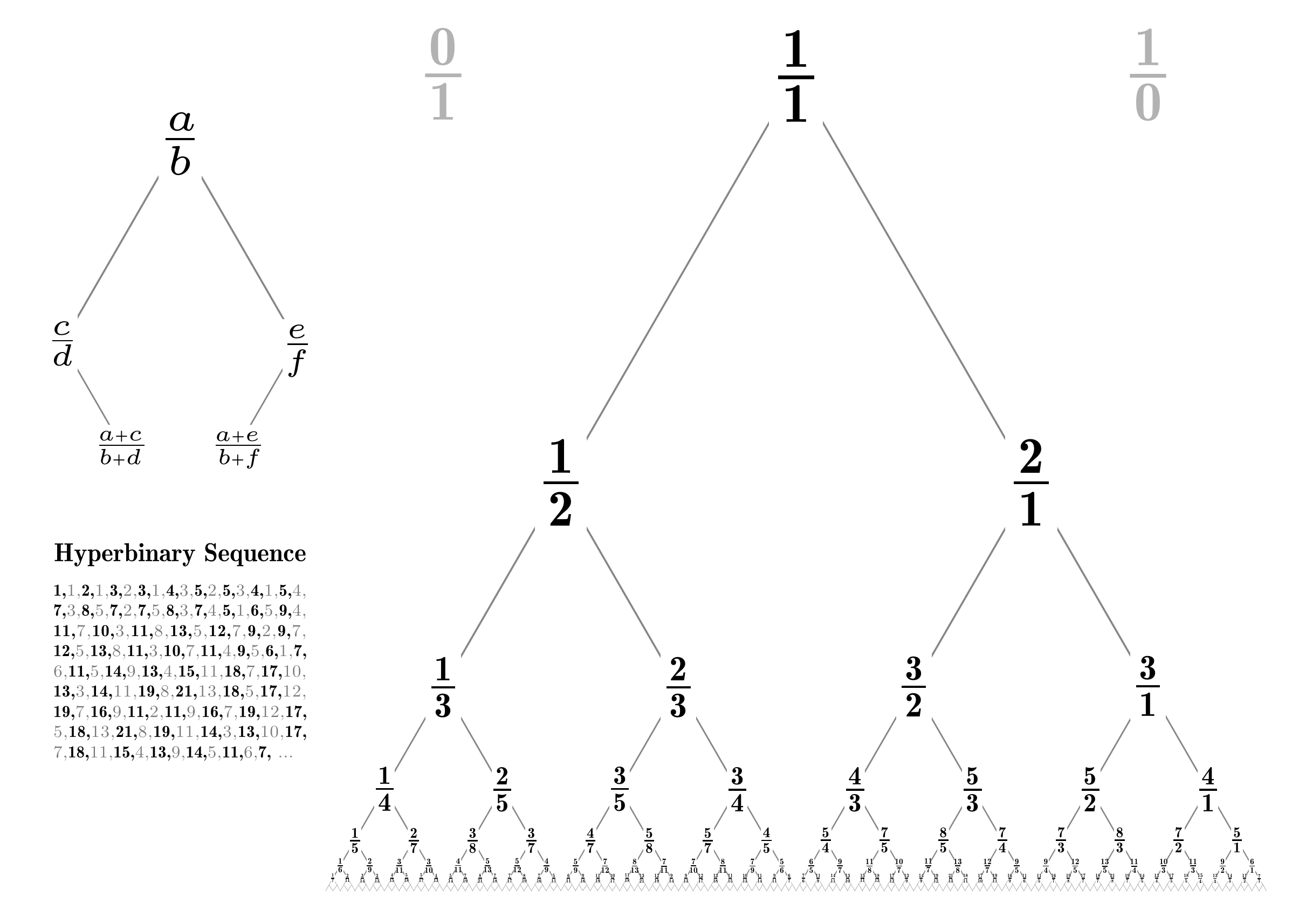}
    \caption{The Stern-Brocot Tree}
    \label{fig: Stern-Brocot Tree}
\end{figure}

\subsection{The Natural Diamond}

Unlike the illustrations in Figures \ref{fig: Calkin-Wilf Tree} and \ref{fig: Stern-Brocot Tree}, the CW and SB trees are typically graphed as simple binary trees without consideration for branch length. However, due to the trees' structure, their geometry is isomorphic to a subset of Sierpinski's Gasket, as asymptotically bifurcating graphs with branches that halve with each generation. When the drawing is inverted and lines drawn to render triangular tiles (Figure \ref{fig: Natural Diamond}), we generate a new fractal, which will be referred to as the natural diamond (ND). Furthermore, we recognize a similar suggestion of perspective as generated by the golden diamond when facet base vertices are connected in the same manner (Figure \ref{fig: Natural Diamond perspective}). Likewise, we can evaluate the resulting projected domain and find a similar grid (Figure {fig: Projective Natural Diamond}) as found in \S \ref{subsection: phinary domain}. In contrast to the Fibonacci word grid of the GD, the natural diamond generates a simple Cartesian grid of uniform intervals. That is, the domain associated by the perspective projection of the ND is the domain of nonnegative integers.

\vspace{5mm}

\begin{figure}[ht]
     \centering
    \begin{subfigure}[b]{0.49\textwidth}
        \includegraphics[width=\textwidth]{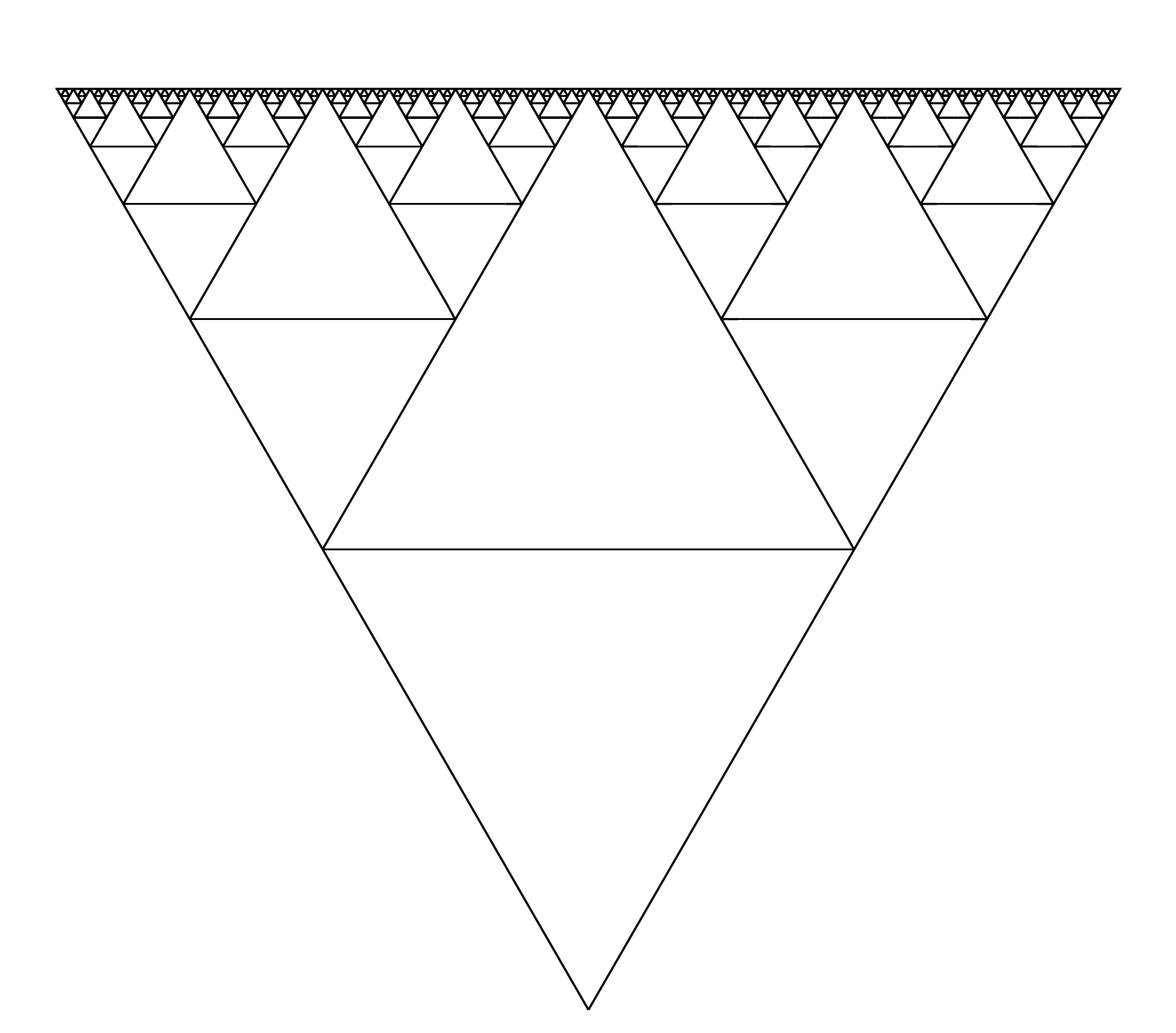}
        \caption{}
        \label{fig: Natural Diamond}
    \end{subfigure}
   %\par
     \begin{subfigure}[b]{0.49\textwidth}
        \includegraphics[width=\textwidth]{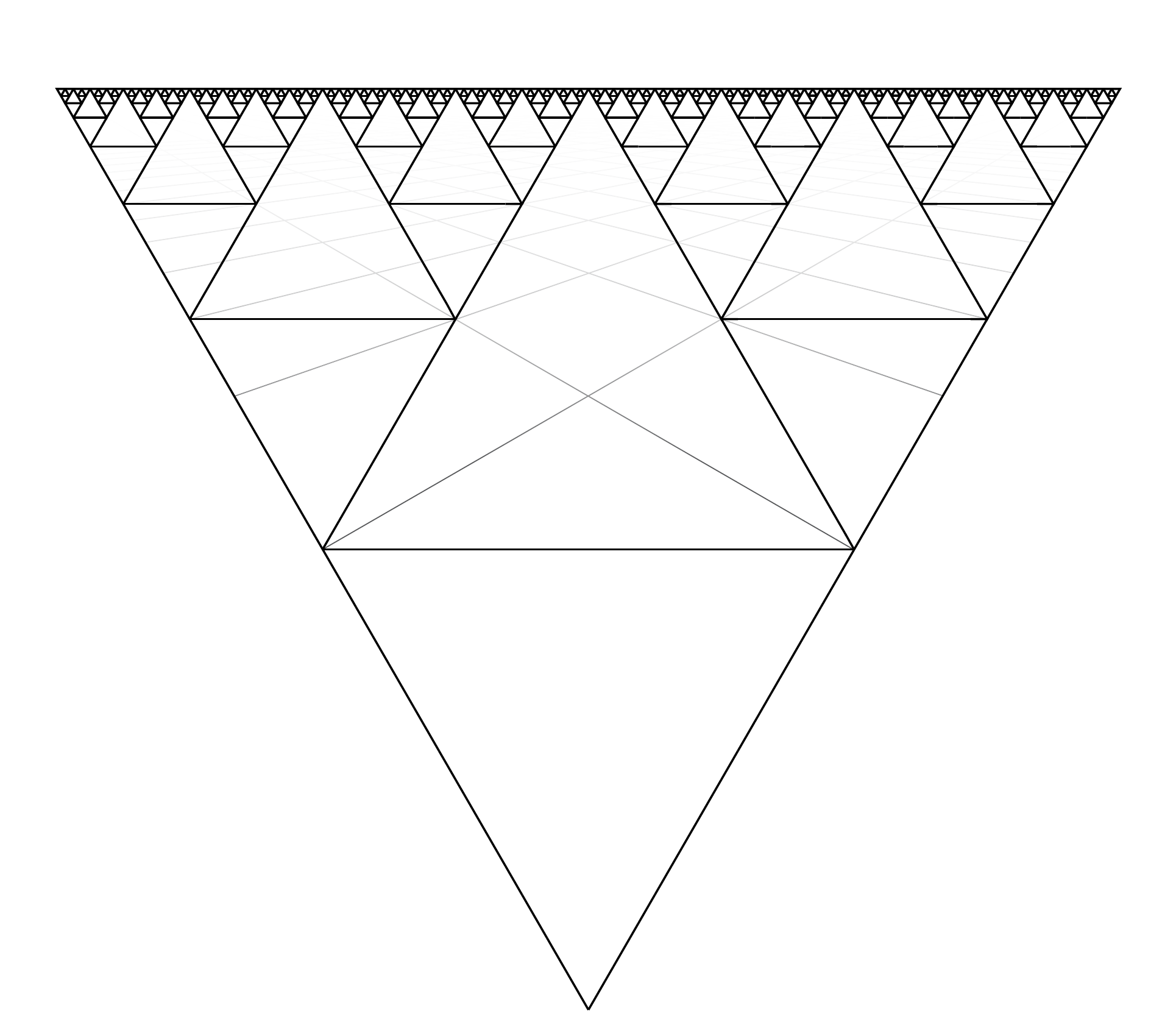}
        \caption{}
        \label{fig: Natural Diamond perspective}
    \end{subfigure}
    \caption{(a) The natural diamond (b) The natural diamond with ``perspective lines"}
\end{figure}

\begin{figure}[!p]
   \centering
   \begin{subfigure}[b]{\textwidth}
   \includegraphics[width=\textwidth]{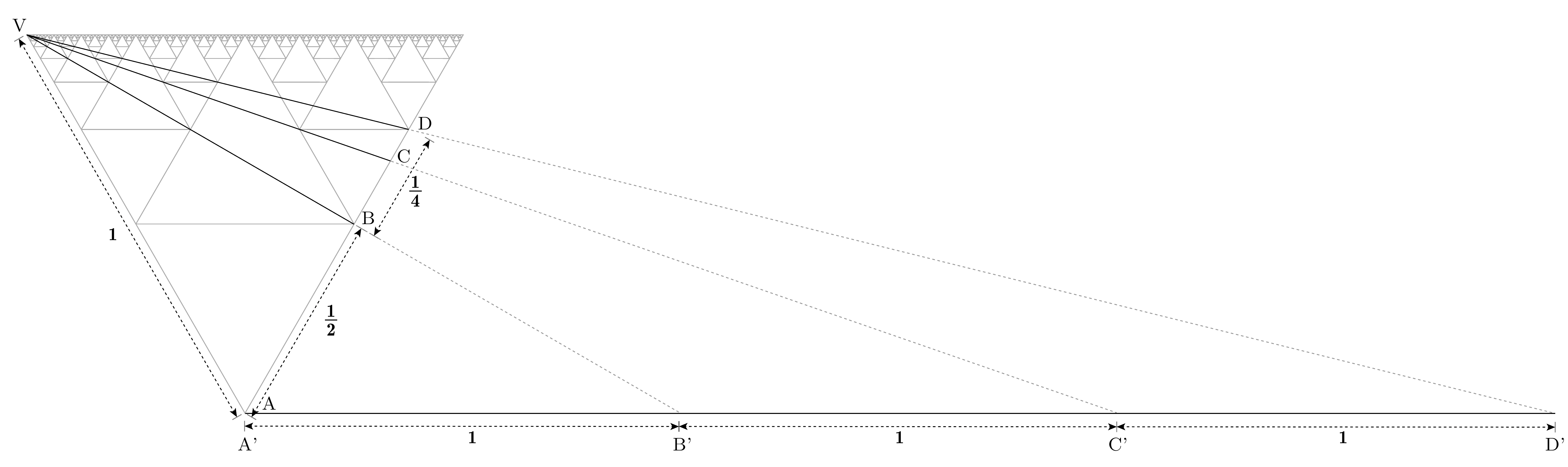}
    \caption{Projective rays VA, VB, VC, and VD produce the first three projected interval markings, AB, BC, and CD. Through the use of the cross-ratio, we can determine the ``true" interval values, A'B', B'C', and B'D', as they occur in the natural number domain.}
    \label{fig: Projective Natural Diamond}
    \end{subfigure}

\begin{subfigure}[b]{\textwidth}
   \centering
   \includegraphics[width=\textwidth]{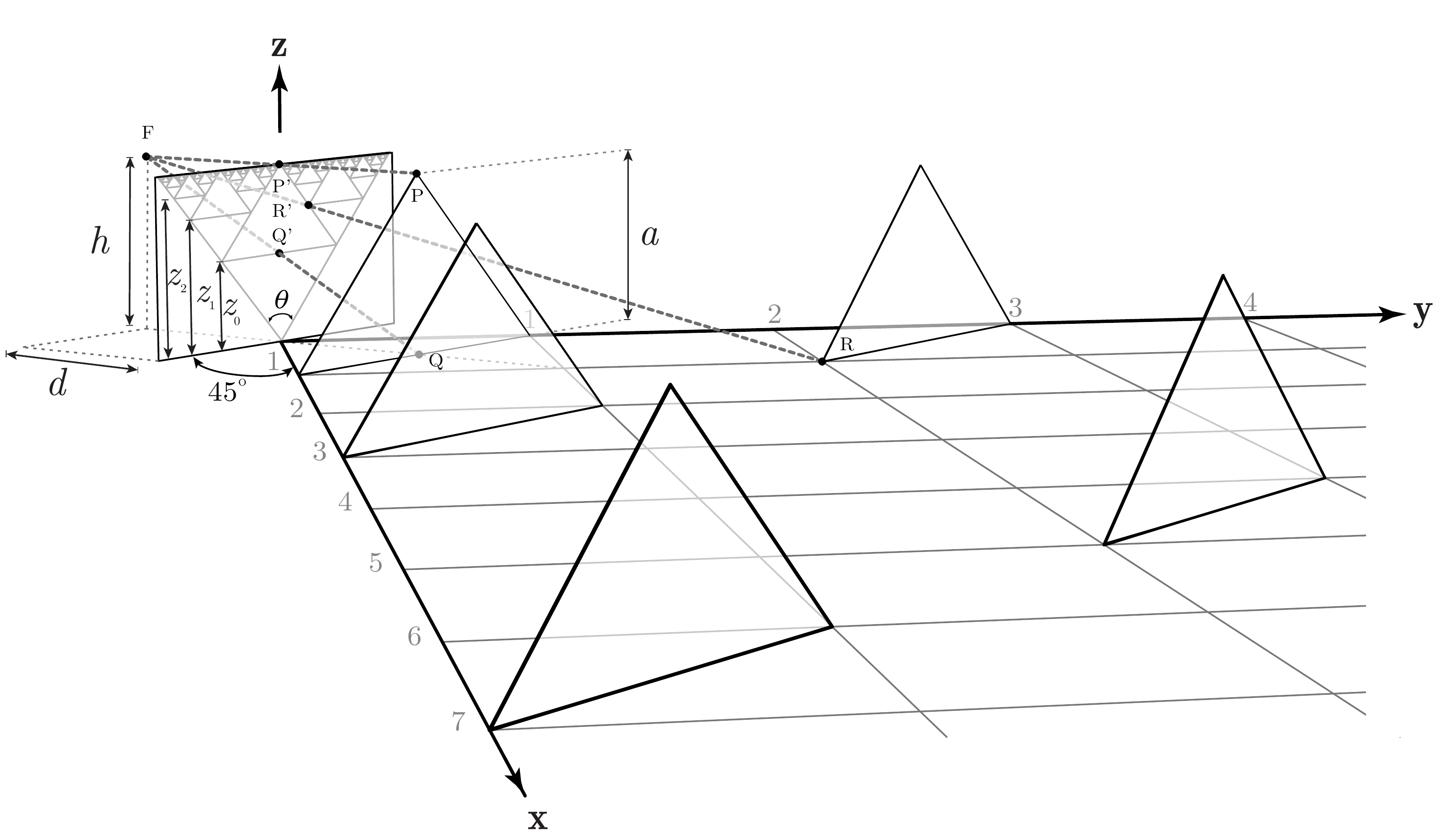}
    \caption{Standing triangles in the first quadrant of the Cartesian plane are projected onto an orthogonal plane, generating the natural diamond via a perspective projection.}
    \label{fig: Image Plane 3D ND pyramids}
    \end{subfigure}
    \caption{}
\end{figure}

\begin{theorem}
The domain from which the ND is a perspective projection is the domain of nonnegative integers $\mathbbm{Z}^*\times \mathbbm{Z}^*$.
\end{theorem}

\begin{proof}
The proof directly parallels that of Lemma \ref{lemma: phinary cross ratio}. See Figure \ref{fig: Projective Natural Diamond} for an illustrated ``proof without words."
\end{proof}

\vspace{5mm}

\begin{lemma}
The number of facet pairs in the $n$th row of the ND equals $2^n$.
\end{lemma}

\begin{proof}
The facets and rows are defined in the same sense as has been done for the GD in \S \ref{subsection: golden diamond}. As each pair of facets is duplicated in the succeeding row, and one pair of facets is in row 1, the lemma follows.
\end{proof}

\vspace{5mm}

The mapped points, corresponding to the tree's vertices, are strikingly reminiscent of the points in the phinary domain as generated by the GD. We present the following property, analogous to Lemma \ref{lemma: GD points}:

\vspace{5mm}

\begin{figure}[!p]
   \centering
   \includegraphics[width=\textwidth]{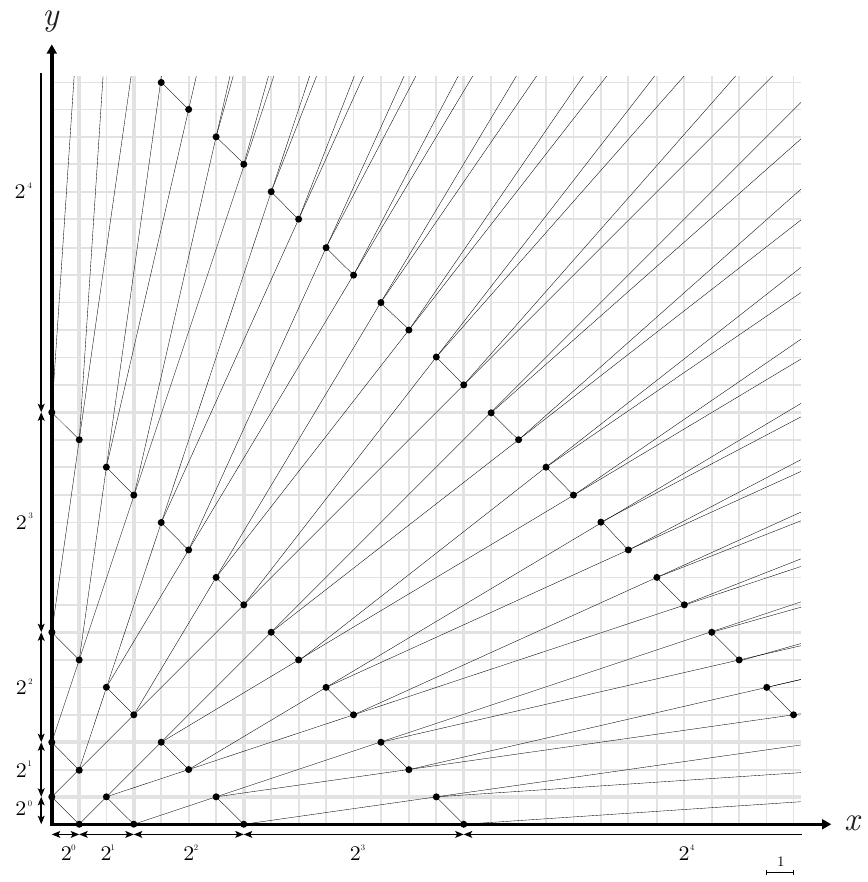}
    \caption{The natural diamond co-image. The base vertices of the standing triangles are indicated by dark points, located at coordinates $(x,y)$, such that $x+y = \sum_{i=0}^{n}2^i$, where $n$ is determined by the diagonal within which the point is located. The lines indicate the ``shadow" of each standing triangle as projected from the perspective focal point.}
    \label{fig: Grid - ND}
\end{figure}

\begin{lemma} \label{lemma: BT points}
The vertices of the standing triangles in the coordinate space have coordinates equal to $(x, y)$ where
\begin{align*}
x+y = \sum_{i=0}^{n}2^i 
\end{align*}
for positive integer $n$.
\end{lemma}

\begin{proof}
See Figure \ref{fig: Grid - ND}.
\end{proof}

\vspace{5mm}

\begin{theorem} \label{theorem: ND perspective projection plane}
The ND is the image of the perspective projection $\mathbbm{R}^3 \to \mathbbm{R}^2$ onto the plane $x+y=0$ through the focal point $(-d_x,-d_y,h)=(-\frac{1}{2},-\frac{1}{2},\frac{1}{\sqrt{2}})$ of an infinite set of isosceles right triangles with altitudes $a=\frac{1}{\sqrt{2}}$, base lengths of $\sqrt{2}$, and base vertices $(u_1,v_1,0)$ and $(u_2,v_2,0)$ with
\begin{align*}
u_1+v_1=u_2+v_2=\sum_{i=0}^{n}2^i
\end{align*}
for all $n\in \mathbbm{Z}, n\geq0$.
\end{theorem}

\begin{proof}
The proof follows in the same manner as for Theorem \ref{theorem: GD perspective projection plane} for the GD projection. See Figure \ref{fig: Image Plane 3D ND pyramids} for a diagram of the projection.

To begin with, we consider the case outlined in Lemma \ref{lemma: reciprocal geometric progression projection}, in which a series of collinear points, spaced via a geometric progression of common ratio $r$ are mapped by a perspective projection onto an image featuring the reciprocal geometric progression with ratio $\frac{1}{r}$. This special case is a good starting point, as the points in the phinary domain spaced by powers of $\phi$ are mapped onto points spaced by powers of $\frac{1}{\phi}$. The lemma defines the focal point at a height above the domain and distance from the projection plane---which we will denote as $h^*$ and $d^*$, respectively---with $h^*=\frac{r}{r-1}$ and $d^*=\frac{1}{r-1}$. However, as can be seen in Figure \ref{fig: Image Plane 3D ND pyramids}, the perspective projection defined in the lemma for $\mathbbm{R}^2 \to \mathbbm{R}$ will correspond to a two-dimensional cross section, in which we map the points from the line $y=x$ in the Cartesian plane onto the $z$-axis. Therefore, we may use $r=2$ as our ratio but must scale the dimensions of all intervals by a factor of $\frac{1}{\sqrt{2}}$, yielding
\begin{align}
h^*&=\frac{r}{\sqrt{2}(r-1)}=\frac{2}{\sqrt{2}(2-1)}=\sqrt{2}\\
d^*&=\frac{1}{\sqrt{2}(r-1)}=\frac{1}{\sqrt{2}(2-1)}=\frac{1}{\sqrt{2}}.
\end{align}
As this is a perspective projection, we are assured that lines in the domain which are parallel to the projection plane will be mapped onto horizontal lines in the image, as desired.

At this stage, we must consider a final condition that has not been addressed by the parameters of Lemma \ref{lemma: reciprocal geometric progression projection}. We require that the apex of each triangle map to a point on the top edge of the ND. For example, referring to the points in Figure \ref{fig: Image Plane 3D ND pyramids}, points P, P', and F should all have height $z=h$. To be consistent with the golden diamond, which had triangles of altitudes $a=\frac{1}{\sqrt{2}}$, we must divide $h^*$ by two. This scaling preserves the common ratio $\frac{1}{2}$ in the ND image while reducing the overall size, yielding
\begin{align*}
h=\frac{h^*}{2}=\frac{1}{\sqrt{2}}.
\end{align*} 
Therefore, we have
\begin{align*}
h&=\frac{1}{\sqrt{2}}\\
d&=\frac{1}{\sqrt{2}}\\
a&=\frac{1}{\sqrt{2}}.
\end{align*} 
This gives $(x,y,z)$-coordinates of the focal point as $F=(-\frac{1}{2},-\frac{1}{2},\frac{1}{\sqrt{2}})$.
The coordinates of the base vertices of the standing triangles come from Lemma \ref{lemma: BT points}, such that the sum of the $x$ and $y$ coordinates of each equals $\sum_{i=0}^{n}2^i$ for the same nonnegative integer $n$ (See Figure \ref{fig: Grid - ND}). By the values defined here, the remaining points of the GD can be easily found to map as desired.
\end{proof}

\vspace{5mm}

We will be utilizing the above properties shortly, but first we return to the recurrence trees for an overview.

\subsection{Recurrence Trees}

We will now explore how to generate the Calkin-Wilf and Stern-Brocot trees via the hyperbinary sequence.

\begin{definition} \label{def: hb}
The hyperbinary sequence\footnote{The hyperbinary sequence is closely related to Stern's diatomic sequence S(n) (often inappropriately called `Stern's diatomic series') \cite[A002487]{sloane2003line}, which differs from the former by the inclusion of an extra zero in front: $S(n)=\{0,1,1,2,1,3,2,3,1,4,3,5,2,5,3,4,1,...\}$. For more on Stern's diatomic sequence see \cite{10.2307/2299356}\cite{calkin2000recounting}\cite{northshield2015three}.} $\text{H}(n)=\{1,1,2,1,3,2,3,1,4,3,5,2,5,3,4,1,...\}$ enumerates the partitions of $n\geq0$ as nonnegative integer powers of two, where each power of two can be used at most twice; any such possible partition is referred to as a hyperbinary representation of $n$. $\text{H}(0)=1$.
\end{definition}

\begin{remark}
Note that if only one instance per power of two were permitted, we would simply recover the binary, or base-two, representation of each number, as each such representation is unique---hence, the hyper- prefix in the term hyperbinary.
\end{remark}

\begin{center}
\begin{table}[ht]
\centering
\begin{tabular}{c c r}
$n$	& $\text{H}(n)$ &	\multicolumn{1}{c}{\textit{Hyperbinary representations of n}} \\
 \hline
 0	&1&	0												\\
 1	&1&	$2^0$											\\
 2	&2&	$2^1$, \ $2^0+2^0$									\\
 3	&1&	$2^1+2^1$										\\
 4	&3&	$2^2$, \ $2^1+2^1$, \ $2^1+2^0+2^0$					\\
 5	&2&	$2^2+2^0$, \ $2^1+2^1+2^0$							\\
 6	&3&	$2^2+2^1$, \ $2^2+2^0+2^0$, \ $2^1+2^1+2^0+2^0$			\\
 7	&1&	$2^2+2^1+2^0$									\\
 8	&4&	$2^3$, \ $2^2+2^2$, \ $2^2+2^1+2^1$, $2^2+2^1+2^1+2^0$	
\end{tabular}
\caption{The hyperbinary representations for $0\leq n\leq8$, where $\text{H}(n)$ enumerates their quantity.}
\end{table}
\end{center}

We can now generate the rational numbers that appear in the aforementioned trees. As the Calkin-Wilf tree is more straightforward to construct, we begin there: We denote a row of the Calkin-Wilf tree as a group of nodes in the same hierarchical level. The numerators within a row are generated by strings from $\text{H}(n)$ that begin with the value 1 and  terminate immediately before the next occurrence of 1, where a string denotes an ordered subset. The strings appear within $\text{H}(n)$ in the same order as the rows---the first string $\{1\}$ corresponding to the top row, the second string $\{1,2\}$ corresponding to the second row, the third string $\{1,3,2,3\}$ corresponding to the third row, and so on. The denominators in each row are generated by the same strings as their numerators, but in reverse order: $\{1\}$ for row one, $\{2,1\}$ for row two, $\{3,2,3,1\}$ for row three, and so on.

\vspace{5mm}

\begin{definition} \label{definition: calkin-wilf} 
The Calkin-Wilf tree is a binary tree of rational numbers $\frac{a}{b}$ generated through the hyperbinary sequence, such that in the $n$th row, the $k$th numerator $a_{n,k}$ and denominator $b_{n,k}$ are generated by 
\begin{align*}
a_{n,k} &= \left \{\text{H}(2^n+k), \ | \ (\mathbbm{Z}^+)^2 \to \mathbbm{Z}^+, \  n,k \in \mathbbm{Z}^+ : 0 \leq k < 2^n, \ n \geq 0 \right \} \\
b_{n,k} &= \left \{\text{H}(2^{n+1}-(k+1))  \ | \  (\mathbbm{Z}^+)^2 \to \mathbbm{Z}^+, \ n,k \in \mathbbm{Z}^+ : 0 \leq k < 2^n, \ n \geq 0 \right \}.
\end{align*}
\end{definition}

\vspace{5mm}

The Stern-Brocot tree, whose rational numbers appear in ascending order, is generated in a similar fashion to the Calkin-Wilf tree. However, instead of using strings from $\text{H}(n)$, we take strings made from the even-indexed elements of $\text{H}(n)$---in other words, from the elements of $\text{H}(2n)=\{1,2,3,3,4,5,5,4,5,7,8,...\}$. This time, the strings always begin with $\text{H}(0)$ but double in length for each successive row. Again, the numerators correspond to each value of the string, with $\{1\}$ for row one, $\{1,2\}$ for row two, $\{1,2,3,3\}$ for row three, and so on. Additionally, the denominators are generated in a similar fashion by reversing the order of each string. We will summarize the above construction as follows: 

\vspace{5mm}

\begin{definition} \label{definition: stern-brocot}
The Stern-Brocot tree is a binary tree of rational numbers $\frac{a}{b}$ generated through the even-indexed elements of the hyperbinary sequence, $\text{H}(2n)$, such that in the $n$th row, the $k$th numerator $a_{n,k}$ and denominator $b_{n,k}$ are generated by 
\begin{align*}
a_{n,k} &= \left \{\text{H}(2k) \ | \ \mathbbm{Z}^+ \to \mathbbm{Z}^+, \ n,k \in \mathbbm{Z}^+ : 0 \leq k < 2^n, \ n \geq 0 \right \} \\
b_{n,k} &= \left \{\text{H}(2^{n+1}-2(k+1)) \ | \ (\mathbbm{Z}^+)^2 \to \mathbbm{Z}^+, \ n,k \in \mathbbm{Z}^+ : 0 \leq k < 2^n, \ n \geq 0 \right \}.
\end{align*}
\end{definition}

\vspace{5mm}

Other methods of constructing the SB and CW trees are possible, although these details are outside the scope of this paper (See \cite{BATES20101637}).

\subsection{The Hyperbinary Sequence}

The hyperbinary sequence is a fascinating object in its own right, and we must take some time to define it properly and describe a few of its properties.

\vspace{5mm}

\begin{lemma} \label{lemma: odd hb}
The hyperbinary sequence is equal to its own proper subset, the sequence of odd-indexed elements.
 \begin{align*}
 \text{H}(2n+1) &= \text{H}(n)
 \end{align*}
 for $n \in \mathbbm{Z}^*$.
\end{lemma}

\begin{proof}
Let us denote a hyperbinary representation of $n \in \mathbbm{Z}^+$ in a sort of binary notation, meaning the radix is two, or equivalently, each place value, or bit, represents a power of two---the modification being that, in addition to a 0 or 1, a bit may assume a 2, as well. For example, the representations of $n=5$ are $101_{2^+}$ and $21_{2^+}$. Let $q=2n+1$ be an odd, positive integer. In our notation, every $q$ ends in a 1. Therefore, the representations of $q$ are concerned with combinations of the most significant bits before the final value of 1. It follows that the number of representations for $q$ is the same as the number of representations for $\frac{q-1}{2}$, which is what we set out to prove.
\end{proof}

\vspace{5mm}

\begin{lemma} \label{lemma: even hb}
The even-indexed elements of the hyperbinary sequence greater than zero are equal to the sum of two previous values in the sequence. In particular,
\begin{align*}
 \text{H}(2n+2) = \text{H}(n)+\text{H}(n+1)
\end{align*}
for $n \in \mathbbm{Z}^*$.
\end{lemma}

\begin{proof}
As in the preceding proof, we consider hyperbinary representations in modified binary notation. We note that all even-indexed values $r=2n+2$ have hyperbinary representations ending in either 0 or 2. In a similar observation as before, we find that all the representations ending in 0 are equal in quantity to the representations with the 0 removed, namely $\frac{r}{2}$. Likewise, the representations ending in 2 are equal in quantity to the representations of $\frac{r-2}{2}$. Therefore, $\text{H}(r)$ is equal to the sum of $\text{H}(\frac{r}{2})$ and $\text{H}(\frac{r-2}{2})$. 
\end{proof}

\vspace{5mm}

\begin{theorem} \label{theorem: recurrence hb}
The hyperbinary sequence is defined by the recurrence relation
\begin{align}
\text{H}(0) &= 1\\
\text{H}(2n+1) &= \text{H}(n) \label{eq: odd hb}\\
\text{H}(2n+2) &= \text{H}(n)+\text{H}(n+1)
\end{align}
for $n \in \mathbbm{Z}^*$.
\end{theorem}

\begin{proof}
By Definition \ref{def: hb}, Lemma \ref{lemma: odd hb}, and Lemma \ref{lemma: even hb}, we find that all values of $\text{H}$ can be accounted for.
\end{proof}

\vspace{5mm}

In light of the first numerator in each row of the Calkin-Wilf tree having a value of one, it is easy to recognize that $\text{H}(2^k-1)=1$ for all $k \in \mathbbm{Z}^*$. As it turns out, this is a special case of a more general property. Brent Yorgey \cite{YorgeyMathLessTravelledMorehyperbinaryfun} proved the following result.

\vspace{5mm}

\begin{theorem} \label{theorem: odd reoccurrence}
All odd-indexed values of the hyperbinary sequence are reoccurrences of previous even-indexed values.
\begin{align*}
\text{H}(p \cdot 2^k-1) = \{ \text{H}(p-1) \ | \ \forall n, k \in \mathbbm{Z}^* : p=2n+1 \}
\end{align*}
\end{theorem}

\begin{proof}
For $n \in \mathbbm{Z}^*$, we have odd $p=2n+1$ and therefore, $q=p \cdot 2^k-1$ is also odd. For a given value of $p$, we know by Equation \ref{eq: odd hb} in Theorem \ref{theorem: recurrence hb} that
\begin{align*}
\text{H}(2(p \cdot 2^k-1)+1)&=\text{H}(p \cdot 2^k-1)
\end{align*}
which we can write as
\begin{align*}
\text{H}(p \cdot 2^{k+1}-1)&=\text{H}(p \cdot 2^k-1).
\end{align*}
This says that for a particular value of $p$, $\text{H}(p \cdot 2^k-1)$ is the same for all $k \in \mathbbm{Z}^*$. Therefore, If we choose $k=0$, $p \cdot 2^k-1$ becomes even---namely, $p-1$. Therefore, by induction, $\text{H}(p \cdot 2^k-1)=\text{H}(p-1)$.
\end{proof}

\begin{remark}
Yorgey noted that although each value in the hyperbinary sequence occurs an infinite number of times, Theorem \ref{theorem: odd reoccurrence} implies that all odd-indexed elements of the hyperbinary sequence are akin to ``copies" of earlier even-indexed elements of the form $p-1$. Moreover, any given value occurs only finitely often in even-indexed elements. Therefore, Yorgey refers to these even-indexed occurrences of hyperbinary values fittingly as ``primary occurrences."
\end{remark}

\begin{example}
The value of 2 occurs infinitely often in $\text{H}$, corresponding to $p=3$, such that $\text{H}(3\cdot 2^k-1) = 2$ for all $k \in \mathbbm{Z}^*$. However, the primary occurrence of 2 is singular at  $\text{H}(p-1) = \text{H}(2)$.
\end{example}

\begin{remark}
Note, not all primary occurrences are singular. For example, the value of 3 originates from two separate primary occurrences at $\text{H}(4)$ and $\text{H}(6)$, for $p=5$ and $p=6$, respectively. It has been conjectured by Yorgey that the number of primary occurrences of $n$ equals Euler's totient function $\Phi(n)$.
\end{remark}

\subsection{The Fibonacci Diatomic Sequence} \label{subsection: The Fibonacci Diatomic Sequence}

In order to connect the above to the phinary ordinals, we now turn to another well-studied sequence, the Fibonacci diatomic sequence \cite[A000119]{sloane2003line}\cite{bicknell2004fibonacci}\cite{stockmeyer2008smooth}\cite{bicknell1999number}.

\begin{definition}
The Fibonacci diatomic sequence $\text{F}(n)=\{1,1,1,2,1,2,2,1,3,2,2,3,1,3,3,2,4,2,3,3,1,...\}$ enumerates the partitions of $n\geq0$ as Fibonacci numbers $\text{F}_k$ for $k>1$,\footnote{Note that both $\text{F}_1$ and $\text{F}_2$ equal one. For this reason, we disallow representations of $n$ with $\text{F}_1$.} where each value is used no more than once. Any such possible partition is referred to as a Fibonacci representation of $n$. $\text{F}(0)=1$.
\end{definition}

\begin{example}
$\text{F}(11)=3$, because $n=11$ can be written in any of the three following ways: $\text{F}_6+\text{F}_4=8+3$ or $\text{F}_6+\text{F}_3+\text{F}_2=8+2+1$ or $\text{F}_5+\text{F}_4+\text{F}_3+\text{F}_2=5+3+2+1$.
\end{example}

\vspace{5mm}

\begin{table}
\centering
\begin{tabular}{c c r r}
$n$	&	$\text{F}(n)$	&	\multicolumn{1}{c}{\textit{Fibonacci diatomic representations of n}}	&	\multicolumn{1}{c}{\textit{Zeckendorf representations of n}}  \\
 \hline
 0	&1&	0												&	0\\
 1	&1&	$F_2$											&	$F_2$\\
 2	&1&	$F_3$											&	$F_3$\\
 3	&2&	$F_4$, \ $F_3+F_2$									&	$F_4$\\
 4	&1&	$F_4+F_2	$										&	$F_4+F_2	$\\
 5	&1&	$F_5$, \ $F_4+F_3$									&	$F_5$\\
 6	&2&	$F_5+F_2$, \ $F_4+F_3+F_2$							&	$F_5+F_2$\\
 7	&1&	$F_5+F_3$										&	$F_5+F_3$\\
 8	&3&	$F_6$, \ $F_5+F_4$, \ $F_5+F_3+F_2$					&	$F_6$
\end{tabular}
\caption{The Fibonacci diatomic representations for $0\leq n\leq8$, where $\text{F}(n)$ enumerates their quantity and $F_n$ is the $n$th Fibonacci number. Zeckendorf representations constitute the partitions of $n$ into a unique sum of nonconsecutive Fibonacci numbers.}
\end{table}

Several properties of the Fibonacci diatomic sequence have been identified, a few of which will be discussed here. For more, see work by Marjorie Bicknell-Johnson, who has done extensive research on the subject (See, for example, \cite{bicknell1999number}\cite{bicknell2004fibonacci}).

\vspace{5mm}

As in the hyperbinary representations of a number, Fibonacci representations of a number are not unique. However, Gerrit Lekkerkerker \cite{lekkerkerker1951voorstelling} and later Edouard Zeckendorf \cite{zeckendorf1972representations} proved that unique Fibonacci representations of each number can be made if the following requirement is imposed; these are known in the literature as Zeckendorf representations.

\vspace{5mm}

\begin{theorem} [Lekkerkerker-Zeckendorf] \label{theorem: zeckendorf}
A Zeckendorf representation of $n$ is a Fibonacci representation such that no consecutive values of $\text{F}_k$ are used. For every $n$, only one Zeckendorf representation exists and it is unique.
\end{theorem}

\begin{example}
The Zeckendorf representation of 11 is $\text{F}_6+\text{F}_4=8+3$. Notice, the other two Fibonacci representations from the previous example are not permitted as they each contain consecutive Fibonacci numbers.
\end{example}

\vspace{5mm}

Donald E. Knuth \cite{KNUTH198857} discovered an interesting operation, known as Fibonacci multiplication, which can be performed upon Zeckendorf representations of $\mathbbm{Z}^+$---an operation which is surprisingly both associative and commutative.

\vspace{5mm}

\begin{theorem} [Fibonacci Multiplication] \label{theorem: Fibonacci multiplication}
We denote the operation $(\mathbbm{Z}^+, \circ)$ as Fibonacci multiplication where
\begin{align*}
a \circ b = \left \{\sum_{i=0}^k\sum_{j=0}^l \text{F}_{c_i+d_j} \ | \ a, b \in \mathbbm{Z}^+ : a=\sum_{i=0}^k \text{F}_{c_i}, \ b=\sum_{j=0}^l \text{F}_{d_j}, \ \forall c_i, d_j \geq 2 \right \}
\end{align*}
with $a$ and $b$ in Zeckendorf representations.
\end{theorem}
 
\begin{example}
\begin{align*}
4 \circ 7 &= (\text{F}_4+\text{F}_2) \circ (\text{F}_5+\text{F}_3) \\
&= \text{F}_{4+5}+\text{F}_{4+3}+\text{F}_{2+5}+\text{F}_{2+3} \\
&=  \text{F}_{9}+2\text{F}_{7}+\text{F}_{5} \\
&= 34+26+5 \\
&= 65
\end{align*}
\begin{align*}
7 \circ 4 &=  (\text{F}_5+\text{F}_3) \circ (\text{F}_4+\text{F}_2) \\
&= \text{F}_{5+4}+\text{F}_{5+2}+\text{F}_{3+4}+\text{F}_{3+2} \\
&=  \text{F}_{9}+2\text{F}_{7}+\text{F}_{5} \\
&= 34+26+5 \\
&= 65
\end{align*}
\end{example}

\begin{remark}
Note that Fibonacci multiplication is not equivalent to standard multiplication, as is clearly seen in the above example. We will return to this subject in the following section.
\end{remark}

\subsection{The Hyperphinary Sequence} \label{subsection: The Hyperphinary Sequence}

 We can now ask a similar question of the phinary ordinals as was posed by the hyperbinary sequence of the positive integers: how many partitions of powers of the golden ratio can represent a phinary number? The answer is perhaps surprising.

\begin{definition} \label{def: hp}
The hyperphinary sequence $\text{H}_\Phi(p)=\{1,1,1,2,1,2,2,1,3,2,2,3,1,3,3,2,4,2,3,3,1,...\}$ enumerates the partitions of $p \in \mathbbm{Z}^*_{\Phi}$, $p\geq0$ as non-negative integer powers of the golden ratio, $\phi^0, \phi^1, \phi^2, \phi^3, \phi^4, \phi^5, ...$ where each value is used no more than once. Any such possible partition is referred to as a hyperphinary representation of $p$. $\text{H}_\Phi(0)=1$.
\end{definition}

\begin{table}[ht]
\centering
\begin{tabular}{c c r r}
$p$	&	$\text{H}_\Phi(p)$	&	\multicolumn{1}{c}{\textit{Hyperphinary representations of p}}	&	\multicolumn{1}{c}{\textit{Phinary representations of p}}  \\
 \hline
 0			&1&	0												&	0\\
 1			&1&	$\phi^0$											&	$\phi^0$\\
 $\phi$		&1&	$\phi^1$											&	$\phi^1$\\
 $\phi+1$		&2&	$\phi^2$, \ $\phi^1+\phi^0$							&	$\phi^2$\\
 $\phi+2$		&1&	$\phi^2+\phi^0$										&	$\phi^2+\phi^0$	\\
 $2\phi+1$	&1&	$\phi^3$, \ $\phi^2+\phi^1$							&	$\phi^3$\\
 $2\phi+2$	&2&	$\phi^3+\phi^0$, \ $\phi^2+\phi^1+\phi^0$					&	$\phi^3+\phi^0$\\
 $3\phi+1$	&1&	$\phi^3+\phi^1$										&	$\phi^3+\phi^1$	\\
 $3\phi+2$	&3&	$\phi^4$, \ $\phi^3+\phi^2$, \ $\phi^3+\phi^1+\phi^0$			&	$\phi^4$
\end{tabular}
\caption{The hyperphinary representations for $0\leq p\leq 3\phi+2$, where $\text{H}_\Phi(p)$ enumerates their quantity. Phinary representations constitute the partitions of $p$ into a unique sum of nonconsecutive powers of phi.}
\end{table}

We can make the following observation.

\vspace{5mm}

\begin{theorem} \label{theorem: hyperphinary - fibonacci diatomic equiv}
The hyperphinary sequence equals the Fibonacci diatomic sequence.
\begin{align}
\text{H}_\Phi(p)=\text{F}(n)
\end{align}
for $p\in \mathbbm{Z}^*_{\Phi}$ and $n \in \mathbbm{Z}^*$.
\end{theorem}

\begin{proof}
From Theorem \ref{theorem: fib phi iso}, the powers of the golden ratio with respect to the phinary numbers are isomorphic to the Fibonacci numbers with respect to the positive integers. Each phinary number equals a sum of powers of phi, or equivalently, a sum of phinary Fibonacci numbers. For this reason, the number of possible sums of powers of phi that equal a given phinary number is the same as the number of possible sums of Fibonacci numbers that equal a positive integer.
\end{proof}

\vspace{5mm}

As with the Fibonacci representations of the positive integers, the hyperphinary representations of the phinary ordinals are not unique; however, as with the former, phinary numbers can assume a Zeckendorf-type of representation. In Bergman's base-phi representation, this is known as \textit{standard form}, as described in \S \ref{subsection: phinary numbers}.

We are going to expand the definition of the Fibonacci diatomic sequence. Instead of defining the sequence as the number of partitions of the positive integers into Fibonacci numbers, we will make a small but meaningful change.

\vspace{5mm}

\begin{proposition}
The Fibonacci diatomic sequence enumerates the number of partitions of elements in an ordinal set into its Fibonacci numbers. Therefore, the sequence may be indexed by any ordinal set that contains native Fibonacci numbers.
\begin{align*}
\text{F}(n)=\text{F}(p)
\end{align*}
where $n \in \mathbbm{Z}^*$ and $p \in \mathbbm{Z}^*_\Phi$.
\end{proposition}

\begin{proof}
	This follows from Theorem \ref{theorem: hyperphinary - fibonacci diatomic equiv}, which shows that $\text{H}_\Phi(p)=\text{F}(n)$.
\end{proof}

\begin{remark}
The emphasis here is that the Fibonacci diatomic sequence evaluates the number of partitions on \textit{any} set of ordinal numbers with respect to the Fibonacci numbers defined \textit{on those numbers}.
\end{remark}

\vspace{5mm}

An emergent property of standard form representations of phinary numbers is their ability to be multiplied in an analogous manner to the Fibonacci multiplication of the Zeckendorf representations of $\mathbbm{Z}^+$, shown in Theorem \ref{theorem: Fibonacci multiplication}.

\vspace{5mm}

\begin{theorem} \label{theorem: phinary multiplication}
Let the operation $(\mathbbm{Z}^+_{\Phi}, \circ)$ be the analog of Fibonacci multiplication applied to phinary numbers, where
\begin{align*}
p \circ q = \left \{\sum_{i=0}^k\sum_{j=0}^l \phi^{c_i+d_j} \ | \ p, q \in \mathbbm{Z}^+_{\Phi} : p=\sum_{i=0}^k \phi^{c_i}, \ q=\sum_{j=0}^l \phi^{d_j}, \ \forall c_i, d_j \in \mathbbm{Z}^* \right \}
\end{align*}
with $p$ and $q$ in standard form representations.
\end{theorem} 

\begin{example}
\begin{align*}
(\phi + 2) \circ (3\phi + 1) &= (\phi^2+\phi^0) \circ (\phi^3+\phi^1) \\
&= \phi^{2+3}+\phi^{2+1} + \phi^{0+3}+\phi^{0+1} \\
&= \phi^{5}+2\phi^{3}+\phi^{1} \\
&= (5\phi+3)+(4\phi+2)+\phi \\
&= 10\phi+5.
\end{align*}
\end{example}

\vspace{5mm}

It is immediately clear that the operation defined above (Theorem \ref{theorem: phinary multiplication}) is simply equivalent to standard multiplication on phinary numbers, as it is the definition of the product rule for exponentials of like-base. This contrasts with Fibonacci multiplication on $\mathbbm{Z}^+$ (Theorem \ref{theorem: Fibonacci multiplication}), which although a commutative/associative operation, is not equivalent to a standard product on the positive integers. The above product, however, does not ensure closure on phinary numbers as discussed in \S \ref{section: operations}.

\vspace{5mm}

\begin{theorem}
The $\circ$ operation is isomorphic to natural multiplication on the phinary ordinals.
\begin{align*}
(\mathbbm{Z}^+_{\Phi}, \circ) \cong (\mathbbm{Z}^+_{\Phi}, \cdot)
\end{align*}
\end{theorem}

\begin{proof}
The $\circ$ operation on powers of phi is equivalent to the product rule for exponentials.
\begin{align}
\phi^a \circ \phi^b = \phi^{a+b}
\end{align}
The theorem holds, because any phinary number can be written as a unique sum of phi powers (Theorem \ref{theorem: sums of phi pow}); the phinary powers are isomorphic to Fibonacci numbers in $\mathbbm{Z}^+_{\Phi}$ (Theorem \ref{theorem: fib phi iso}); and the $\circ$ operation is commutative and associative on Fibonacci numbers \cite{KNUTH198857}.
\end{proof}

\begin{example}
\begin{align*}
(\phi + 2)(3\phi + 1) &= 3\phi^2+7\phi+2 \\
&= (3\phi+3)+7\phi+2 \\
&= 10\phi+5
\end{align*}
This is the same result as found in the previous example.
\end{example}

\subsection{The Phinary Recurrence Trees} \label{The Nested CW and SB Trees}

\begin{figure}[!p]
    \centering
    \includegraphics[width=\textwidth,height=\textheight,keepaspectratio]{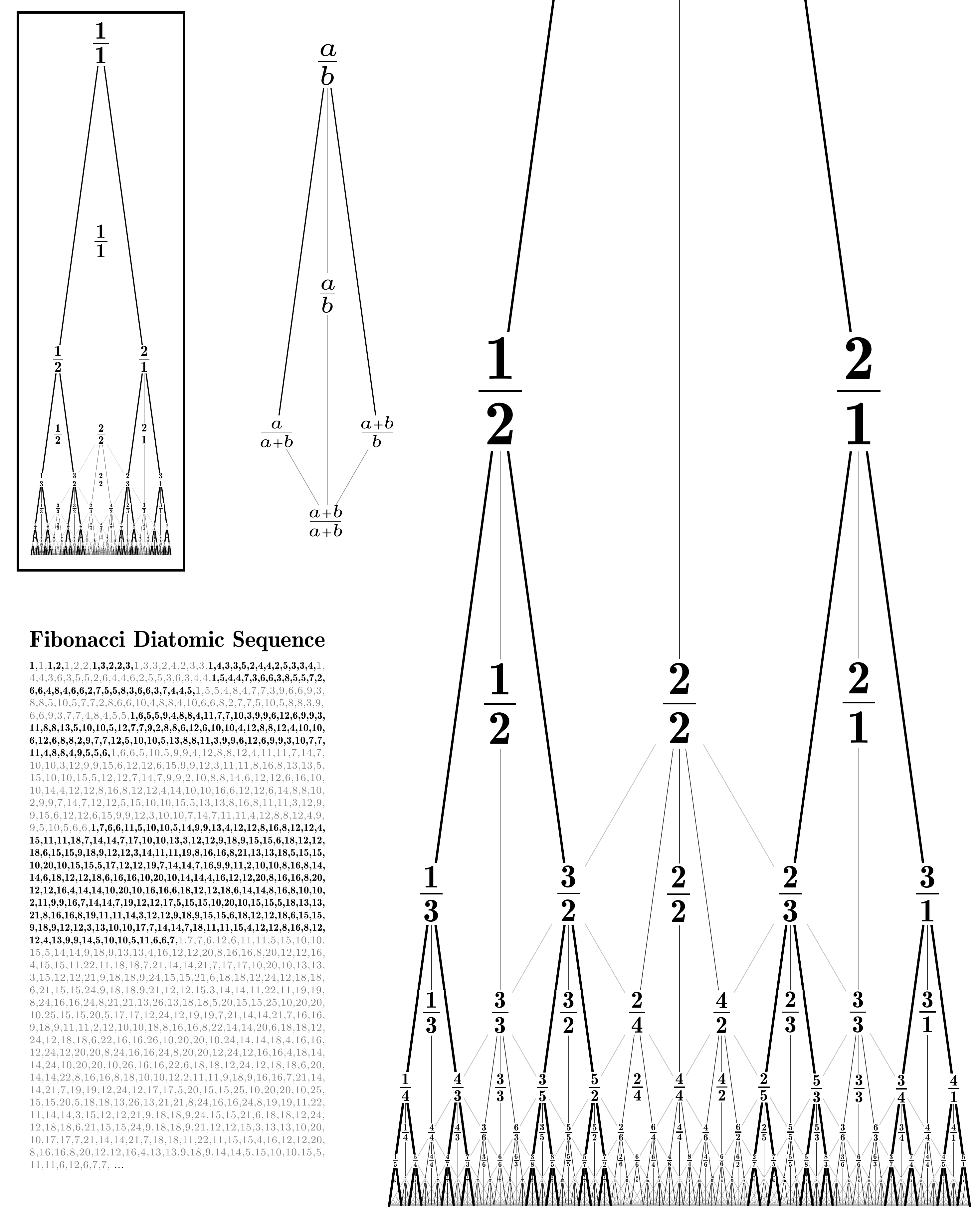}
    \caption{The phinary recurrence tree, of which the Calkin-Wilf tree is a subtree---indicated by darkened branches. Compare to Figure \ref{fig: Calkin-Wilf Tree} for the Calkin-Wilf tree.}
    \label{fig: Expanded CW Tree}
\end{figure}

\begin{figure}[!p]
    \centering
    \includegraphics[width=\textwidth,height=\textheight,keepaspectratio]{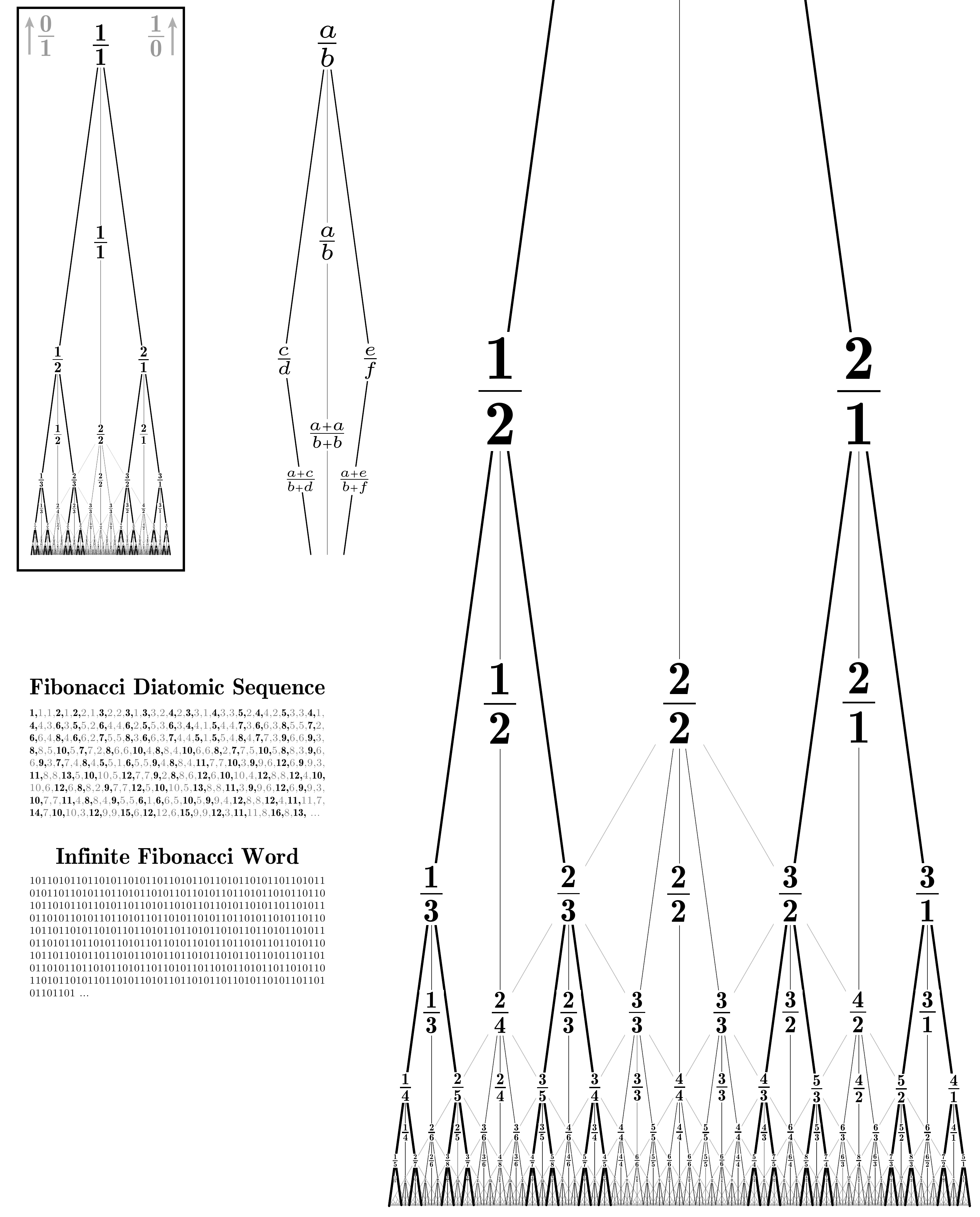}
    \caption{The phinary even-recurrence tree, of which the Stern-Brocot tree is a subtree---indicated by darkened branches. Compare to Figure \ref{fig: Stern-Brocot Tree} for the Stern-Brocot tree.}
    \label{fig: Expanded SB Tree}
\end{figure}

In light of the phinary numbers, their hyperphinary representation, and the associated golden diamond fractal, which could be viewed as a tree graph, one naturally wonders whether an analogue to the Calkin-Wilf or Stern-Brocot may exist. Remarkably, two analogous trees can indeed be generated in the same manner as those corresponding to the positive integers. Curiously, however, the values or nodes of the phinary trees occupy a space beyond the confines of the tree itself. Moreover, the resulting set of rational numbers contains infinitely many copies of all the rational numbers in every possible unreduced form. We will refer to these new trees as the phinary recurrence and even-recurrence trees, which are shown in Figures \ref{fig: Expanded CW Tree} and \ref{fig: Expanded SB Tree}.\footnote{At this stage of the author's research, it was found that Sam Northshield had earlier noticed a connection between the Fibonacci diatomic sequence and a tree he referred to as the hyperbolic graph $\text{S}_{2,3}$\cite{article:NorthshieldSam}\cite{northshield2015three}, which is the same tree resulting from the golden diamond. However, he did not use the sequence to generate an analogue to the Calkin-Wilf and Stern-Brocot trees, as we demonstrate here. To the author's knowledge, this is the only other reference to this geometry. Northshield notes the occurrence of the golden ratio in the lengths of the graph's rectangles.} 

In order to define the elements of the phinary recurrence tree in a closed-form equation, we will use the nonnegative phinary integers $\mathbbm{Z}^*_\Phi$ instead of the nonnegative ordinary integers to index the domain of the Fibonacci diatomic sequence. For example, we had $\text{F}(0)=1, \ \text{F}(1)=1, \ \text{F}(2)=2, \ \text{F}(3)=1, \ \text{F}(4)=2, ...$ in $\mathbbm{Z}^*$; equivalently, we now have $\text{F}(0)=1, \ \text{F}(1)=1, \ \text{F}(\phi)=2, \ \text{F}(\phi+1)=1, \ \text{F}(\phi+2)=2, ...$ in $\mathbbm{Z}^*_\Phi$. 
\begin{definition} \label{definition: ecw}
The phinary recurrence tree (PRT) is a tree of rational numbers $\frac{a}{b}$ generated through the Fibonacci diatomic sequence, such that in the $n$th row, the $p$th numerator $a_{n,p}$ and denominator $b_{n,p}$ are given by 
\begin{align*}
a_{n,p} &= \left \{\text{F}(\phi^n \dagger p), \ (\mathbbm{Z}^*_\Phi)^2 \to \mathbbm{Z}^+ \ | \ n \in \mathbbm{Z}^*, \ p \in \mathbbm{Z}^*_\Phi : \ 0 \leq p < \phi^n \right \} \\
b_{n,p} &= \left \{\text{F}(\phi^{n+1} \rightharpoondown (p \dagger1)), \ (\mathbbm{Z}^*_\Phi)^2 \to \mathbbm{Z}^+ \ | \ n \in \mathbbm{Z}^*, \ p \in \mathbbm{Z}^*_\Phi : \ 0 \leq p < \phi^n \right \}
\end{align*}
where $\dagger$ and $\rightharpoondown$ are the phinary addition and subtraction operators, respectively, as defined in Definition \ref{definition: phinary operators}.
\end{definition}

\begin{remark}
Compare to Definition \ref{definition: calkin-wilf} for the Calkin-Wilf tree.
\end{remark}

\begin{definition} \label{definition: esb}
The phinary even-recurrence tree (PERT) is a tree of rational numbers $\frac{a}{b}$ generated through a subset of the Fibonacci diatomic sequence, $\text{F}(\phi^2p)$, such that in the $n$th row, the $p$th numerator $a_{n,p}$ and denominator $b_{n,p}$ are given by
\begin{align*}
a_{n,p} &= \left \{\text{F} \left(\phi^2p \right), \ \mathbbm{Z}^+_\Phi \to \mathbbm{Z}^+ \ | \ n \in \mathbbm{Z}^+, \ p \in \mathbbm{Z}^+_\Phi : \ 0 \leq p < \phi^n, \ n \geq 0  \right \} \\
b_{n,p} &= \left \{\text{F}(\phi^{n+1} \rightharpoondown \phi^2(p \dagger 1)), \ (\mathbbm{Z}^+_\Phi)^2 \to \mathbbm{Z}^+ \ | \ n \in \mathbbm{Z}^+, \ p \in \mathbbm{Z}^+_\Phi : \ 0 \leq p < \phi^n, \ n \geq 0 \right \}
\end{align*}
where $\text{F}_0=0$.
\end{definition}

\begin{remark}
Compare to Definition \ref{definition: stern-brocot} for the Stern-Brocot tree. Recall that the original Stern-Brocot tree was generated from the even elements of the hyperbinary sequence, $\text{H}(2n)$. Here, the phinary even-recurrence tree is generated from $\text{F}(\phi^2p)$. Interestingly, the sequence that generates these numerators is equivalent to the number of representations of n as a sum of Fibonacci numbers where 1 is allowed twice as a part, instead of once, as is done for the Fibonacci diatomic sequence \cite[A000121]{sloane2003line}.
\end{remark}

\vspace{5mm}

Below every node in the PRT, we can observe a subtree whose elements are multiples of the parent tree, resulting in an infinite amount of multiple copies of itself.

Naturally, the PERT also contains a series of subtrees, but their structure differs from that of their PRT counterpart. The subtree generated by a node of the PERT tree contains infinitely many unreduced copies of the parent node's rational number.

\subsection{The Fibonacci Diatomic Recurrence Relation} \label{The Fibonacci diatomic Recurrence Relation}

With the tools we have defined, we can now write a recurrence relation for the Fibonacci diatomic sequence. This new relation shares a number of symmetries with the recurrence definition of the hyperbinary sequence---revealing a number of deep features of these numerical structures.

\begin{definition}
We define the following functions $f, e, o, d, c : \mathbbm{Z}^*_\Phi \to \mathbbm{Z}^*_\Phi$ such that
\begin{align*}
f(p) &:= \phi p \\
e(p) &:= \phi^2 p \\
o(p) &:= \phi^2 p + 1 \\
d(p) &:= \phi^2 (p \dagger 1) \\
c(p) &:= \phi^3 p + \phi
\end{align*}
for $p \in \mathbbm{Z}^*_\Phi$.
\end{definition}

\begin{remark}
We recall that these functions will indeed map the phinary numbers to themselves, as proven in \S \ref{subsection: functions}. One should be careful not to distribute the $\phi^2$ in $d(p)$ as ordinary multiplication is not distributive over phinary addition (Lemma \ref{lemma: nondistributivity}).
\end{remark}

\vspace{5mm}

\begin{lemma} \label{lemma: odd/curious}
The following relation between elements of the Fibonacci diatomic sequence holds.
\begin{align*}
\text{F}(f(p))=\text{F}(o(p))=\text{F}(c(p))
\end{align*}
\end{lemma}

\begin{proof}
Consider the hyperphinary representations of $f(p), o(p)$ and $c(p)$. The first fives sets of representations for $p \in \mathbbm{Z}^*_\Phi$ are as follows:
\begin{center}
\begin{tabular} { m{1cm} m{1cm} m{1cm} m{1cm} m{1cm}}
{\begin{align*}
f(0) &= 0_\Phi \\
o(0) &= 1_\Phi \\
c(0) &= 10_\Phi
\end{align*}}
&
{\begin{align*}
f(1) &= 10_\Phi \\
o(1) &= 101_\Phi \\
c(1) &= 1010_\Phi
\end{align*}}
&
{\begin{align*}
f(\phi) &= 100_\Phi, 101_\Phi \\
o(\phi) &= 1001_\Phi, 1011_\Phi  \\
c(\phi) &= 10010_\Phi, 10110_\Phi 
\end{align*}}
&
{\begin{align*}
f(\phi+1) &= 1000_\Phi, 110_\Phi  \\
o(\phi+1) &= 10001_\Phi, 1101_\Phi \\
c(\phi+1) &= 100010_\Phi, 11010_\Phi 
\end{align*}}
&
{\begin{align*}
f(\phi+2) &= 1010_\Phi \\
o(\phi+2) &= 10101_\Phi \\
c(\phi+2) &= 101010_\Phi
\end{align*}}
\end{tabular}
\end{center}
As we can see, the hyperphinary representations of $o(p)$ are the same as those of $f(p)$ but concatenated with a final 1---a fact that should have been clear, as
\begin{align*}
o(p) = \phi^2p + 1=\phi(\phi p) + 1=\phi f(p) + 1.
\end{align*}
In this way, the number of partitions of $f(p)$ into powers of phi will be the same as considered for $o(p)$. In other words,
\begin{align*}
\text{F}(o(p))=\text{F}(f(p)).
\end{align*}
Furthermore, the hyperphinary representations of $c(p)$ are the same as those of $o(p)$ but concatenated with a final 0---again, made clear by the fact that
\begin{align*}
c(p) = \phi(\phi^2 p + 1)=\phi o(p)
\end{align*}
Therefore, 
\begin{align*}
\text{F}(c(p))=\text{F}(o(p))
\end{align*}
\end{proof}

\begin{remark}
Notice that the trick in this proof does not work when going from hyperphinary representations of $f(p)$ to $d(p)$. Although $d(p)=\phi f(p)$ and therefore has the same representations as $f(p)$ concatenated with a final 0, we notice that the values associated with $f$ and $e$ are not mutually exclusive. Furthermore, the values of $d(p)$ differ from $f(p)$ in such a way that additional partitions become available, resulting in different values when evaluated by the Fibonacci diatomic sequence.
\end{remark}

\vspace{5mm}

\begin{lemma} \label{lemma: even}
The following relation holds.
\begin{align*}
\text{F}(\phi^2 (p \dagger 1) ) &= \text{F}(p) + \text{F}(p \dagger 1)
\end{align*}
for $p \in \mathbbm{Z}^*_\Phi$.
\end{lemma}

\begin{proof}
We define a new function $e(p)=\phi^2 p$ and notice that all of its hyperphinary representations end in either 00, 10, or 11. We find that all the representations ending in 00 are equal in quantity to the representations with the 00 removed, namely $\frac{e(p)}{\phi^2}=p$. Likewise, the representations ending in 10 are equal in quantity to the representations of $\frac{e(p)-\phi}{\phi^2}=p-\psi$. And finally, the representations ending in 11 are equal in quantity to the representations of $\frac{e(p)-\phi^2}{\phi^2}=p-1$. We notice, however, that $p-1$ and $p-\psi$ can never coexist; only one is possible for a given value of $p$. More to the point, we can simply speak of $p$'s predecessor $p \rightharpoondown 1$, which will take on the value of either $p-1$ or $p-\psi$. Therefore, the total number of partitions of $e(p)$ into powers of the golden ratio is equal to the sum of $\text{F}(p \rightharpoondown 1)$ and $\text{F}(p)$. In formulas,
\begin{align}
\text{F}(\phi^2 p) = \text{F}(p \rightharpoondown 1) + \text{F}(p) \label{eq: d(p) eq}
\end{align}
However, we are interested in $\text{F}(d(p))$. As $d(p)=\phi^2 (p \dagger 1) =e(p \dagger 1)$, we can simply $dagger$ a 1 to $p$ in Equation \ref{eq: d(p) eq} and evaluate:
\begin{align*}
\text{F}(\phi^2 (p \dagger 1) ) = \text{F}(p) + \text{F}(p \dagger 1)
\end{align*}
\end{proof}

\vspace{5mm}

We can now define a new recurrence relation for the Fibonacci diatomic sequence, analogous to that for the hyperbinary sequence. 

\begin{figure}[!p]
     \centering
    \begin{subfigure}[b]{\textwidth}
        \includegraphics[width=\textwidth]{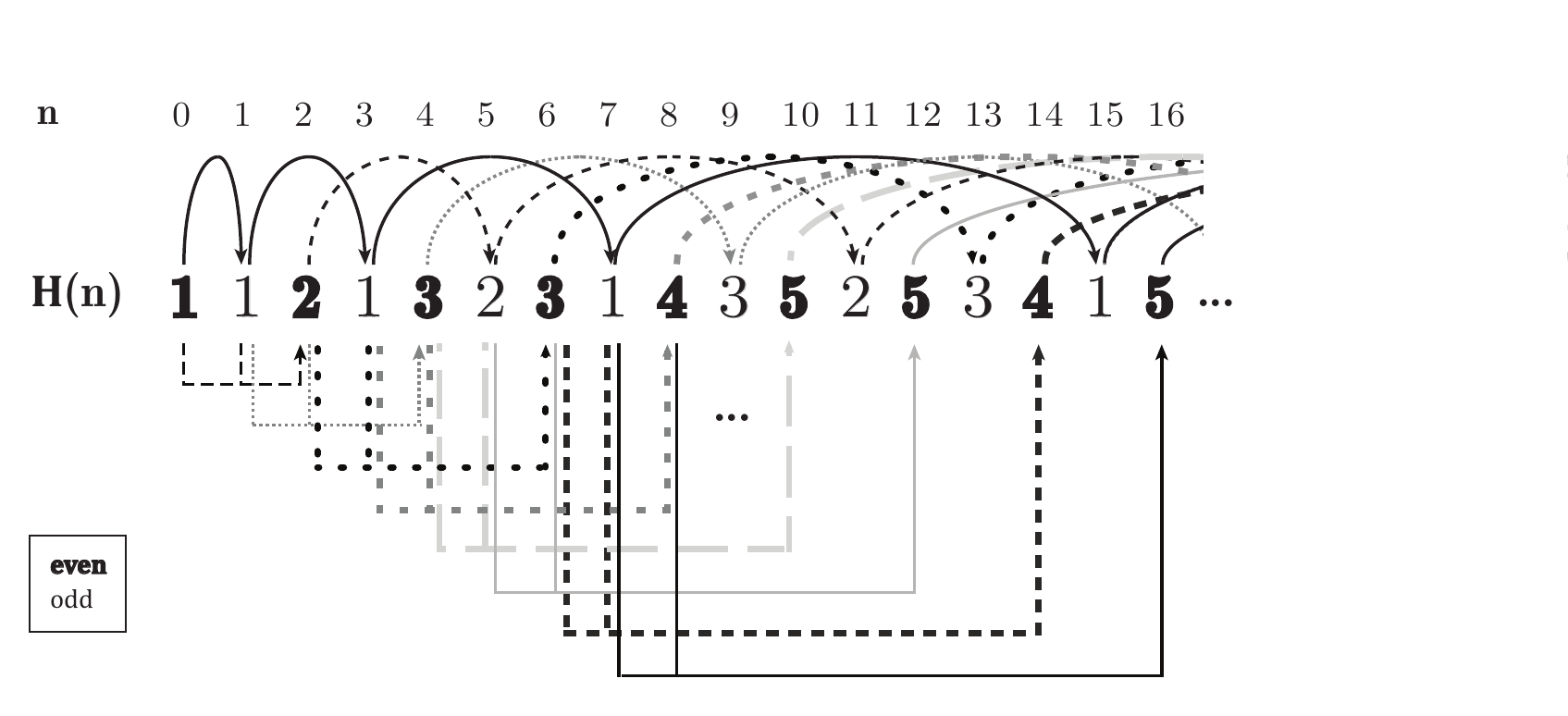}
        \caption{The hyperbinary sequence}
        \label{fig: hyperbinary sequence}
    \end{subfigure}
    
     %add desired spacing between images, e. g. ~, \quad, \qquad, \hfill etc. 
    %(or a blank line to force the subfigure onto a new line)
     \begin{subfigure}[b]{\textwidth}
        \includegraphics[width=\textwidth]{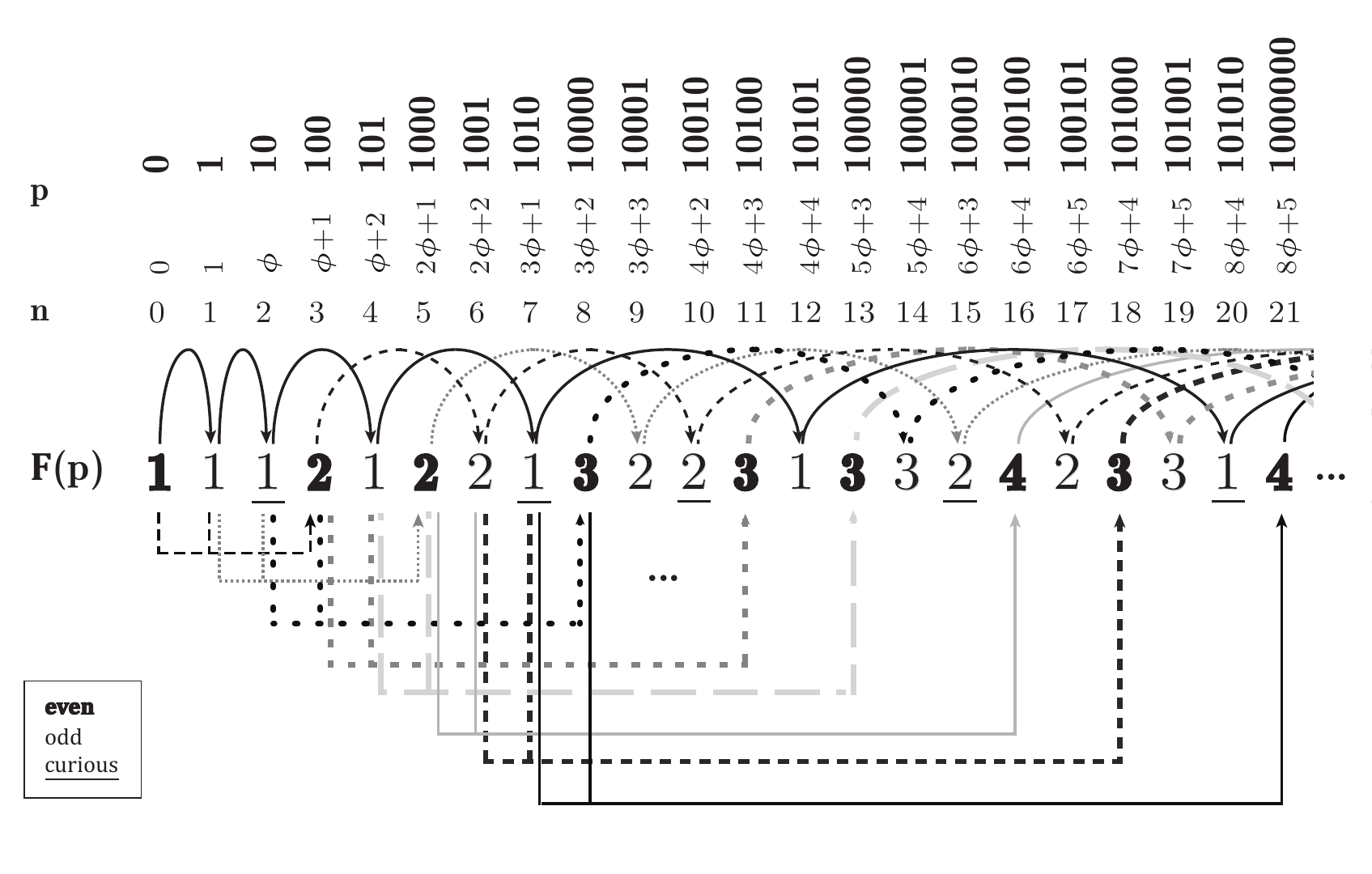}
        \caption{The Fibonacci diatomic sequence}
        \label{fig: Fibonacci diatomic sequence}
    \end{subfigure}
    ~%add desired spacing between images, e. g. ~, \quad, \qquad, \hfill etc. 
    %(or a blank line to force the subfigure onto a new line)
    \caption{The recurrence relations of the hyperbinary and Fibonacci diatomic sequence are illustrated with arrows. }
    \label{fig: hyperbinary and Fibonacci diatomic sequences}
\end{figure}

\vspace{5mm}

\begin{theorem} \label{theorem: recurrence hp}
The Fibonacci diatomic sequence is defined by the recurrence relation
\begin{align}
\text{F}(0) &= 1 \\
\text{F}(\phi^2 p + 1) &= \text{F}(\phi p) \label{eq: odd rec hp} \\
\text{F}(\phi^3 p + \phi) &= \text{F}(\phi^2p + 1) \label{eq: curious rec hp} \\
\text{F}(\phi^2 (p \dagger 1) ) &= \text{F}(p) + \text{F}(p \dagger 1) \label{eq: even rec hp}
\end{align}
for $p \in \mathbbm{Z}^*_\Phi$.
\end{theorem}

\begin{proof}
By Definition \ref{def: hp}, Lemma \ref{lemma: odd/curious}, and Lemma \ref{lemma: even}, we find that all values of $\text{F}(p)$ can be accounted for. The functions defined via phinary operators are permitted by Lemma \ref{lemma: p plus power} and Corollary \ref{corollary: addition iso}.
\end{proof}

\begin{remark}
Notice the similarity of the Fibonacci diatomic recurrence relation to the hyperbinary recurrence relation (Theorem \ref{theorem: recurrence hb}), which we show again here:
\begin{align}
\text{H}(0) &= 1 \nonumber \\
\text{H}(2n+1) &= \text{H}(n) \nonumber \\
\text{H}(2n+2) &= \text{H}(n)+\text{H}(n+1) \label{eq: even hb rec rel}
\end{align}
Particularly noteworthy is the form of Equation \ref{eq: even rec hp} in comparison to Equation \ref{eq: even hb rec rel}. The L.H.S. of Equation \ref{eq: even rec hp} cannot be written as $\text{F}(\phi^2 p \dagger \phi^2))$ or $\text{F}(\phi^2 p + \phi^2))$, as multiplication is not distributive over phinary addition (Lemma \ref{lemma: nondistributivity}). Naturally, however, the L.H.S. of Equation \ref{eq: even hb rec rel} can be written as $\text{H}(2n+2)$ or $\text{H}(2(n+1))$, as multiplication is distributive over addition. Nevertheless, Equation \ref{eq: even rec hp} seems to imply that $\text{H}(2(n+1))$ is perhaps a more natural or meaningful representation. It is this sort of subtlety revealed by phinary arithmetic that could potentially offer great insight into the underlying mechanics of the positive integers.
\end{remark}

As we will now see, the phinary numbers exhibit a sense of parity, which expresses itself in the recurrences of values within the Fibonacci diatomic sequence. As mentioned in the introduction, the multiplicity of the phinary parity is three, the details of which will be discussed in the following section. For now, let's assume the previously defined functions $e(p), o(p),$ and $c(p)$ respectively define what we will call the even, odd, and curious parity sets.

Unlike the hyperbinary sequence, which was equivalent to the sequence composed of its odd elements (Theorem \ref{theorem: recurrence hb}, eq. \ref{eq: odd hb}), the Fibonacci diatomic sequence does not simply equal the set composed of elements $\text{F}(\phi^2n+1)$, for example. However, we can make the following observations.

\vspace{5mm}

\begin{theorem} \label{theorem: phinary odd reoccurrence}
All odd-indexed values of the Fibonacci diatomic sequence are reoccurrences of preceding even-indexed or curious-indexed values, and all curious-indexed values are reoccurrences of preceding odd-indexed values, such that
\begin{align}
\text{F}(o(\phi q)) = \left \{ \text{F}(e(q)) \ | \ \forall p,q \in \mathbbm{Z}^*_\Phi : o(p)=\phi^2 p+1, \ e(p)=\phi^2 p \right \},
\end{align}
\begin{align}
\text{F}(o(o(q))) = \left \{ \text{F}(c(q)) \ | \ \forall p,q \in \mathbbm{Z}^*_\Phi : o(p)=\phi^2 p+1, \ c(p)=\phi o(p) \right \},
\end{align}
\begin{align}
\text{F}(c(q)) = \left \{ \text{F}(o(q)) \ | \ \forall p,q \in \mathbbm{Z}^*_\Phi : o(p)=\phi^2 p+1, \ c(p)=\phi o(p) \right \}.
\end{align}
\end{theorem}

%%%%%%%%%%%%%%%%%%%%%%%%%%%

 \begin{proof}
For phinary number q, we have $o(q)=\phi^2 q+1$ and $e(q)=\phi^2 q$. By the first recurrence formula, Equation \ref{eq: odd rec hp} in Theorem \ref{theorem: recurrence hp},
\begin{align*}
	\text{F}(\phi^2 p + 1) = \text{F}(\phi p).
\end{align*}
Letting $p=\phi q$, we have
\begin{align*}
	\text{F}(\phi^2 (\phi q) + 1) & = \text{F}(\phi (\phi q))\\
						& = \text{F}(\phi^2 q)).
\end{align*}
As $\phi^2 q=e(q)$ and $\phi^2 (\phi q) + 1=o(\phi q)$, we can write
\begin{align*}
	\text{F}(o(\phi q)) = \text{F}(e(q)),
\end{align*}
which proves the first part of the theorem. If now, instead, we let $p=o(q)$, the first recurrence formula yields
\begin{align*}
	\text{F}(\phi^2 (o(q)) + 1) & = \text{F}(\phi (o(q))).
\end{align*}
As $\phi (o(q))=c(q)$ and $\phi^2 (o(q)) + 1=o(o(q))$, we can write
\begin{align*}
	\text{F}(o(o(q))) = \text{F}(c(q)),
\end{align*}
which proves the second part of the theorem. Finally, by the second recurence formula, Equation \ref{eq: curious rec hp} in Theorem \ref{theorem: recurrence hp}, we know that
\begin{align*}
	\text{F}(\phi^3 p + \phi) &= \text{F}(\phi^2p + 1),
\end{align*}
which can be rewritten as
\begin{align*}
	\text{F}(c(p)) &= \text{F}(o(p)),
\end{align*}
as written in the third part of the theorem.
\end{proof}

\begin{remark}
The above theorem is directly analogous to the Theorem \ref{theorem: odd reoccurrence} for the hyperbinary sequence, which showed that the odd-indexed values of the hyperbinary sequence are reoccurrences of previous even-indexed values. In the same manner, odd-indexed values of the Fibonacci diatomic sequence are akin to copies of either even-indexed or odd-indexed values---to which the curious values will be further copies of these odd ones, resulting in an infinite cycle of copying between odd and curious terms. As with the hyperbinary sequence, the even-indexed values of the Fibonacci diatomic sequence serve as the ``primary occurrences" for any given value within the sequence.
\end{remark}

\begin{example}
The value of 2 occurs infinitely often in $\text{F}$. A prime occurrence is associated with the first even index $e(1)=\phi^2$, such that $\text{F}(e(1))=\text{F}(\phi^2)=2$. The first reoccurrence appears at the odd-indexed value $o(\phi)=\phi^3+1=2\phi+2$, whereby $\text{F}(o(\phi))=\text{F}(2\phi+2)=2$. The second reoccurrence comes at the curious-indexed value $c(o(\phi))=c(\phi^3+1)=\phi(\phi^3+1)+\phi=4\phi+2$, such that $\text{F}(c(o(\phi)))=\text{F}(4\phi+2)=2$. See Figure \ref{fig: Fibonacci diatomic sequence} for an illustration.
\end{example}

\begin{remark}
Note, just as for the hyperbinary sequence, not all primary occurrences are singular. In fact, the above example is not the sole primary occurrence of 2. As the number of primary occurrences for a given value in the hyperbinary sequence is related to Euler's totient function, the number of primary occurrences of $p$ in the Fibonacci diatomic sequence will very likely be associated with a phinary analogue to Euler's totient function, a topic worthy of future study, as it could reveal deep symmetries with regard to the structure of primes in an ordinal set.
\end{remark}

\subsection{Phinary Parity} \label{section: phinary parity}

\begin{theorem}
The phinary system has a parity triplet composed of the even, odd, and curious subsets of $\mathbbm{Z}^*_{\Phi}$.
\begin{align*}
\mathbbm{Z}^*_{\Phi even}&=\bigcup_p \phi^{2}p \\
\mathbbm{Z}^*_{\Phi odd}&=\bigcup_p\phi^2p+1 \\
\mathbbm{Z}^*_{\Phi curious}&=\bigcup_p\phi^3p+\phi
\end{align*}
\end{theorem}

\begin{proof}
It must be noted that the definition of parity used here is made with some sense of choice. For example, one could define the even and odd phinary parities as $\phi \star p$ and $\phi \star p \dagger 1$, respectively. This would reproduce the traditional every-other-number notion of parity assumed by the natural numbers. However, doing so would shoehorn a pattern that does not particularly fit the phinary numbers. For example, $\phi \star p$ does not contain the powers of the number base, $\phi$, as we would expect of its even elements. On the other hand, we could claim that $\phi \star p$ contains perhaps some other notion of number-base powers---namely, $\phi$, $\phi^{\underline{\phi}}$, $\phi^{\underline{\phi + 1}}$, $\phi^{\underline{\phi + 2}}$, and so on (see \S \ref{subsection: phinary operators} on phinary operations)---but this would simply recreate an exact isomorphism to the positive integers. In other words, we would have only changed the labels of the positive integers and revealed nothing new about this system. Furthermore, the hyperbinary and Fibonacci diatomic sequences highlight important properties of their respective number systems when adopting this paper's notion of parity: the even elements are endowed with the largest values by virtue of having the greatest number of representations in their respective number base; the number of representations of non-even elements are defined via their even counterparts (Theorems \ref{theorem: odd reoccurrence} and \ref{theorem: phinary odd reoccurrence}); and even numbers contain all the powers of their number base, with the exception of 1 and, in phinary, $\phi$ (Theorem \ref{Theorem: phinary even powers}). Moreover, the even-recurrence trees are defined by our choice of even parity, as the name suggests.

As for the label of ``odd" parity, the name is applied to phinary numbers that are ``one greater" than even values, as it is done for the natural numbers. The remaining phinary numbers, which are neither even or odd, are dubbed curious. As shown in Theorem \ref{theorem: phinary odd reoccurrence}, odd and curious values share certain properties in relation to the Fibonacci diatomic sequence but are nevertheless unique; for example, the distribution of odd and curious values through the phinary numbers, defined by the Fibonacci word pattern, varies respectively in terms of differences between consecutive values.

For all of these reasons, the parity label is used in the manner presented above.
\end{proof}

\vspace{5mm}

\begin{theorem} \label{Theorem: phinary even powers}
In phinary, powers of phi---or equivalently the phinary Fibonacci numbers---are always even in parity, with the exception of $1$ and $\phi$.
\end{theorem}

\begin{proof}
As any even phinary number takes the form $\phi^{2}p$ for $p \in \mathbbm{Z}^*_\Phi$, it is easy to see that all powers of phi, $\phi^n, n>1$ are even in parity.
\end{proof}

\vspace{5mm}

\begin{corollary}
The phinary count appears in the sequence of one's place values of the base-phi representation of ordinals.
\end{corollary}

\begin{proof}
This is, of course, a result of the odd phinary numbers, the only ones for which the standard form ends in a value of 1. If we enumerate the phinary numbers in base-phi and create a sequence $S$ from successive one's place values, we find

\begin{center}
\centering
\begin{tabular}{r r}
$\textbf{0}$					& $\textcolor{gray}{1010}\textbf{0}$ \\
$\textbf{1}$					& $\textcolor{gray}{1010}\textbf{1}$ \\
$\textcolor{gray}{1}\textbf{0}$		& $\textcolor{gray}{10000}\textbf{0}$ \\
$\textcolor{gray}{10}\textbf{0}$		& $\textcolor{gray}{10000}\textbf{1}$ \\
$\textcolor{gray}{10}\textbf{1}$		& $\textcolor{gray}{10001}\textbf{0}$ \\
$\textcolor{gray}{100}\textbf{0}$	& $\textcolor{gray}{10010}\textbf{0}$ \\
$\textcolor{gray}{100}\textbf{1}$	& $\textcolor{gray}{10010}\textbf{1}$ \\
$\textcolor{gray}{101}\textbf{0}$	& $\textcolor{gray}{10100}\textbf{0}$ \\
$\textcolor{gray}{1000}\textbf{0}$	& $\textcolor{gray}{10100}\textbf{1}$ \\
$\textcolor{gray}{1000}\textbf{1}$	& $\textcolor{gray}{10101}\textbf{0}$\\
$\textcolor{gray}{1001}\textbf{0}$	& $\textcolor{gray}{100000}\textbf{0}$
\end{tabular}
\end{center}

\begin{align*}
S=\{0,1,0,0,1,0,1,0,0,1,0,0,1,0,1,0,0,1,0,1,0,0,...\}
\end{align*}

The pattern continues by induction as the infinite Fibonacci word.
\end{proof}

\section{Cardinal Multiplicity} \label{section: cardinal multiplicity}

Using the natural and golden diamond geometries to define bijections between the set of natural and phinary numbers, respectively, we can make some surprising claims about transfinite cardinality.

\subsection{Transfinite Cardinalities}

In set theory, a set of ordinals can be extended into transfinite values through the process of transfinite induction. For example, after \textit{all} the natural numbers, $1, 2, 3, 4, 5, \dots$, comes the first infinite ordinal, denoted $\omega_0$ or simply $\omega$ \cite[p. 44]{ciesielski1997set}. By definition, an ordinal is a set, formalized by the Von Neumann construction \cite[p. 37]{ciesielski1997set}, that contains all the ordinals less than it. In this way, the cardinality of an ordinal is the cardinality of the set it defines. The cardinality of $\omega$ is therefore equivalent to the cardinality of the set of natural numbers and is denoted $\aleph_0$---a value that has not been well-defined in the literature. This is the basis for the axiom of infinity, which posits the existence of at least one infinite set---namely, the set of natural numbers \cite{zermelo1908untersuchungen}.

The ordinal $\omega$, often called the first countably infinite ordinal, is not a number in the traditional sense, as its cardinality is unaffected by any finite arithmetical operation\footnote{Arithmetic on transfinite ordinals is a subtle topic, which will not be explored in detail here. For example, addition is not commutative, such that $1+\omega$ is not defined in the same way as $\omega+1$ \cite[p. 23]{jech2013set}---the former equalling $\omega$ and the latter serving as $\omega$'s successor. Nevertheless, both values have the same cardinality, as they are countably infinite.} applied to it. For example, the cardinality of $\omega+1$ is still $\aleph_0$, as is the cardinality of $\omega+\omega$, and so on \cite[pp. 29-30]{jech2013set}. Despite this, it is common to define values like $\omega+1$ as ordinals---this example being the successor of $\omega$  \cite[p. 46]{ciesielski1997set}. However, as we will see, there are potential issues with this conceptualization. Nevertheless, one can think abstractly about such a notion as $\omega+1$ and its later successors---values such as $\omega+\omega$ or even $\omega^{\omega}$. Again, these values still have the same cardinality, $\aleph_0$, until one continues to an ``indefinitely" greater value, labelled $\omega_1$, which is often described as the first uncountably infinite ordinal \cite[p. 30]{jech2013set}\cite[p. 66]{ciesielski1997set}. The cardinality of $\omega_1$ is denoted $\aleph_1$ and describes the size of the set containing all values less than $\omega_1$. Georg Cantor, who first described these notions, showed, by several means, that $\aleph_1=2^{\aleph_0}$ and that it defines the cardinality of the set of real numbers $\mathbbm{R}$---an argument known as the Continuum Hypothesis \cite[p. 37]{jech2013set}. We will now review the proof of this notion, known as Cantor's ``diagonal argument."

\vspace{5mm}

\begin{theorem} [Cantor's Diagonal Argument] \label{theorem: Cantor's Diagonal Argument}
The cardinality of $\mathbbm{R}$ is
\begin{align*}
\aleph_1=2^{\aleph_0}.
\end{align*}
\end{theorem}

\begin{proof}
Let $T$ be a set of all infinite sequences of binary digits, i.e. strings of zeros and ones. If we attempt to list  all the elements of T, i.e. the sequences $s_1, s_2, s_3, \dots, s_n, \dots$, there will always be a sequence $s$, which is not included in the list. This can be easily observed by generating a new sequence from all the sequences already in T. For example, consider the sequences
\begin{align*}
s_1 &= (0,0,0,0,0,0,0,0,\dots)\\
s_2 &= (1,0,1,0,1,0,1,0,\dots)\\
s_3 &= (1,1,0,0,1,1,0,0,\dots)\\
s_4 &= (0,0,1,1,0,0,1,1,\dots)\\
s_5 &= (1,0,1,1,0,1,0,1,\dots)\\
s_6 &= (0,1,0,0,0,0,1,0,\dots)\\
s_7 &= (1,1,0,0,1,0,1,0,\dots)\\
s_6 &= (0,1,1,1,0,1,1,1,\dots)\\
& \ \ \ \ \ \ \ \ \ \ \ \  \ \ \ \ \ \vdots
\end{align*}
One can generate a new sequence $s$ by listing the inverted $n$th digit of each sequence $s_n$. Let $s_1(1)$ denote the first digit of sequence $s_1$, which equals $0$, and let its inversion be $\overline{s_1}(1)=1$. In this way, our new sequence $s$ is defined as
\begin{align*}
s&=(\overline{s_1}(1),\overline{s_2}(2),\overline{s_3}(3),\overline{s_4}(4),\overline{s_5}(5),\overline{s_6}(6),\overline{s_7}(7),\overline{s_8}(8),\dots )\\
&=(1,1,1,0,1,1,0,0,\dots ).
\end{align*}
Consequently, $s$ was not previously a member of $T$, as it will always differ by at least one digit with any sequence we had previously defined. In this sense, the cardinality of $T$ is deemed uncountable. Furthermore, as the cardinality of each sequence in $T$ is equal to $\aleph_0$, and each digit assumes only one of two possible values, i.e. zero or one, the resulting cardinality of $T$ is $2^{\aleph_0}$.

As any real number can be written as an infinite string of binary digits, it follows that the set of real numbers is also uncountable, such that $|\mathbbm{R}|=2^{\aleph_0}$. 
\end{proof}

\begin{figure}[!p]
    \centering
    \includegraphics[width=\textwidth]{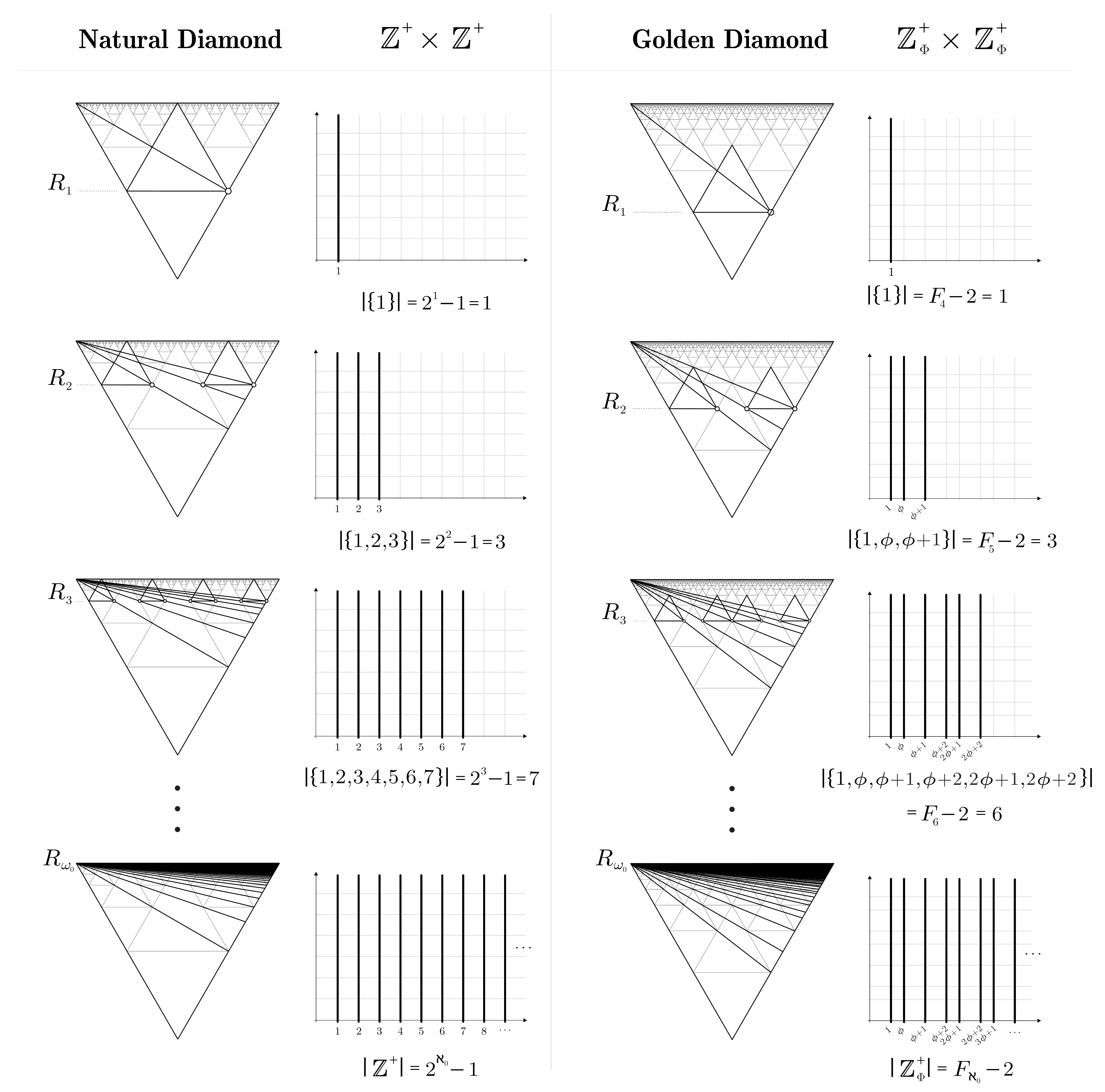}
    \caption{As the natural and golden diamonds are perspective projections of their respective domains, there exists a bijection that maps projected interval markings or grid lines of each image to their ordinal values. This figure iteratively demonstrates the projection of interval markings for each of the diamonds, thereby quantifying the cardinality of both the infinite sets of natural and phinary numbers.}
    \label{fig: projected interval markings}
\end{figure}

\subsection{Cardinality of the Set of Natural Numbers}

The natural diamond has features that are isomorphic to the set of natural numbers, providing an opportunity to explore the transfinite via a finite geometry. To begin with, we detail some aspects of cardinality relating to the above mentioned ND features.

\vspace{5mm}

\begin{lemma} \label{lemma: recurrence cardinality}
The number of nodes in the $n$th row of the natural recurrence tree, e.g. Calkin-Wilf/Stern-Brocot tree, is given by the relation $2^N$, where $N=|n|$.
\end{lemma}

\begin{proof}
The importance of this lemma is for its subtle emphasis on cardinality, which will be paramount in following transfinite proofs. Cardinal exponentiation between two cardinals $A=|a|$ and $B=|b|$ of the form $A^B$ evaluates the number of functions which map any set of $A$ elements to any set of $B$ elements \cite[pp. 51-52]{jech2013set}. We will define $a$ and $b$ such that $a=\{\text{L},\text{R}\}$ and $b=\{1,2,3,\dots,n\}$ where L and R denote \textit{left} and \textit{right}. The cardinality of $A^B$ can be interpreted as the number of possible paths through the bifurcating tree, after $n$ bifurcations. For example, if $n=3$, the possible paths from the root of the tree are LLL, RRR, LLR, LRL, RLL, RRL, RLR, LRR, of which there are $2^3=8$, i.e. the number of nodes in row 3 of the tree. 
\end{proof}

\vspace{5mm}

\begin{definition}
Let the $n$th set of \textit{cutting points} of the natural diamond be a subset of the vertices within the ND's $n$th row, i.e. the nodes of the $n$th row in the underlying recurrence tree, such that the first, i.e. left-most, node is omitted. The number of cutting points in the $n$th row of the ND is therefore $2^N-1$, where $N=|n|$.
\end{definition}

\begin{remark}
See Figure \ref{fig: projected interval markings} for an illustration of cutting points, indicated by white circles within a row of the diamond.
\end{remark}

\vspace{5mm}

\begin{proposition} \label{proposition: infinite positive integer parity}
	The cardinality $\aleph_0$ of the set of all natural numbers $\mathbbm{Z}^+$ is 
	\begin{align*}
		\aleph_0= 2^{\aleph_0}-1.
	\end{align*}
\end{proposition}

\begin{proof}
	This odd-looking theorem requires a careful examination. The rays extending from a vanishing point of the ND, i.e. from one of the two top ND vertices, through cutting points  in the limiting row of the ND, project interval markings along the right edge of the ND, the collection of which is isomorphic to the set of positive integers (see Figure \ref{fig: projected interval markings}). As the $n$th row of the ND contains $2^N-1$ cutting points, by transfinite induction, we can use this feature to determine the cardinality of the natural numbers. We will demonstrate this formally by considering the following true statements:
\begin{center}
\begin{tabular}{ r l}
	I.) 	&There is a bijection from the cutting points in a row of the ND to ordinals in an initial subset\\
		&of the natural numbers $\mathbbm{Z}^+$.\\
	II.) 	&The number of cutting points in the $n$th row is $2^N-1$, where $N=|n|$.\\
	III.)	&The number of rows in the ND is transfinite, i.e. infinite.
\end{tabular}
\end{center}
Additionally, we state the following definitions for $\omega$ and the countably transfinite:
\begin{center}
\begin{tabular}{ r l}
	A.) 	&The value $\omega$ is an ordinal.\\
	B.) 	&The cardinality of $\omega$ is ``countably transfinite" and therefore $\omega$ is a countably transfinite ordinal.\\
	C.)	&All countably transfinite ordinals are equal in cardinality to the set of natural numbers.\\
	D.)	&A countably transfinite cardinality is not finite and is less than all other transfinite cardinalities.\\
\end{tabular}
\end{center}

Statement I is proven by the perspective projection between the natural number domain and the ND (Theorem \ref{theorem: ND perspective projection plane}). Statements II and III are true by inspection of the ND geometry. However, we must justify the subtle point in Statement II, showing that this calculation reflects cardinal arithmetic---most importantly, that $2^N$ is cardinal exponentiation. This point has been proven by Lemma \ref{lemma: recurrence cardinality}, which showed that the cardinality of nodes in a row of the natural recurrence tree is defined by possible paths through the tree. In general, if we remove the path that contains only L's, the remaining total $2^N-1$ is the number of cutting points in the $n$th row\footnote{Note that the set defined here by $R$ contains only cutting points and is different from the definition used previously in this paper, in which $R$ contained pairs of facets from the GD.} of the ND and, as we have shown, represents a cardinal arithmetical expression.

By Statement I, a row is isomorphic to an initial subset of $\mathbbm{Z}^+$. Therefore, we will refer to each row $R$ as the actual subset, $R \subseteq \mathbbm{Z}^+$. By Statement III, there are infinite rows in the ND, and therefore, there is some countably transfinite ordinal $\alpha$, such that $R_\alpha$ is a row in the ND, by Statement D. By Statement II, the cardinality of $R_\alpha$, for $A=|\alpha|$, equals $2^A-1$, which is also transfinite by the rules of transfinite cardinal arithmetic \cite[pp. 51-55]{jech2013set}. This implies that the first transfinite ordinal $\omega$ must be a member of $R_\alpha$, that is, $\omega \in R_\alpha$. Our approach will be to identify the minimum value of $\alpha$ such that $\omega \in R_\alpha$ and subsequently pin-down the ``location" of $\omega$ within $R_\alpha$ and thereby quantify its cardinality.

There are three values of $\alpha$ to consider: $\alpha<\omega$, $\alpha=\omega$, and $\alpha>\omega$. Let's consider the first option and assume that $\omega \in R_\alpha$ for $\alpha<\omega$. We know that 
\begin{align*}
|R_\alpha|=2^A-1.
\end{align*}
However, as $\omega$ is the first transfinite ordinal and $\alpha<\omega$, $A$ must be finite and therefore $|R_\alpha|$ is also finite. This is a contradiction by Statement D, as $\omega$ cannot be a member of a finite initial subset of $\mathbbm{Z}^+$. We now consider the second option, $\alpha=\omega$. The cardinality of row $R_{\alpha=\omega}$ is
\begin{align*}
|R_{\omega}|	&=2^{\aleph_0}-1
\end{align*}
which is also countably transfinite by Statement I and C. Further proof that $|R_{\omega}|$ is countably transfinite comes from the Stern-Brocot/Calkin-Wilf trees, which show that the rational numbers can be matched one-for-one with the nodes in the underlying recurrence tree of the ND; as the set of rationals is countably transfinite, the cutting points in the limiting row of the ND can be no greater than this cardinality. Therefore, $R_{\omega}$ must contain $\omega$. Moreover, as there is no ``smaller" row for which $\omega$ is a member, we have proven that $R_{\omega}$ is the first row that contains $\omega$.

Since $|\omega|$ and $|R_{\omega}|$ are both countably transfinite, they are both equal to $\aleph_0$ (Statement C). Thus,
\begin{align}
\aleph_0=|\mathbbm{Z}^+|=|\omega|=|R_{\omega}|=2^{\aleph_0}-1.
\end{align}
\end{proof}

\begin{remark}
It should now be clear why Lemma \ref{lemma: recurrence cardinality} was worth proving: Had we defined the number of cutting points in the $n$th row to be $2^n-1$, we would have found that $\aleph_0=2^{\omega}-1=\omega-1$, which is undefined.

Although the conclusions that follow are unusual and perhaps controversial, it must be strongly noted here that the above result came about by simply applying the logic inherent in $\omega$'s definition.
\end{remark}

\vspace{5mm}

If the above result is to be accepted, we are forced into the conclusion of the following theorem:

\vspace{5mm}

\begin{theorem} \label{theorem: aleph1}
The first uncountably transfinite cardinality $\aleph_1$ is one greater than the cardinality of the countably transfinite $\aleph_0$.
\begin{align*}
\aleph_1 = \aleph_0+1
\end{align*}
\end{theorem}

\begin{proof}
The cardinality $\aleph_1$ of the first uncountably transfinite set was shown to be equal to $2^{\aleph_0}$ (Theorem \ref{theorem: Cantor's Diagonal Argument}). From the result $\aleph_0=2^{\aleph_0}-1$ of Proposition \ref{proposition: infinite positive integer parity}, we clearly see that $\aleph_1=\aleph_0+1$.
\end{proof}

\begin{remark}
This result implies that the first uncountably transfinite ordinal $w_1$, whose cardinality is equal to $\aleph_1$, is the set of all natural numbers union with some predecessor or successor ordinal of the natural numbers. Two options present themselves: Neither 0 or $w_0$ were included in the set of natural numbers, and both ``sandwich" the members of $\mathbbm{Z}^+$. Therefore, we can write
\begin{equation*}
\aleph_1=|0 \cup \mathbbm{Z}^+|=|\mathbbm{Z}^+ \cup w_0|.
\end{equation*}
\end{remark}

\vspace{5mm}

If we are to continue to accept the result of Proposition \ref{proposition: infinite positive integer parity}, then we are again forced into some unexpected conclusions. Intuitively, one may perceive the cardinality of the set of natural numbers as some inconceivably large value. However, the equation we have found, $\aleph_0=2^{\aleph_0}-1$, prohibits such a notion. Nevertheless, it does admit two unexpected solutions, as described below.

\vspace{5mm}

\begin{theorem} \label{theorem: natural cardinality 0,1}
The cardinality of the countably transfinite is defined by two values: zero and one.
\begin{align*}
\aleph_0=\{0,1\}
\end{align*}
\end{theorem}

\begin{proof}
This comes directly from evaluating the expression $\aleph_0=2^{\aleph_0}-1$ from Proposition \ref{proposition: infinite positive integer parity}, whereby
\begin{align*}
0=2^{\aleph_0}-\aleph_0-1,
\end{align*}
which has two solutions: $\{0,1\}$.
\end{proof}

\begin{remark}
This result implies that the ``size" of the countably transfinite is both nonexistent and unitary---a surprising mathematical result. Note, however, had we required in Proposition \ref{proposition: infinite positive integer parity} that $\aleph_0>n, \forall n\in\mathbbm{Z}$, the answer would be undefined. If accepted, this invalidates the concept of ``infinity" as some number greater than all the natural numbers and points instead to something more subtle.\footnote{Noteworthy is the similarity between the equation $0=2^x-x-1$, which defines the cardinality of $\omega$, $\{0,1\}$, and the equation $0=x^2-x-1$, which defines the golden ratio and its conjugate $\{-\frac{1}{\phi},\phi\}$.}
\end{remark}

\vspace{5mm}

Continuing this logic, one is led to the following theorem.

\vspace{5mm}

\begin{theorem} 
The cardinality of the continuum is defined by two values: one and two.
\begin{align*}
\aleph_1=\{1,2\}
\end{align*}
\end{theorem}

\begin{proof}
By Theorem \ref{theorem: aleph1} and \ref{theorem: natural cardinality 0,1},
\begin{align*}
\aleph_1 = \aleph_0+1,
\end{align*}
of which there are two solutions: $\{1,2\}$.
\end{proof}

\vspace{5mm}

The results presented above are unexpected for a number of reasons. For one, the cardinality of the natural numbers is expressed as either zero, one, or both---not some extremely large, divergent value. Furthermore, the cardinality of $\aleph_1$ is shown to be one greater than $\aleph_0$, despite an infinite set of ordinals that supposedly exist between $\omega_0$ and $\omega_1$. However, as the countably transfinite ordinals just mentioned are all of the same cardinality, i.e. $\aleph_0$, there does seem to be some sense implied in this result.

In the following section, we will apply the same logic to the golden diamond and phinary numbers. Despite a very different geometry and resulting equation for the cardinality of the countably transfinite, as we will see, the conclusion is nevertheless the same as above.

\subsection{Cardinality of the Set of Phinary Numbers} 

Similar to the previous section, here we will use the golden diamond to consider the cardinality of the set of phinary numbers.

\vspace{5mm}

\begin{lemma} \label{lemma: GD cutting points}
The number of cutting points in the $n$th row of the GD is 
\begin{align*}
F_{N+3} - 2
\end{align*}
in terms of natural numbers or
\begin{align*}
\phi^{N+1} \rightharpoondown \phi
\end{align*}
in terms of phinary numbers, where $F_{N+3}$ is the $N+3$rd Fibonacci number and $N=|n|$.
\end{lemma}

\begin{proof}
In the same way Lemma \ref{lemma: recurrence cardinality} showed that the number of nodes in the $n$th row of the ND was defined in terms of the cardinality $N=|n|$, this lemma intends to show that the number of cutting points in the $n$th row of the GD is also defined in terms of $N$.

Theorem \ref{theorem: palindromic Fibonacci subwords GD} demonstrated that the union of any two consecutive rows in the GD is isomorphic to an initial palindromic Fibonacci subword (See Figure \ref{fig: Fibonacci word GD}). Furthermore, Theorem \ref{theorem: initial palindromic Fibonacci subwords in GD} stated that the truncating of a proper Fibonacci word $w_n$ for $n>2$ by its last two letters produces an initial palindromic Fibonacci subword. These proofs came from defining the golden diamond in terms of the golden trapezoid, such that the $n$th row of the GD is a subset of the $n+2$nd row in the GT. As the union of consecutive rows, $R_n$ and $R_{n+1}$, of the GT is isomorphic to the proper Fibonacci word $w_n$ (Theorem \ref{theorem: Fibonacci words in golden trapezoid}, Figure \ref{fig: Fibonacci Trapezoid}), which contains $F_{n+1}$ letters, it follows that the union of two consecutive rows, $R_n$ and $R_{n+1}$, of the GD results in an initial palindromic Fibonacci subword of length $F_{n+3}-2$.

Moreover, we can show that the value $F_{n+3}-2$ is more accurately defined as $F_{N+3}-2$, for $N=|n|$. This comes from the fact that each row of the GT, and therefore also of the GD, is defined by the Fibonacci word substitution rule (Theorem \ref{theorem: fib word substitution rule}), a morphism that increases the number of facets in each row as a function of the succeeding row's cardinality. That is, once the cardinality of the first row is defined, the subsequent row cardinalities are, as well.

Since the number of cutting points in the $n$th row of the GD is equivalent to the number of letters in the associated initial Fibonacci subword, the number of cutting points is therefore $F_{N+3}-2$.

We can now consider the cardinality in terms of phinary numbers. That is, we will still use a natural number index for the row of the GD, but will evaluate its cardinality in terms of $\mathbbm{Z}^+_\Phi$. The lemma states that the number of cutting points associated with the $n$th row of the GD is $\phi^{N+1} \rightharpoondown \phi$. This expression is easily verified by reviewing the table in \S \ref{subsection: phinary numbers}, which showed that $F_{n+2} \in \mathbbm{Z}^+$  is associated with $\phi^n \in \mathbbm{Z}^+_\Phi$. Furthermore, the ``minus 2" of the natural number expression is more precisely defined as the predecessor of the predecessor, which in phinary is denoted by ``hook $\phi$". Consequently, the natural cardinality $F_{N+3}-2$ becomes the phinary cardinality $\phi^{N+1} \rightharpoondown \phi$.
\end{proof}

\vspace{5mm}

\begin{proposition} \label{proposition: phinary cardinality}
	The cardinality of the set of phinary numbers is $\aleph_0$, defined by
	\begin{align*}
		\aleph_0=F_{\aleph_0+3} - 2
	\end{align*}
	in terms of natural numbers and
	\begin{align*}
		\aleph_0=\phi^{\aleph_0+1} \rightharpoondown \phi
	\end{align*}
	in terms of phinary numbers. 
\end{proposition}

\begin{proof}
Analogous to Proposition \ref{proposition: infinite positive integer parity}, we consider the following true statements:
\begin{center}
\begin{tabular}{ r l}
	I.) 	&There is a bijection from the cutting points in a row of the GD to phinary ordinals in an initial subset\\
		&of the phinary numbers $\mathbbm{Z}^+_\Phi$.\\
	II.) 	&The number of cutting points in the $n$th row of the GD is $F_{N+3} - 2$, in terms of natural numbers,\\
		&or $\phi^{N+1} \rightharpoondown \phi$, in terms of phinary numbers, where $N=|n|$.\\
	III.)	&The number of rows in the GD is transfinite, i.e. infinite.
\end{tabular}
\end{center}

In this proof, we will consider the first countably transfinite phinary number $\upsilon$. We state the following definitions for $\upsilon$ and the countably transfinite:
\begin{center}
\begin{tabular}{ r l}
	A.) 	&The value $\upsilon$ is an ordinal.\\
	B.) 	&The cardinality of $\upsilon$ is ``countably transfinite" and therefore $\upsilon$ is a countably transfinite ordinal.\\
	C.)	&All countably transfinite ordinals are equal in cardinality to the set of phinary numbers.\\
	D.)	&A countably transfinite cardinality is not finite and is less than all other transfinite cardinalities.\\
\end{tabular}
\end{center}

By Statement I, a row is isomorphic to an initial subset of $\mathbbm{Z}^+_\Phi$. Therefore, we will refer to each row $R$ as the actual subset, $R \subseteq \mathbbm{Z}^+_\Phi$. Statement II, comes from Lemma \ref{lemma: GD cutting points}.  By Statement III, there are infinite rows in the GD, and therefore, there is some transfinite ordinal $\alpha$, such that $R_\alpha$ is a row in the GD, by Statement D. Note, however, that we must be clear about which ordinal set we are using to index each row. We will choose to define the row index such that $\alpha$ is a natural number. By Statement II, the cardinality of row\footnote{As in the last section, the cardinality of a row $R$ is defined here to mean the number of cutting points in that row---not to be confused with the $R$ used to contain facets in earlier theorems.} $R_\alpha$ equals $F_{A+3}-2$, for $A=|\alpha|$, which is also transfinite by the rules of transfinite cardinal arithmetic \cite[pp. 51-55]{jech2013set}. This implies that the first transfinite phinary ordinal, which we define as $\upsilon$ and whose cardinality is equivalent to $\mathbbm{Z}^+_\Phi$ by Statement C, must be a member of $R_\alpha$, that is, $\upsilon \in R_\alpha$. It should be noted that here we are making the reasonable assumption that phinary numbers behave in the same manner as the natural numbers, in that the first transfinite ordinal has a cardinality equal to the associated set of phinary ordinals.\footnote{A Von Neumann-like approach to defining the phinary ordinals would be desirable to make this rigorous.} Our approach will be to identify the minimum value of $\alpha$ such that $\upsilon \in R_\alpha$ and subsequently pin-down the ``location" of $\upsilon$ within $R_\alpha$ and thereby quantify its cardinality.

There are three values of $\alpha$ to consider: $\alpha<\upsilon$, $\alpha=\upsilon$, and $\alpha>\upsilon$. Let's consider the first option and assume that $\upsilon \in R_\alpha$ for $\alpha<\upsilon$. We know that 
\begin{align*}
|R_\alpha|=F_{A+3}-2.
\end{align*}
However, as $\upsilon$ is the first transfinite phinary ordinal, by definition, $\alpha$ must be finite and therefore $|R_\alpha|$ is also finite. This is a contradiction by Statement D, as $\upsilon$ cannot be a member of a finite initial subset of $\mathbbm{Z}^+_\Phi$. We now consider the second option, that $\alpha=\upsilon$. In finite circumstances, this equality would not be permitted, as we have defined $\alpha$ as a natural number and $\upsilon$ as a phinary number, which would only be true for $\alpha=\upsilon=1$. However, we can substitute the first natural transfinite ordinal $\omega$ for $\upsilon$, because both are countably infinite and therefore isomorphic (Statement C)---resulting in $\alpha=\omega$. Furthermore, as the cardinality of $\omega$ is $\aleph_0$, the cardinality of the row $R_{\alpha=\omega}$ is therefore
\begin{align*}
|R_{\omega}|	&=F_{\aleph_0+3}-2
\end{align*}
which is also countably transfinite by Statement I and C, and therefore must contain $\upsilon$. Moreover, as there is no ``smaller" row for which $\upsilon$ is a member, we have proven that $R_{\omega}$ is the first row that contains $\upsilon$.

Since $|\omega|$ and $|R_{\omega}|$ are both countably transfinite, they are equal to $\aleph_0$ (Statement C). Thus,
\begin{align}
\aleph_0=|\omega|=|\upsilon|=|\mathbbm{Z}^+_\Phi|=|R_{\omega}|=F_{\aleph_0+3}-2.
\end{align}
Finally, we consider the cardinality in terms of phinary numbers. That is, we will still use a natural number index for the row of the GD, but will evaluate its cardinality in terms of $\mathbbm{Z}^+_\Phi$. As expressed in Statement II, the number of phinary ordinals associated with the $n$th row of the GD is $\phi^{N+1} \rightharpoondown \phi$. Subsequently, we have
\begin{align*}
	\aleph_0	&= \phi^{\aleph_0+1} \rightharpoondown \phi.
\end{align*}
\end{proof}

\vspace{5mm}

\begin{theorem}
The cardinality of the phinary numbers is equal to the cardinality of the natural numbers, defined by the dual value
\begin{align*}
\aleph_0=\{0,1\}.
\end{align*}
\end{theorem}

\begin{proof}
By evaluating the expressions of Proposition \ref{proposition: phinary cardinality}, we find that in terms of natural numbers, the solution to $x=F_{x+3} - 2$ is $x=\{0,1\}$ for $x \in \mathbbm{Z}^+$, and in terms of phinary numbers, the solution to $x=\phi^{x+1} \rightharpoondown \phi$ is also $x=\{0,1\}$, for $x \in \mathbbm{Z}^+_\Phi$.

These are the same values found in Theorem \ref{theorem: natural cardinality 0,1} for the cardinality of the natural numbers, and were to be expected, as both the natural and phinary numbers are ordinals and therefore enjoy an order isomorphism, which implies an equal cardinality \cite[pp. 38-39]{ciesielski1997set}.
\end{proof}

\begin{remark}
It should be emphasized that despite using the very unique geometry of the GD, we have nevertheless obtained the same result as Theorem ~ \ref{theorem: natural cardinality 0,1}, which was found through the use of the ND---thus, confirming a logically consistent definition for the cardinality $\aleph_0$. Again, it should be emphasized that no where did we require that $\aleph_0$ be greater than all finite values, for had we done so, $\aleph_0$ would be undefined. As previously stated, if accepted, this result invalidates the concept of ``infinity" as some number greater than all finite numbers and points instead to something more subtle.
\end{remark}

\subsection{Interpretation}

From the above results, we are presented with an unexpected proposition. In the attempt to define the ordinal $\omega$ with a transfinite cardinality, one is forced into a logical quandary. If one requires that $\omega$ be larger than all finite values, here it has been shown that its cardinality is undefined, as both $\aleph_0=2^{\aleph_0}-1$ and $\aleph_0=F_{\aleph_0+3}-2$ prohibit such a solution. However, if that condition is relaxed, two answers result for the cardinality of $\omega$: zero and one.\footnote{If this result holds, it may have potential for describing states of superposition, in which quantities are known to simultaneously assume more than one value, such as in properties inherent to quantum mechanics.}

\section{Concluding Remarks} \label{section: concluding remarks}

In terms of the phinary numbers, we have seen that an alternate set of ordinals can be applied to functions that are not conveniently defined via the natural numbers. The Fibonacci diatomic sequence has not admitted a recurrence relation by traditional methods, but by indexing the sequence with phinary values, a concise relation was obtained, which shares symmetries with the hyperbinary sequence recurrence relation. Central to the applicability of the set of phinary numbers in this case was its parity triplet. The even phinary recurrence tree, which shared symmetries with the Stern-Brocot tree, was again defined via the parity of the phinary numbers when utilized as the ordinal set to index the Fibonacci diatomic sequence. Applications to the study of fractals may very well benefit from alternative ordinal systems, as some of the examples in this paper have demonstrated. Furthermore, the studying of alternative ordinal sets that are not as homogeneous as the natural numbers offers great potential for illuminating structures associated with ordinals, such as prime numbers---an avenue of study that has only been hinted at here. 

The natural and golden diamonds\footnote{The existence of other ``diamonds" has been found from tentative research, resulting in similar applicability to their respective ordinal sets.} each serve as a tool for analyzing a specific set of ordinals. A relationship between the ``infinite" and ``infinitesimal" is made geometric and manageable, allowing for deep insights into the nature of transfinite constructions.

Many questions arise. Stern's diatomic sequence/hyperbinary sequence is particularly interesting, as it provides a bijection from the natural numbers to the rationals. It is likely that the Fibonacci diatomic sequence can lead to similar such applications on the phinary numbers. Further research on the nature of the phinary recursion trees and their relationship to a phinary rational number system is to be pursued. Should a system of phinary rationals be rigorously defined, a Dedekind cut approach to constructing the set of real numbers could also be possible. However, would the topology of the reals as defined by the phinary rationals be homeomorphic to the Euclidean topology? With regard to the Fibonacci diatomic sequence, the number of primary occurrences for a given value could likely lead to a phinary analogue of Euler's totient function, should the latter be proven to enumerate the primary occurrences of the hyperbinary sequence, as it has been conjectured to do so; what information will this reveal about the nature of prime numbers within an ordinal set? And do the revelations on ordinal set cardinality disprove the axiom of infinity? If so, what is the logic behind this subtler notion of cardinal multiplicity?

\newpage

\centering
\textit{``That which is Below corresponds to that which is Above, and that which is Above corresponds to that which is Below, to accomplish the miracle of the One Thing."}\\
\centering
\begin{tabular}{l l l l l l l l l l l r l}
& \ \ \ \ \ \  &\ \ \ \ \ \ & \ \ \ \ \ \ &\ \ \ \ \ \ &\ \ \ \ \ \ &\ \ \ \ \ \ &\ \ \ \ \ \ &\ \ \ \ \ \ & &\\
& \ \ \ \ \ \  &\ \ \ \ \ \ & \ \ \ \ \ \ &\ \ \ \ \ \ &\ \ \ \ \ \ &\ \ \ \ \ \ &\ \ \ \ \ \ &\ \ \ \ \ \ &\ \ \ \ \ \ & &\multicolumn{1}{l}{-Hermes}	\\
& \ \ \ \ \ \  &\ \ \ \ \ \ & \ \ \ \ \ \ &\ \ \ \ \ \ &\ \ \ \ \ \ &\ \ \ \ \ \ &\ \ \ \ \ \ &\ \ \ \ \ \ & &\\
& \ \ \ \ \ \  &\ \ \ \ \ \ & \ \ \ \ \ \ &\ \ \ \ \ \ &\ \ \ \ \ \ &\ \ \ \ \ \ &\ \ \ \ \ \ &\ \ \ \ \ \ & &\\
 &&&&&&&&&&				&\footnotesize{\textit{The Emerald Tablet of Hermes Trismegistus}}\\
 &&&&&&&&&&  			& \scriptsize{translated by Dennis W. Hauck} \cite{hauck1999emerald}
\end{tabular}

\subsection*{Acknowledgements}

Much appreciated were the discussions with Kevin Tezlaf during his visit to Europe, and the on-going chats with the ever-patient and keenly curious Ava B. Filz---their support was thoroughly motivating.

%\bibliographystyle{IEEEtran}
%\bibliography{S_Tezlaf_-_On_ordinal_dynamics_and_the_multiplicity_of_transfinite_cardinality}

%\addbibresource{S_Tezlaf_-_On_ordinal_dynamics_and_the_multiplicity_of_transfinite_cardinality.bib}

%\printbibliography

%\iffalse

%\fi

\end{document}